\documentclass[12pt]{amsart}
\usepackage{amsmath, amsthm, amssymb}
\usepackage{fullpage}
\usepackage{color}
\usepackage{hyperref}
\usepackage{soul}
\usepackage{enumitem}
\usepackage[makeroom]{cancel}
\usepackage{mathtools}

\mathtoolsset{showonlyrefs}

\newtheorem{theorem}{Theorem}
\newtheorem{lemma}{Lemma}
\newtheorem{proposition}[lemma]{Proposition}
\newtheorem{corollary}[lemma]{Corollary}
\newtheorem{definition}[lemma]{Definition}
\newtheorem{remark}[lemma]{Remark}
\newtheorem{conjecture}{Conjecture}
\numberwithin{lemma}{section}

\numberwithin{equation}{section}

\newcommand{\R}{{\mathbb R}}
\newcommand{\C}{{\mathbb C}}
\newcommand{\N}{{\mathbb N}}

\renewcommand{\R}{\mathbb R}

\newcommand{\bM}{\mathbf M}
\newcommand{\bP}{\mathbf P}
\newcommand{\bE}{\mathbf E}

\newcommand{\bQ}{\mathbf Q}

\newcommand{\bI}{\mathbf I}
\newcommand{\bJ}{\mathbf J}
\newcommand{\bK}{\mathbf K}

\newcommand{\bu}{{\bar u}}
\newcommand{\bv}{{\bar v}}

\newcommand{\step}{\mathfrak{c}}

\newcommand{\la}{\langle}
\newcommand{\ra}{\rangle}

\newcommand{\ol}{\overline}
\newcommand{\ms}{M^\sharp}
\newcommand{\ps}{P^\sharp}

\newcommand{\calR}{\mathcal{R}}

\begin{document}

\title{Global solutions for 1D cubic dispersive equations, Part III: the quasilinear Schr\"odinger flow}

\author{Mihaela Ifrim}
\address{Department of Mathematics, University of Wisconsin, Madison}
\email{ifrim@wisc.edu}

\author{ Daniel Tataru}
\address{Department of Mathematics, University of California at Berkeley}
\email{tataru@math.berkeley.edu}

\begin{abstract}
The first target of this article is  the 
local well-posedness question for 1D quasilinear
Schr\"odinger equations with cubic nonlinearities.
The study of this class of problems, in all dimensions, was initiated 
in pioneering work of Kenig-Ponce-Vega for localized initial data, and then continued by 
Marzuola-Metcalfe-Tataru for initial data in Sobolev spaces. Our objective here is to fully redevelop the study of this problem in the 1D case,
and to prove a \emph{sharp local well-posedness}
result.

The second goal of this article is to consider 
the long-time/global existence of solutions for the same problem. This is motivated by a broad 
conjecture formulated by the authors in earlier work, which reads as follows: ``\emph{Cubic defocusing dispersive one dimensional flows with small initial data have global dispersive solutions}''; the conjecture was initially proved for  a class of semilinear  Schr\"odinger type models.

Our work here establishes the above conjecture 
for 1D quasilinear Schr\"{o}dinger flows.
Precisely, we show that if the problem has 
\emph{phase rotation symmetry} and is \emph{conservative and defocusing}, then 
small data in Sobolev spaces yields global,
scattering solutions. This is the first result of this type for 1D quasilinear dispersive flows where no localization condition is imposed on the data.
Furthermore, we prove the global well-posedness at the minimal Sobolev regularity as in our local well-posedness result.

The defocusing condition is essential in our global result. Without it, the authors 
have conjectured that \emph{small, $\epsilon$ size
data yields long-time solutions on the $\epsilon^{-8}$ time-scale}. A third goal of this paper is to also prove this second conjecture 
for 1D quasilinear Schr\"{o}dinger flows, also at minimal regularity.

\end{abstract}


\maketitle

\setcounter{tocdepth}{1}
\tableofcontents

%%%%%%%%%%%%%%%%%%%%%%%%%%%%%%%%%%%%%%%%%%%%%%%%%%%%%%%%%%%%%%%%%%%%%%%%%%%%%%%%%%%%%%%%%%%%%%%%%%%%%%%%%%%%%%%%%%%%%%%%%
%%%%%%%%%%%%%%%%%%%%%%%%%%%%%%%%%%%%%%%%%%%%%%%%%%%%%%%%%%%%%%%%%%%%%%%%%%%%%%%%%%%%%%%%%%%%%%%%%%%%%%%%%%%%%%%%%%%%%%%%%

\section{Introduction}

This article is devoted to the study of quasilinear Schr\"odinger flows in one space 
dimension. Our objective is to approach 
both the local well-posedness question and the 
long-time/global well-posedness question
for small initial data in Sobolev spaces, and to 
establish sharp results for both problems.

Irrespective of the dimension, local well-posedness
for quasilinear Schr\"odinger flows is a complex question, which is why the theory has historically been lagging behind the corresponding theory for nonlinear wave equations. 
The study of this class of problems was initiated in breakthrough work of Kenig-Ponce-Vega~\cite{KPV}, who considered initial data which is both regular and localized; see also \cite{KPRV1} and \cite{KPRV2}. The question of local well-posedness in translation invariant Sobolev spaces, and also  at lower regularity, was then considered by Marzuola-Metcalfe-Tataru~\cite{MMT1,MMT2}. 
As it turns out, one needs to differentiate between quadratic and cubic nonlinearities, and the classical $H^s$ Sobolev spaces can  only be employed in to  the cubic case. There
is also a substantial difference between the small and large data problems, and here we refer the reader to the work in \cite{MMT2} where a more detailed discussion is available.

\medskip

Our interest here is in 1D quasilinear Schr\"odinger flows with cubic nonlinearities,
and with initial data which is small in $H^s$. 
Indeed, the first objective of this paper is to fully redevelop the study of the local well-posedness problem in the 1D case at lower regularity having only a smallness assumption on our Sobolev initial data.  In this context, we are able to
prove a \emph{sharp local well-posedness} result.

\medskip

The second goal of this article is to consider 
the long-time/global existence of solutions for the same problem. This is motivated by a broad 
conjecture formulated by the authors in the recent work  \cite{IT-global}. Our global well-posedness (GWP) conjecture, which applies to both semilinear and quasilinear 1D problems, is as follows:

\begin{conjecture}[Non-localized data defocusing GWP conjecture]\label{c:nld}
One dimensional dispersive flows on the real line with cubic defocusing nonlinearities and small initial data have global in time, scattering solutions.
\end{conjecture}

Critically, compared with any earlier work in this direction, this conjecture requires  neither localization nor higher regularity for the initial data.
The main result of \cite{IT-global} asserts that this conjecture is true in the semilinear setting under suitable assumptions, most notably that the dispersion relation is the Schr\"odinger dispersion relation.  That was the first global in time well-posedness result of this type.  Scattering  here is interpreted in a weak sense, to mean that the solution satisfies global $L^6_{t,x}$ Strichartz estimates and bilinear $L^2_{t,x}$ bounds. This is due to the strong nonlinear effects, which preclude any classical scattering  for this class of problems.

Our work here establishes the above conjecture  for 1D quasilinear Schr\"{o}dinger flows, and represents the first validation of the conjecture
in a quasilinear setting. Precisely, we show that if the problem has \emph{phase rotation symmetry} and is \emph{conservative and defocusing}, then  small data in Sobolev spaces yields global, scattering solutions.  This constitutes the first result of this kind for any 1D quasilinear dispersive flow. Furthermore, we prove it at the same minimal Sobolev regularity as in our local well-posedness result.

\medskip

The defocusing condition for the nonlinearity is essential for our global result. In the focusing case, the existence of small solitons generally prevents global, scattering solutions. Nonetheless, in another recent paper, \cite{IT-focusing}, the authors conjectured that long-time solutions can be  obtained on  an optimal\footnote{The optimality should be interpreted as valid for generic problems within the class of evolutions for which the conjecture applies.}  time-scale: 

\begin{conjecture}[Non-localized data long-time well-posedness conjecture]\label{c:nld-focusing}
One-dimensional \\ \mbox{dispersive} flows on the real line with cubic conservative nonlinearities, and  initial data of size $\epsilon \ll 1$, have long-time solutions on the $\epsilon^{-8}$ time-scale.
\end{conjecture}
The conservative assumption here is very natural, and heuristically aims to prevent nonlinear ode blow-up of solutions with wave packet localization. \\

 A third goal of this paper is to also prove this second conjecture for 1D quasilinear Schr\"{o}dinger flows with cubic, conservative nonlinearities. This is also achieved  at minimal regularity, as in our  local well-posedness result below.

\subsection{ Quasilinear Schr\"odinger flows in 1D: local solutions}

The most general form for a  quasilinear Schr\"odinger flow in one space dimension is
\begin{equation}\tag{DQNLS}
\label{dqnls}
\left\{ \begin{array}{l}
i u_t + g(u,\partial_x u ) \partial_x^2 u = 
N(u,\partial_x u) , \quad u:
\R \times \R \to \C ,\\ \\
u(0,x) = u_0 (x),
\end{array} \right. 
\end{equation}
with a metric $g$ which is a real valued function and a source term $N$ which is a complex valued
smooth function of its arguments. Here smoothness 
is interpreted in the real sense. But if $g$ 
and $N$ are analytic, then they can also be thought of as separately (complex) analytic functions of $u$ and $\bu$. This is the interpretation that is employed in the present paper.

In parallel we consider the simpler problem
\begin{equation}\tag{QNLS}
\label{qnls}
\left\{ \begin{array}{l}
i u_t + g(u) \partial_x^2 u = 
N(u,\partial_x u) , \quad u:
\R \times \R \to \C, \\ \\
u(0,x) = u_0 (x),
\end{array} \right. 
\end{equation}
where $N$ is  at most quadratic in $\partial_x u$. This can be seen as the differentiated form of \eqref{dqnls}. We will state our results for both
of these flows, but, in order to keep the exposition focused on the important issues, we will provide complete arguments only for \eqref{qnls}.
\medskip

We first address the question of local well-posedness within the context of Sobolev spaces.
This is not at all a straightforward question,
and, in particular, it is significantly more challenging  compared to the corresponding question for nonlinear wave equations. One primary reason accounting for this added difficulty is the infinite speed of propagation, which may introduce 
short-time growth effects that  are only visible 
over long time-scales for hyperbolic problems. 
One of the challenges associated with this difficulty has to do with a potential loss of derivatives in the estimates, which occurs even in  linear problems such as
\[
i u_t + \Delta u = b^j(x) \partial u_j,
\]
and was already noted in early work of Mizohata~\cite{Mizohata}. 
What Mizohata discovered was a necessary and sufficient condition for linear $L^2$ well-posedness expressed in terms of integrals
of $\Re b$ over straight lines\footnote{ Here $\Re b$ denotes the real part of  the complex valued coefficient  $b$.}. One effect of this is that in problems with quadratic derivative nonlinearities $N$, one in general does not have local well-posedness in any $H^s$ space,
at least not without making additional decay assumptions at infinity for the initial data.

The additional discovery of local smoothing effect, see e.g. Sj\"{o}lin~\cite{Sj}, Constantin-Saut~\cite{CS} and Vega~\cite{Ve} shed some further light on this, as did the subsequent work of  Doi~\cite{Doi}. This led to extensive work on derivative nonlinear Schr\"odinger equations, beginning with Kenig-Ponce-Vega~\cite{KPV-s}, Hayashi and Ozawa~\cite{HO}, Chihara~\cite{Chi}, etc.; a complete discussion is somewhat outside the scope of this paper.

 Turning our attention to quasilinear problems, the first local well-posedness results for these problems go a long way back,  to work of Kenig-Ponce-Vega~\cite{KPV}, see also \cite{KPRV1,KPRV2}, where these equations were studied for  regular and localized initial data. The next step was to study the local well-posedness question in translation-invariant Sobolev spaces
and at lower regularity. After some initial 1D results at high regularity  due to de Bouard-Hayashi-Naumkin-Saut~\cite{BNPS}, Colin~\cite{Co}, Poppenberg~\cite{Pop} and Lin-Ponce~\cite{Lin-Ponce}, this was accomplished in full generality by Marzuola-Metcalfe-Tataru in \cite{MMT1,MMT2,MMT3}.  One key distinction made along the way was between quadratic and cubic nonlinearities. It is only in the latter case that the local well-posedness  is proved in classical $H^s$ Sobolev spaces, while in the former case some additional $\ell^1$ structure with respect to spatial cubes is needed. Another important distinction is between small and large data; in the latter case, large energy growth may occur on arbitrarily short-time scales, and a nontrapping condition is also required. However, the nontrapping condition is trivially satisfied in one space dimension and thus the large data local well-posedness result becomes more accessible.

In order to work in $H^s$ Sobolev spaces, in the present article we restrict our attention to the cubic case. For clarity we we introduce the following definition:

\begin{definition}
 We say that the equation \eqref{qnls}/\eqref{dqnls} is \emph{cubic} if $g$ is at least quadratic and $N$ is at least cubic at zero. 
\end{definition}
 
 For reference and added context, we begin by stating the earlier results for such problems, specialized to cubic nonlinearities in 1D:

\begin{theorem}[cubic nonlinearities \cite{MMT2,MMT3}]\label{t:regular}
 The cubic problem \eqref{qnls} is locally well-posed 
in $H^s$ for $s > 2$, and the cubic problem \eqref{dqnls} is locally well-posed 
in $H^s$ for $s > 3$.
\end{theorem}

This result applies to both small and large data,
with a lifespan that depends on the initial data size. In higher dimensions one also needs a nontrapping assumption on the initial data. This is automatically satisfied in one dimension.

The initial target of this work was to examine global well-posedness; however, it quickly became evident that  our methods also yield a drastic improvement in the local well-posedness theory, to the point where we are able to obtain a sharp result.

\begin{theorem}\label{t:local}
The cubic problem \eqref{qnls} is locally well-posed for small data in $H^s$ for $s > 1$, and the cubic problem \eqref{dqnls} is locally well-posed for small data in $H^s$ for $s > 2$.
\end{theorem}

Here well-posedness is interpreted in the Hadamard sense, and includes:
\begin{itemize}
    \item existence of solutions in  $C([0,T]; H^s)$.
    \item uniqueness of regular solutions.
    \item uniqueness of rough solutions as 
    uniform limits of regular solutions.
    \item continuous dependence on the initial data.
    \item higher regularity: more regular data yields more regular solutions.
\end{itemize}
We also show a weak-Lipschitz dependence of the solutions on the initial data, at the $L^2$ level.

In addition, as part of the proof of the theorem we show that the solutions satisfy 
a range of linear and bilinear dispersive estimates, which 
includes:
\begin{itemize}
    \item sharp $L^6_{t,x}$ Strichartz estimates.
    \item sharp bilinear $L^2_{t,x}$ bounds.
\end{itemize}
These are discussed later in Section~\ref{s:boot}, where an
expanded form of the above result is provided.

To better understand the significance of this result, one  can compare our range of Sobolev
exponents with the leading order scaling for these problems, which corresponds to the exponents $s_c = \frac32$ 
for \eqref{dqnls}, respectively $s_c = \frac12$ for \eqref{qnls}.
This might lead one to believe that there might be room for further 
improvement. 

However, there is also a semilinear comparison to make; precisely,
both results in Theorem~\ref{t:local} are on the same scale as the 1D cubic NLS problem.
This is well-posed in $L^2$, but generally not below this threshold, unless one considers 
the exact integrable model. Thus, in all likelihood our result is sharp ! Precisely, we conjecture:
%, but we strongly believe this is the case, and plan to consider this in subsequent work: \red{?}

\begin{conjecture}
The range for $s$ in Theorem~\ref{t:local} is sharp, in the sense that generically the problem \eqref{qnls} is ill-posed in $H^s$ for $s < 1$, and \eqref{dqnls} is ill-posed in $H^s$ for $s < 2$.  \end{conjecture}

The endpoint case $s = 1$ for \eqref{qnls} and $s=2$ for \eqref{dqnls} remains open for now.  

To further motivate this threshold, we note that heuristically, the worst case scenario is represented by the self-interaction of a single wave packet, which becomes critical exactly at $s = 1$; some further discussion of this threshold is provided by the authors in the overview article \cite{IT-expository}.  From this perspective, one may view our local well-posedness result as the direct Schr\"odinger counterpart of the nonlinear wave  result of Smith-Tataru~\cite{ST}, 
where the worst case scenario is also 
based on a single wave packet self-interaction.

\bigskip

The starting point for the proof of our result is provided by the earlier results stated  in 
Theorem~\ref{t:regular}. The lifespan of these solutions, say for \eqref{qnls}, a-priori depends on the $H^{2+}$ size of the initial data, which is not necessarily small.
However, once we prove suitable a-priori bounds
on these solutions, we can extend the lifespan so that it depends only on the $H^s$ size of the data.  Once we have regular solutions on uniform time scales, we produce the rough, $H^s$ solutions as 
unique uniform limits of regular solutions. 
This requires careful estimates for the linearized 
flow, which we do at the $L^2$ level.

To keep the ideas clear and to avoid excessive technicalities we consider only the small data case
in this paper.  However, we believe that the results also carry over to the large data case:

\begin{conjecture}
The result in Theorem~\ref{t:local} also holds in the large data case.    
\end{conjecture}

\subsection{ Quasilinear Schr\"odinger flows in 1D: long-time and global solutions}

As mentioned earlier, this part of the paper, as well as our entire approach, is motivated 
by the authors' earlier work in \cite{IT-global}
and \cite{IT-focusing}, where Conjecture~\ref{c:nld} and 
Conjecture~\ref{c:nld-focusing} are put forward.
These conjectures are proved in \cite{IT-global}
and \cite{IT-focusing} for a semilinear Schr\"odinger model with a cubic nonlinearity.
Our goal here is to prove these conjectures in the
quasilinear Schr\"odinger setting; an extra bonus is that we can do this at minimal regularity, exactly as the earlier results obtained by the authors in the semilinear setting. 

By contrast to the above general local well-posedness result, refined long-time and  global in time results require some natural structural assumptions on the equations, which will be described next. We begin with a sequence of definitions that we will use throughout our exposition.

\begin{definition}
We say that the equation \eqref{qnls}/\eqref{dqnls}
has  \emph{phase rotation} symmetry if it is 
invariant with respect to the transformation 
$u \to u e^{i\theta}$ for $\theta \in \R$.
\end{definition}

One consequence of this assumption is that 
the Taylor series of the nonlinearity at zero only has 
odd terms, which are multilinear expressions 
in $u$ and $\bar u$ (and their derivatives) with one more $u$ factor. It suffices in effect to assume that this holds for the cubic terms in the equation.

 Even though  our problem is quasilinear, in order to state the next assumptions it is convenient to write it in a semilinear way. Assuming the phase rotation symmetry, and also harmlessly setting  $g(0) = 1$, this semilinear form is
\begin{equation}\label{eq-cubic}
i u_t + \partial^2_x u = C(u,\bar u, u) + \text{higher order terms},
\end{equation}
where $C$ is a translation invariant trilinear form with symbol $c(\xi_1,\xi_2,\xi_3)$. We refer
the reader to Section~\ref{s:notations} for 
a description of our notations.

\begin{definition}
We say that the equation \eqref{qnls}/\eqref{dqnls}
is \emph{conservative} if  the cubic component of the nonlinearity satisfies 
\[
c(\xi,\xi,\xi), \, \partial_{\xi_j}
c(\xi,\xi,\xi) \in \R .
\]
\end{definition}
This condition represents a slight generalization of the corresponding condition in \cite{IT-global}
(a similar generalization applies there). It is automatic for the contributions from $g$,
so it is solely a condition on $N$.  It is also immediately satisfied if the problem admits a coercive conservation law, but, conversely, it does not imply the existence of such a conservation law. Heuristically, one may think of this condition as providing an approximate conservation law for wave packet solutions.

This is all we need for the enhanced long-time  results. However, for the global in time results 
it is also necessary to differentiate between focusing and defocusing problems.

\begin{definition}
We say that the equation \eqref{qnls}/\eqref{dqnls}
is \emph{defocusing} if   $c(\xi,\xi,\xi)$ 
    is positive definite. This means that
    \begin{equation}\label{defocusing}
    c(\xi,\xi,\xi) \gtrsim \la \xi \ra^{2+2k},
    \end{equation}
    where $k = 1$ for \eqref{dqnls}
    respectively $k=0$ for \eqref{qnls}.
\end{definition}

As a model \eqref{qnls} type problem  one may consider for instance the equation
\begin{equation}\label{qnls=model}
\begin{aligned}
i u_t + (1 + a|u|^2) \partial_x^2 u = b u |\partial_x u|^2+ c u|u|^2,  \qquad u(0) = u_0, 
\end{aligned}
\end{equation}
with real $a$, $b$, $c$ where the sign choice $a-b,c > 0$ yields a defocusing
problem and either $a < b$ or $c < 0$ yields a focusing one. 

\bigskip

Now we can state our results, beginning with the long-time result for the general case:

\begin{theorem}\label{t:long}
 a) Assume that the equation \eqref{qnls} has phase rotation symmetry and is conservative. Given $s > 1$, assume  that the initial data $u_0$ is small in $H^s$,
 \begin{equation}
 \|u_0\|_{H^s} \leq \epsilon \ll 1.   
 \end{equation}
Then  the solutions have a lifespan which is at least $O(\epsilon^{-8})$, and satisfy
\begin{equation}
\| u\|_{L^\infty_t(0,T; H^s_x)} \lesssim \epsilon, \qquad T \ll \epsilon^{-8}.
\end{equation}

b) Assume that the equation \eqref{dqnls} has phase rotation symmetry and is conservative. Given $s >2$, 
assume that the initial data $u_0$ is small in $H^s$,
 \begin{equation}
 \|u_0\|_{H^s} \leq \epsilon \ll 1.   
 \end{equation}
 Then  the solutions have a lifespan which is at least $O(\epsilon^{-8})$, and satisfy
\begin{equation}
\| u\|_{L^\infty_t(0,T; H^s_x)} \lesssim \epsilon, \qquad T \ll \epsilon^{-8}.
\end{equation}

\end{theorem}

For our final main result we restrict
ourselves to the defocusing problem, where we prove a global in time result:

\begin{theorem}\label{t:global}
 a) Assume that the equation \eqref{qnls} has phase rotation symmetry and is conservative and defocusing.
  Given $s > 1$, assume  that
  the initial data $u_0$ is small in $H^s$,
 \begin{equation}
 \|u_0\|_{H^s} \leq \epsilon \ll 1.   
 \end{equation}
Then the solutions are global in time, and satisfy 
\begin{equation}
\| u\|_{L^\infty_t H^s_x} \lesssim \epsilon.
\end{equation}

b) Assume that the equation \eqref{dqnls} has phase rotation symmetry and is conservative and defocusing.
 Given $s > 2$, assume  that
the initial data $u_0$ is small in $H^s$,
 \begin{equation}
 \|u_0\|_{H^s} \leq \epsilon \ll 1.   
 \end{equation}
Then the solutions are global in time, and satisfy 
\begin{equation}
\| u\|_{L^\infty_t H^s_x} \lesssim \epsilon.
\end{equation}
\end{theorem}

As previously emphasized, this is the first quasilinear global well-posedness result of its type, and in particular the first global well-posedness result for a one-dimensional quasilinear Schr\"odinger equation where no localization on the initial data is being assumed.
 The only earlier result establishing a similar result in 1D is the one in \cite{IT-global}, where a semilinear model was considered.  

We note that the above result is stated here in a short form, but 
the proof shows in effect that the solutions satisfy good dispersive bounds. The reader is referred to Theorem~\ref{t:global-fe}
for a more complete, frequency envelope based description of 
the result. For now, we will simply point out two consequences 
of Theorem~\ref{t:global-fe}.
These are stated below in the \eqref{qnls} case, but they also apply in the \eqref{dqnls}
case at one level of regularity higher:

\begin{enumerate}[label=(\roman*)]
    \item $L^6_{t,x}$ Strichartz bounds:
  \begin{equation}\label{sample1}
   \| \la D\ra^{\frac56+} u\|_{L^6_{t,x}} \lesssim \epsilon^\frac23 .  
  \end{equation}  

\item{Bilinear $L^2_{t,x}$ bounds}, which can be stated in a balanced form
\begin{equation}\label{sample2}
   \| \partial |\la D\ra^{\frac34+} u|^2\|_{L^2_{t,x}} \lesssim \epsilon^2 ,
  \end{equation}  
and in an imbalanced form\footnote{ Here $T_u$ is the standard paraproduct operator capturing the low-high frequency interactions, see e.g. \cite{Metivier}.}
\begin{equation}\label{sample3}
   \| T_u \bu\|_{L^2_tH^{\frac32+}_x} \lesssim \epsilon^2.     \end{equation} 
These bounds hold also when $u$ and $\bu$ are frequency localized in a dyadic fashion.
\end{enumerate}

The exponents in \eqref{sample1}-\eqref{sample3} above correspond to the limiting case $s \to 1$.
One should compare the Strichartz bound \eqref{sample1}, which has a loss of up to $1/6$ derivative, 
with the sharp $L^6_{t,x}$ bounds we prove for the local well-posedness result. The difference is of course that these bounds here are global in time.

While these bounds hold globally in time for the defocusing problem in Theorem~\ref{t:global}, they also apply in the more general context of the long-time result in Theorem~\ref{t:long}, but only on the $\epsilon^{-6}$ timescale. There, they are critical in order to propagate the uniform energy bounds up to the $\epsilon^{-8}$ timescale.

\bigskip

One may also frame our results within the broader question of obtaining long-time solutions for one dimensional dispersive flows with quadratic/cubic nonlinearities, which has attracted a lot of attention in recent years. One can distinguish two different but closely related types of results that have emerged, as well as several successful approaches.

On one hand, \emph{normal form methods} have been developed 
in order to extend the lifespan of solutions, beginning with 
\cite{Shatah} in the early '80's. Around 2000, the \emph{ I-method}, introduced in \cite{I-method} brought forth the idea of constructing  better almost conserved quantities. These two ideas serve well in the study of semilinear flows, where it was later understood that they are connected, see for instance ~\cite{Bourgain-nf}.

Neither of these techniques can be directly applied to quasilinear problems. Addressing this problem, it was discovered in the work of the authors and collaborators \cite{BH},
\cite{IT-g} that one can adapt the normal form method 
to quasilinear problems by constructing energies which 
simultaneously capture both the quasilinear and the normal form structures. This idea was called the \emph{modified energy method}, and can also be seen in some way as a quasilinear adaptation of the I-method.  Other alternate approaches, also in the quasilinear setting, are provided by the \emph{ flow method } of  \cite{hi}, where a better normal form transformation is constructed using a well chosen auxiliary flow, and by the \emph{paradiagonalization method} of Alazard and-Delort~\cite{AD},
where a paradifferential symmetrization is applied instead.

Going further, the question of obtaining
scattering, global in time solutions for one dimensional dispersive flows with quadratic/cubic nonlinearities  has also been extensively studied in the last two decades  for a number of models, under the assumption that the initial data is both \emph{small} and \emph{localized};  without being exhaustive, see for instance \cite{HN,HN1,LS,KP,IT-NLS}, and the references within.
The nonlinearities in these models are primarily cubic, though the analysis has also been extended via normal form and modified energy methods to problems which also have nonresonant quadratic interactions; several such examples are \cite{AD,IT-g,D,IT-c,LLS}, see also further references therein, as well as the authors' expository paper on the modified scattering phenomena \cite{IT-packet}.

Comparing the above class of results, where the initial data
is both small and localized, with the present results, without any localization assumption, it is clear that in the latter case the problem becomes much more difficult, because the absence of localization allows for far stronger nonlinear interactions over long time-scales. Another twist is that it is now necessary  to distinguish between the focusing and defocusing case, as we do  here and also in our earlier semilinear work in \cite{IT-global} and \cite{IT-focusing}. 

\subsection{Outline of the paper} 
As it is the case with many quasilinear well-posedness results, the main requirement is to establish energy estimates, which are needed for 
\begin{enumerate}[label=(\alph*)]
\item the full equation, in a range of Sobolev spaces $H^s$.
\item the linearized equation, where it suffices to work in a single space, $L^2$ in our case. 
\end{enumerate}
At the heart of both of these problems we identify a common core, namely
\begin{enumerate}[resume*]
\item the linear 
paradifferential equation, where bounds are essentially frequency localized. 
\end{enumerate}
\medskip

The proof of the energy estimates is not a direct argument, but instead requires access to a full set of dispersive bounds, namely (see the discussion in Section~\ref{s:disp})
\begin{enumerate}[label=(\roman*)]
    \item Bilinear $L^2$ bounds, which measure the interaction of transversal waves.
    \item Strichartz estimates, which measure 
    the dispersion of a single wave.
\end{enumerate}
These two types of bounds, in turn, loop back to themselves, so they are proved within a complex bootstrap argument. 

\bigskip

To carry out the above plan, we have structured the paper into modular blocks as follows:

\medskip
\emph{ Littlewood-Paley theory, multilinear forms and paradifferential expansions.}
These are part of the standard toolbox for such problems, and are described in Section~\ref{s:notations}.
Bony's paradifferential calculus~\cite{Bony} (see also \cite{Metivier}) plays an essential role in the paper. Multilinear paradifferential expansions, which can be seen as more refined versions of the paradifferential calculus, were first introduced in \cite{T-WM}. We also discuss hereresonant and doubly resonant interactions.

\medskip

\emph{Frequency envelopes and the bootstrap argument.} Keeping track of a large family of linear and bilinear bounds is efficiently done using Tao's notion of \emph{frequency envelopes}, first introduced in \cite{Tao-WM}. Our implementation of this idea is presented in Section~\ref{s:boot}, where we use it to recast our full set of estimates and results in a bootstrap fashion. The idea of bootstrapping both linear and bilinear dispersive bounds was introduced in the authors's earlier work on the Benjamin-Ono flow~\cite{IT-BO}.

\medskip

\emph{A full set of equations.}
Starting from the full equations \eqref{qnls} or \eqref{dqnls} and the associated linearized flow, in Section~\ref{s:menagerie}, we introduce the associated linear paradifferential flow and use it to rewrite all the other equations in a paradifferential form. This suffices for the local well-posedness result. On the other hand  for the long-time/global results we need 
a more accurate expansion, where we also separate the cubic doubly resonant part, which carry out the bulk of the nonlinear interactions over long-time scales.

\medskip

\emph{Conservation laws in density-flux  form.}
Propagation of both mass and momentum bounds is important not only in itself, but also as a key element in the proof of the bilinear $L^2_{t,x}$ bounds. For this latter context, it is important to recast these conservation laws in density-flux form. This is done in Section~\ref{s:df}, in a frequency localized setting, broadly following the lead of the authors' previous paper~\cite{IT-global}. Another critical feature here is our use of quartic density/flux corrections, which  allow us to improve the propagation error bounds.

\medskip

\emph{Interaction Morawetz analysis and bilinear $L^2_{t,x}$ bounds for the linear paradifferential equation.} The main tool in the  proof of the bilinear $L^2_{t,x}$ bounds is provided by the \emph{interaction Morawetz estimates}, an idea introduced in the NLS context by the I-team~\cite{MR2415387}.  We develop this idea in Section~\ref{s:para}, following the set-up in authors' previous paper~\cite{IT-global}; but we also refer the reader to \cite{PV} and \cite{MR2527809} for earlier work on 1D NLS.

\medskip

\emph{Strichartz estimates for the linear paradifferential equation.} A standard approach in the study of quasilinear dispersive pde's is  to prove Strichartz estimates with a loss of derivatives, arising from sharp paradifferential Strichartz estimates on short, semiclassical time scales; this was first introduced in \cite{T:nlw1,T:nlw2,T:nlw3,BC,BC1} in the nonlinear wave context. For the Schr\"odinger context we refer the reader to \cite{GS-T}, and also to later work in \cite{BGT}. Proving lossless Strichartz estimates for variable  coefficient Schr\"odinger equations is a more difficult task, see \cite{HTW,RZ,T-St}, and requires not only $C^2$ coefficients but also  asymptotic flatness, neither of which are available here. In one space dimension one has an additional tool available, namely flattening the metric by a change of coordinates, a strategy which was implemented in \cite{BP} to handle the case of time independent BV coefficients.

Nevertheless, in Section~\ref{s:str} we are able  to prove \emph{lossless Strichartz estimates} locally in time. After flattening the metric, the key element in our analysis is to 
use our bilinear $L^2_{t,x}$ bounds in order separate a perturbative component. This leaves one with a better behaved flow, which can be recast to a setting where the wave packet parametrix of \cite{MMT-param} applies. This result may be seen as the counterpart in the \eqref{qnls} setting of the result of \cite{ST} for nonlinear wave equations, or of \cite{Ai} for water waves.

\medskip

\emph{Local well-posedness.}
Once we have the bilinear $L^2_{t,x}$ bounds and the Strichartz bounds for the paradifferential equation, the next step is to leverage these into similar bounds and short-time energy estimates for both the full and the linearized equation. This is carried out in Section~\ref{s:lwp},  and immediately leads to the  proof of the local well-posedness result, following the approach described in the expository paper \cite{IT-primer}.

\medskip

\emph{Long-time/global well-posedness.}
Given local well-posedness at the desired Sobolev regularity, in order to to prove the long-time result and the global well-posedness conjecture in the defocusing case we only need to establish long-time energy bounds. This is the goal of the last section, and is achieved within a complex bootstrap argument  involving Strichartz and bilinear $L^2_{t,x}$ estimates. On one hand, compared to the local well-posedness part, here we no longer need to work with the linearized equation, and it suffices to devote all our attention to the full equation. On the other hand, we now need  to pay close attention to the doubly resonant interactions, which are no longer perturbative, and whose structure is crucial in our analysis. This is where our \emph{conservative} and \emph{defocusing} assumptions play the leading role.

\subsection{Acknowledgements}
The first author was supported by the Sloan Foundation, and by an NSF CAREER grant DMS-1845037. The second author was supported by the NSF grant DMS-2054975 as well as by a Simons Investigator grant from the Simons Foundation.

\section{Notations and preliminaries}
\label{s:notations}

 Here we review a number of standard tools and notations, beginning with Littlewood-Paley decompositions.

\subsection{ The Littlewood-Paley decomposition}
Rather than using a standard Littlewood-Paley decomposition in frequency using powers of $2$, we choose 
a smaller ``dyadic'' step $\step > 1$, but which will be taken very close to $1$, see Remark~\ref{r:nonres} below.

Let $\psi$ be a smooth bump function which is supported in $[-\step,\step]$ and equal to $1$ on
$[-\step^{-1},\step^{-1}]$.  We define the Littlewood-Paley operators $P_{k}$ and
$P_{\leq k} = P_{<k+1}$ for $k \geq 0$ by 
\[
\widehat{P_{\leq k} f}(\xi) := \psi(\xi/\step^k) \hat f(\xi)
\]
for all $k \geq 0$, and $P_k := P_{\leq k} - P_{\leq k-1}$ (with the
convention $P_{\leq -1} = 0$),
whose symbol is 
\begin{equation}\label{def-phi}
\phi_k(\xi) :=\phi(\xi/\step^k), \qquad
\phi(\xi):= \psi(\xi) - \psi(\xi/\step).
\end{equation}

We also define $P_{>k} := P_{\geq k+1}:= 1 - P_{\leq k}$. Note that all the operators $P_k$,
$P_{\leq k}$, $P_{\geq k}$ are bounded on all translation-invariant Banach spaces, due to to Minkowski's inequality.  

Thus, at the multiplier level, the Littlewood-Paley decomposition reads
\begin{equation*}
1=\sum_{k\in \mathbf{N}}P_{k},
\end{equation*}
where the multipliers $P_k$ have smooth symbols localized at frequency $\step^k$. Correspondingly, our solution $u$ will be decomposed as
\[
u = \sum_{k \in \N} u_k, \qquad u_k := P_k u. 
\]
The main estimates we will establish for our solution $u$ to \eqref{qnls}/\eqref{dqnls} will be linear and bilinear estimates for the functions $u_k$.

On occasion it will be more efficient to switch from a discrete to a continuous Littlewood-Paley decomposition. There we think of $k$  as a real nonnegative parameter and define 
\[
P_k u := \frac{d}{dk} P_{< k} u.
\]
Then our Littlewood-Paley decomposition for $u$ takes an integral form
\[
u = u_0 +\int_0^\infty u_k \, dk.
\]

\begin{remark}\label{r:nonres}
The motivation for working with a 
Littlewood-Paley decomposition with a step $\step$ which is close to $1$ is to insure a clean classification 
of multilinear interactions as follows:
\begin{itemize}
\item there are no balanced quadratic interactions (i.e. from one dyadic region to itself).
\item in all trilinear $high \times high \times high \rightarrow low$ interactions,
at least two of the input frequencies are separated. 
\item there are no balanced quartic interactions (i.e. from one dyadic region to itself).
\end{itemize}
\end{remark} 

In terms of notations, for the most part we will avoid representing dyadic frequencies 
as powers of $\step$, and use instead Greek letters
such as $\lambda$ and $\mu$. One must however be careful and remember that all summations with respect to such parameters are dyadic.

\subsection{Strichartz and bilinear $L^2_{t,x}$ bounds} \label{s:disp}

Both of these will be the type of estimates  we establish later on for the solutions to both the full quasilinear problem and for its linearization. For reference and comparison purposes, here we briefly recall these bounds in the constant coefficient  case.

We begin with the classical  Strichartz inequality, which applies to solutions to the inhomogeneous linear Schr\"odinger equation:
\begin{equation}\label{bo-lin-inhom}
(i\partial_t + \partial^2_x)u = f, \qquad u(0) = u_0.
\end{equation}

We define the Strichartz space $S$ associated to the $L^2$ flow by
\[
S := L^\infty_t L^2_x \cap L^4_t L^\infty_x,
\]
as well as its dual
\[
S' = L^1_t L^2_x + L^{\frac43} _t L^1_x .
\]
We will also use the notation 
\[
S^{s} = \langle D \rangle^{-s} S
\]
to denote the similar spaces associated to the flow in $H^s$.

The Strichartz estimates in the $L^2$ setting are a consequence of the nonvanishing curvature of the characteristic set for the Schr\"odinger equation, i.e. the parabola $\{\tau + \xi^2=0\}$.
They are summarized in the following lemma:

\begin{lemma}
Assume that $u$ solves \eqref{bo-lin-inhom} in $[0,T] \times \R$. Then the following estimate holds.
\begin{equation}
\label{strichartz}
\| u\|_S \lesssim \|u_0 \|_{L^2} + \|f\|_{S'} .
\end{equation}
\end{lemma}
One intermediate norm between the two endpoints in $S$ is $L^6_{t,x}$, and its dual is $L^\frac65 _{t,x}$. In our estimates later in the paper, we will give preference to this Strichartz norm and neglect the rest of the family.

\bigskip

The second property of the linear Schr\"odinger equation we want to describe here is the 
bilinear $L^2$ estimate, which is as follows:

\begin{lemma}
\label{l:bi}
Let $u^1$, $u^2$ be two solutions to the inhomogeneous Schr\"odinger equation with data 
$u^1_0$, $u^2_0$ and inhomogeneous terms $f^1$ and $f^2$. Assume 
that the sets 
\[
E_i =  \text{supp } \hat u^i 
\]
are disjoint. Then we have 
\begin{equation}
\label{bi-di}
\| u^1 u^2\|_{L^2} \lesssim \frac{1}{\text{dist}(E_1,E_2)^\frac12} 
( \|u_0^1 \|_{L^2} + \|f^1\|_{S'}) ( \|u_0^2 \|_{L^2} + \|f^2\|_{S'}).
\end{equation}
\end{lemma}

One corollary of this applies in the case when we look at the product of two solutions which are supported in different dyadic regions:
\begin{corollary}\label{c:bi-jk}
Assume that $u^1$ and $u^2$ as above are supported in dyadic regions $\vert \xi\vert \approx \step^j$ and $\vert \xi\vert \approx \step^k$, $\vert j-k\vert >2$, then
\begin{equation}
\label{bi}
\| u^1 u^2\|_{L^2} \lesssim \step^{-\frac{\max\left\lbrace j,k \right\rbrace  }{2}}
( \|u_0^1 \|_{L^2} + \|f^1\|_{S'}) ( \|u_0^2 \|_{L^2} + \|f^2\|_{S'}).
\end{equation}
\end{corollary}

Another useful  case is when we look at the product of two solutions which are supported in the same  dyadic region, but with frequency separation:
\begin{corollary}\label{c:bi-kk}
Assume that $u^1$ and $u^2$ as above are supported in the dyadic region  $\vert \xi\vert \approx \step^k$, but have $O(\step^k)$ frequency separation between their supports. Then 
\begin{equation}
\label{bi-kk}
\| u^1 u^2\|_{L^2} \lesssim \step^{-\frac{k}{2}}
( \|u_0^1 \|_{L^2} + \|f^1\|_{S'}) ( \|u_0^2 \|_{L^2} + \|f^2\|_{S'}).
\end{equation}
\end{corollary}

\subsection{Resonant analysis} In this subsection we introduce the notion of \emph{resonant frequencies}, which applies to nonlinear interactions in the context of the linear dispersion relation associated to our equation. Given that our nonlinearity is cubic, we will discuss this notion in the cubic nonlinearity context. Furthermore, we will be interested in this analysis only in the case where our problem has the phase rotation symmetry. This implies that the cubic part $C$ of the  nonlinearity can be thought of as a trilinear form with arguments $C(u,\bu,u)$.

For such a trilinear form, given three input frequencies $\xi_1, \xi_2,\xi_3$ for our cubic nonlinearity, the output will be at frequency 
\[
\xi_4 = \xi_1-\xi_2+\xi_3.
\]
This relation can be described in a more symmetric fashion as 
\[
\Delta^4 \xi = 0, \quad \mbox{where } \quad  \Delta^4 \xi := \xi_1-\xi_2+\xi_3-\xi_4 .
\]
This is a resonant interaction if and only if we have a similar relation for the associated time frequencies, namely 
\[
\Delta^4 \xi^2 = 0, \quad \mbox{where } \quad \Delta^4 \xi^2 := \xi_1^2-\xi_2^2+\xi_3^2-\xi_4^2 .
\]
Hence, we define the resonant set in a symmetric fashion as  
\[
\calR := \{ \Delta^4 \xi = 0, \ \Delta^4 \xi^2 = 0\}.
\]
In one space dimension it is easy to see that this set can be equivalently described as 
\[
\calR = \{ \{\xi_1,\xi_3\} = \{\xi_2,\xi_4\}\}.
\]

When considering estimates for cubic resonant interactions, it is clear that the case when $\xi_1 \neq \xi_3$ is more favourable, as there we have access to bilinear $L^2_{t,x}$ bounds. 
The unfavourable case is when all four frequencies are equal. We denote this set by 
\[
\calR_2 := \{\xi_1 =\xi_3 = \xi_2=\xi_4\},
\]
and we will refer to it as \emph{the doubly resonant set}. Heuristically, this set carries the bulk of the cubic long range interactions in the nonlinear problem, and can be intuitively associated with wave packet self-interactions.

\subsection{Translations and translation-invariant multilinear forms}

Since our problem as stated is invariant with respect to translations, it is natural  that translation-invariant multilinear forms play an important role in the analysis. The simplest examples are multipliers, which are convolution operators,
\[
m(D) u (x) = \int K(y) u(x-y) \,dy, \quad \hat K = (2\pi)^{-\frac12} m.
\]
Denoting translations by 
\[
u^y(x) := u(x-y),
\]
we will more generally denote by $L$ any convolution operator 
\[
Lu (x) := \int K(y) u^y(x)\, dy,
\]
where $K$ is an integrable kernel, or  a bounded measure, with a universal bound.

Moving on to multilinear operators, we will similarly denote by $L$ any multilinear form 
\[
L(u_1,\cdots, u_k) (x) = 
\int K(y_1, \cdots,y_k) u_1^{y_1}(x) \cdots u_k^{y_k} \, dy,
\]
where, again $K$ is assumed to have a kernel which is integrable, or more generally,  a bounded measure (we allow for the latter in order to be able to include products here).
Multilinear forms satisfy the same bounds as  corresponding products do, as long as we work in translation-invariant Sobolev spaces.

\medskip

A special role in our analysis is played by multilinear forms generated in the analysis of problems with the phase rotation symmetry, where the arguments $u$ and $\bar u$ are alternating. For such forms we borrow the notations from our earlier paper \cite{IT-global}.

Precisely, for an integer $k \geq 2$, we will use translation-invariant $k$-linear  forms 
\[
(\mathcal D(\R))^{k} \ni (u_1, \cdots, u_{k}) \to     Q(u_1,\bu_2,\cdots) \in \mathcal D'(\R),
\]
where the nonconjugated and conjugated entries are alternating.

Such a form is uniquely described by its symbol $q(\xi_1,\xi_2, \cdots,\xi_{k})$
via
\[
\begin{aligned}
Q(u_1,\bu_2,\cdots)(x) = (2\pi)^{-k} & 
\int e^{i(x-x_1)\xi_1} e^{-i(x-x_2)\xi_2}
\cdots 
q(\xi_1,\cdots,\xi_{k})
\\ & \qquad 
u_1(x_1) \bu_2(x_2) \cdots  
dx_1 \cdots dx_{k}\, d\xi_1\cdots d\xi_k,
\end{aligned}
\]
or equivalently on the Fourier side
\[
\mathcal F Q(u_1,\bu_2,\cdots)(\xi)
= (2\pi)^{-\frac{k-1}2} \int_{D}
q(\xi_1,\cdots,\xi_{k})
\hat u_1(\xi_1) \bar{\hat u}_2(\xi_2) \cdots  \,
d\xi_1 \cdots d\xi_{k-1},
\]
where, with alternating signs, 
\[
D := \{ \xi = \xi_1-\xi_2 + \cdots \}.
\]

They can also be described via their kernel
\[
Q(u_1,\bu_2,\cdots)(x) =  
\int K(x-x_1,\cdots,x-x_{k})
u_1(x_1) \bu_2(x_2) \cdots  
dx_1 \cdots dx_{k},
\]
where $K$ is defined in terms of the  
Fourier transform  of $q$
\[
K(x_1,x_2,\cdots,x_{k}) = 
(2\pi)^{-\frac{k}2} \hat q(-x_1,x_2,\cdots,(-1)^k x_{k}).
\]

These notations are 
convenient but slightly nonstandard because of the alternation of complex conjugates. Another important remark is that, for $k$-linear forms, the cases of odd $k$, respectively even $k$ play different roles here, as follows:

\medskip

i) The $2k+1$ multilinear forms will be thought of as functions, e.g. those which appear 
in some of our evolution equations.

\medskip

ii) The $2k$ multilinear forms will be thought of as densities, e.g. which appear 
in some of our density-flux pairs.

\medskip
Correspondingly,   to each $2k$-linear form $Q$ we will associate a $2k$-linear functional $\bQ$ defined by 
\[
\bQ(u_1,\cdots,u_{2k}) := \int_\R Q(u_1,\cdots,\bu_{2k})(x)\, dx,
\]
which takes real or complex values. This may be alternatively expressed on the Fourier side as 
\[
\bQ(u_1,\cdots,u_{2k}) = (2\pi)^{1-k} \int_{D}
q(\xi_1,\cdots,\xi_{2k})
\hat u_1(\xi_1) \bar{\hat u}_2(\xi_2) \cdots  
\bar{\hat u}_{2k}(\xi_{2k})\,d\xi_1 \cdots d\xi_{2k-1},
\]
where, with alternating signs, the diagonal $D_0$ is given by
\[
D_0 = \{ 0 = \xi_1-\xi_2 + \cdots \}.
\]
Note that in order to define the multilinear functional $\bQ$ we only need to know the symbol $q$ on $D_0$.

\subsection{Multilinear paradifferential expansions} 

As written, in the equations \eqref{qnls} and \eqref{dqnls} we encounter not only multilinear expressions, but also nonlinear functions such as $g(u)$. Bony's paradifferential calculus asserts that at the leading order we have
\[
g(u) \approx T_{g'(u)} u,
\]
but this expansion is not sufficiently accurate  for our purposes. The next step would be to use a second order expansion, which employs the continuous Littlewood-Paley decomposition. This has the form
\begin{equation}\label{para-expansion}
\begin{aligned}
g(u) = & \ g(u_0) + \int_{0}^\infty g'(u_{<k}) u_k \, dk
\\
 = & \ g(u_0) + \int_{0}^\infty g'(u_{0}) u_k \, dk + \iint_{0<k_2<k_1} 
 g''(u_{<k_2}) u_{k_2} u_{k_1} \, dk_1 dk_2.
\end{aligned}
\end{equation}
One could continue to reexpand to any order $n_0$,
\begin{equation}\label{para-expansion-n}
\begin{aligned}
g(u) =  & \ g(u_0) + \sum_{n=1}^{n_0-1}
 \int_{0 < k_1 \cdots < k_n} u_{k_1} \cdots u_{k_n} g^{(n)}(u_{0})  \, dk_1 \cdots dk_{n_0}
 \\ & +  \int_{0 < k_1 \cdots < k_{n_0}} u_{k_1} \cdots u_{k_{n_0}} g^{(n_0)}(u_{< k_{n_0}})  \, dk_1 \cdots dk_{n_0}.
\end{aligned}
\end{equation}
For us, the $n_0=4$ expansion will suffice.  A similar expansion is used later on to estimate differences $g(u) - g(u_{<k})$.

We remark that the  nonlinear expression $g^{(n)}(u_{<k_{n_0}})$ is not exactly localized at frequencies $<k_{n_0}$, but does have a better high frequency tail, see for instance Lemma~\ref{l:F-para} later on.

\section{A frequency envelope formulation of the results}

For expository purposes, the main results of this paper, namely Theorem~\ref{t:local}, Theorem~\ref{t:long} and Theorem~\ref{t:global}, were stated in a simplified form in the introduction. However, the full results that we prove provide a much more detailed picture, which gives a full family of $L^6_{t,x}$ Strichartz estimates and bilinear $L^2_{t,x}$ bounds for the solutions. Furthermore, the proofs of the results are complex bootstrap arguments relative to all these bounds, both linear and bilinear. For these reasons, it is important to have a good setup for both the results and for the bootstrap assumptions. An elegant way to do this is to use the language of frequency envelopes. Our goals in this section are 
\begin{enumerate}[label=(\roman*)]
\item to define the frequency envelopes notion we employ here, 
\item to provide a more accurate, frequency envelope formulation of the main results, 
\item  to provide the bootstrap assumptions in the proofs of each of the three theorems, and 
\item to outline the continuity argument which  allows us to use 
these bootstrap assumptions.
\end{enumerate}

\subsection{Frequency envelopes}
Before restating  the main theorems of this paper in a sharper, quantitative form, we revisit the \emph{frequency envelope} notion. This elegant and useful tool will streamline the exposition of our results, and one should think  of it as a bookkeeping device that efficiently tracks the evolution of the energy of the solutions between dyadic energy shells. 

Following Tao's paper \cite{Tao-BO}, we say that a sequence $c_{k}\in l^2$ is an $L^2$ frequency envelope for a function $\phi \in L^2$ if
\begin{itemize}
\item[i)] $\sum_{k=0}^{\infty}c_k^2 \lesssim 1$;\\
\item[ii)] it is slowly varying, $c_j /c_k \leq \step^{\delta \vert j-k\vert}$, with $\delta$ a very small universal constant;\\
\item[iii)] it bounds the dyadic norms of $\phi$, namely $\Vert P_{k}\phi \Vert_{L^2} \leq c_k$. 
\end{itemize}
Given a frequency envelope $c_k$ we define 
\[
 c_{\leq k} := (\sum_{j \leq k} c_j^2)^\frac12, \qquad  c_{\geq k} := (\sum_{j \geq k} c_j^2)^\frac12.
\]
\noindent \textbf{Note:} When using Greek letters like $\lambda$ and $\mu$ to 
represent dyadic frequencies, we will harmlessly  use the same Greek letter as an index for the frequency envelope, as in $c_\lambda$ or $c_\mu$.

\begin{remark}
To avoid dealing with certain issues arising at low frequencies, we can harmlessly make the extra assumption that $c_{0}\approx \|c\|_{\ell^2}$.
\end{remark}

\begin{remark}\label{r:unbal-fe}
Another useful variation is to weaken the slowly varying assumption to
\[
\step^{- \delta \vert j-k\vert} \leq    c_j /c_k \leq \step^{C \vert j-k\vert}, \qquad j < k,
\]
where $C$ is a fixed but possibly large constant. All the results in this paper are compatible with this choice. This offers the extra flexibility of providing higher regularity results by the same argument, as it allows the frequency envelope $c_k$ to be chosen so that we have
\begin{equation}
\|u\|_{H^s}^2 \approx \sum c_k^2 \step^{2k}, \qquad 0 \leq s \leq C-\delta,
\end{equation}
uniformly in $s$.
\end{remark}

\subsection{The frequency envelope form of the results}
\label{s:boot}

The  goal of this section is towfold: (i) restate our  main results in Theorem~\ref{t:local}, Theorem~\ref{t:long} and Theorem~\ref{t:global} in the frequency envelope setting,  and (ii) to set up the bootstrap argument for the proof of the estimates in the theorems. 

The set-up for the bootstrap is most conveniently described using the language  of frequency envelopes. This was originally introduced in the context of dyadic Littlewood-Paley decompositions in work of Tao, see e.g. \cite{Tao-WM}. Here we also use the dyadic setting, but we refer the reader to  \cite{IT-global} for another interesting general frequency envelope setting, based on lattice decompositions. 
\bigskip

To start with, we assume that the initial data has small size in $H^s$,
\begin{equation}\label{small-data}
\| u_0\|_{H^s} \lesssim \epsilon.
\end{equation}
We consider a dyadic frequency decomposition for the initial  data, 
\[
u_0 = \sum_{\lambda \geq 1} u_{0,\lambda}.
\]
where $\lambda = \step ^k$, $k \in \N$  indexes dyadic frequencies. Then we place the initial data components under an admissible  frequency envelope,
\[
\|u_{0,\lambda}\|_{H^s} \leq \epsilon c_\lambda, \qquad c \in \ell^2,
\]
where the envelope $\{c_k\}$ is not too large,
\begin{equation*}
\| c\|_{\ell^2} \approx 1.    
\end{equation*}

Our goal will be to establish similar frequency envelope bounds for the solution. 
We now state the frequency envelope bounds that will be the subject of the next theorems. These are as follows:

\begin{enumerate}[label=(\roman*)]
\item Uniform frequency envelope bound:
\begin{equation}\label{uk-ee}
\| u_\lambda \|_{L^\infty_t L_x^2} \lesssim \epsilon c_\lambda \lambda^{-s},
\end{equation}
\item Balanced bilinear $L^2_{t,x}$ bound:
\begin{equation} \label{uab-bi-bal}
\| \partial_x(u_\lambda \bu_\mu^{x_0})  \|_{L^2_{t,x}} \lesssim \epsilon^2 c_\lambda c_\mu \lambda^{-s} \mu^{-s} \lambda^\frac12 (1+ \lambda |x_0|)
, \qquad \lambda \approx \mu ,
\end{equation}

\item Unbalanced bilinear $L^2_{t,x}$ bounds:
\begin{equation} \label{uab-bi-unbal}
\| \partial_x(u_\lambda \bu_\mu^{x_0})  \|_{L^2_{t,x}} \lesssim \epsilon^2 c_\lambda c_\mu \lambda^{-s} \mu^{-s} \lambda^\frac12 
, \qquad \mu \ll \lambda,
\end{equation}

\item Localized Strichartz bound, short-time version:
\begin{equation}\label{uk-se-short}
\| u_\lambda \|_{L_{t,x}^6}^6 \lesssim (\epsilon c_\lambda)^6 \lambda^{-6s},
\end{equation}

\item Localized Strichartz bound, long-time version:
\begin{equation}\label{uk-se}
\| u_\lambda \|_{L_{t,x}^6}^6 \lesssim (\epsilon c_\lambda)^4 \lambda^{-4s-1}.
\end{equation}
\end{enumerate}

We remark on a special case of \eqref{uab-bi-bal} when $\lambda = \mu$ and $x_0 = 0$, where we get
\begin{equation} \label{uk-bi}
\| \partial_x(|u_\lambda|^2)  \|_{L^2_{t,x}} \lesssim \epsilon^2 c_\lambda^2 \lambda^{-2s}  \lambda^\frac12 .
\end{equation}
This can be seen as a dispersive estimate for $u_\lambda$,  
and will be useful when combined with \eqref{uk-ee} using interpolation.

We begin with the counterpart of Theorem~\ref{t:local}.

\begin{theorem}\label{t:local-fe}
a) Let $s > 1$ and $\epsilon \ll 1$. Consider the equation \eqref{qnls} with cubic nonlinearity.  Let $u \in C([0,T];H^s)$ be a smooth\footnote{This should be simply understood as $u \in CH^k$ with $k$ large enough.}  solution with initial data $u_0$  which has $H^s$ size at most $\epsilon$, and with
 \begin{equation}\label{small-t-loc}
 T \ll \epsilon^{-4}.
 \end{equation}
 Let $\epsilon \{c_\lambda\}$ be a frequency envelope for the initial data  in $H^s$. Then the solution $u$ satisfies the bounds \eqref{uk-ee}, \eqref{uab-bi-bal}, \eqref{uab-bi-unbal} and \eqref{uk-se-short} uniformly with respect to $x_0 \in \R$.

b) The same result holds for \eqref{dqnls} but with $s$ increased by one unit, $s > 2$. 
\end{theorem}

We continue with several remarks concerning our bounds for the solutions.
\begin{remark}
Our short-time $L^6_{t,x}$ Strichartz bounds in \eqref{uk-se-short} are sharp. This is unlike the standard Strichartz-type bounds for quasilinear equations, where one has Strichartz bounds with derivative losses.
\end{remark}

\begin{remark}
We also remark on the need to add translations to the bilinear  $L^2_{t,x}$ estimates. This is because, unlike the linear bounds \eqref{uk-ee-boot} and \eqref{uk-se-boot} which are inherently invariant with respect to translations, bilinear estimates are not invariant with respect to separate translations for the two factors. One immediate corollary of \eqref{uab-bi-unbal}, for instance, is that for any translation invariant bilinear form $L_{t,x}$ we have
\begin{equation} \label{uab-bi-unbal-L}
\| L(u_\lambda, \bar u_\mu)  \|_{L^2_{t,x}} \lesssim \epsilon^2 c_{\lambda} c_{\mu} \lambda^{-s-\frac12} \mu^{-s}, \qquad \mu \ll \lambda.
\end{equation}
This is the primary way we will use this translation invariance in our proofs.
\end{remark}

\begin{remark}
The bilinear $L^2_{t,x}$ bounds should be thought of less as Strichartz-type bounds and more as transversality bounds, i.e. where we pair two waves traveling in transversal directions. This is exactly the case for the unbalanced bound \eqref{uab-bi-unbal}. But it becomes more subtle  in the balanced case \eqref{uab-bi-bal}, where we allow for nearly parallel waves but penalize them with a factor which is proportional to the angle of interaction. In this context, it becomes critical to know that the two waves are moving along the same flow. This is no longer true if we begin to spatially translate one of the two waves, which is why we have to allow for the $(1+\lambda |x_0|)$ loss. 
\end{remark}

\medskip

We next consider the long-time solutions in Theorem~\ref{t:long}.

\begin{theorem}\label{t:long-fe}
a) Consider the equation \eqref{qnls} with phase rotation symmetry 
and conservative nonlinearity. Let $s > 1$. Let $u \in C([0,T];H^s)$ be a smooth solution  with initial data $u_0$  which has $H^s$ size at most $\epsilon$, and with 
\begin{equation}
 T \ll \epsilon^{-8}.
 \end{equation}
Let $\epsilon \{c_\lambda\}$ be a frequency envelope for the initial data 
in $H^s$. Then the solution $u$ satisfies 
the uniform $L^2$ bounds \eqref{uk-ee} in $[0,T]$, as well as 
the bilinear $L^2_{t,x}$ bounds \eqref{uab-bi-bal}, \eqref{uab-bi-unbal}
 and the long-time $L^6_{t,x}$  Strichartz 
bounds \eqref{uk-se} on subintervals of length $\epsilon^{-6}$.

b) The same result holds for \eqref{dqnls} but with $s$ increased by one unit, $s > 2$.
\end{theorem}

This is largely the same conclusion as in the previous theorem, but on a much longer time scale. One significant difference is in the $L^6_{t,x}$ Strichartz estimates, where we now allow for a $1/6$ derivative loss.

Finally, we consider bounds for global solutions in the setting of Theorem~\ref{t:global}.

\begin{theorem}\label{t:global-fe}
a) Consider the equation \eqref{qnls} with phase rotation symmetry  and conservative, defocusing  nonlinearity. Let $s > 1$. Let $u \in C([0,T];H^s)$ be a smooth solution with initial data $u_0$  which has $H^s$ size at most $\epsilon$. Let $\epsilon \{c_\lambda\}$ be a frequency envelope for the initial data 
in $H^s$. Then the solution $u$ satisfies 
the bounds \eqref{uk-ee}, \eqref{uab-bi-bal}, \eqref{uab-bi-unbal} and \eqref{uk-se} in the full time interval $[0,T]$.

b) The same result holds for \eqref{dqnls} but with $s$ increased by one unit, $s > 2$.
\end{theorem}

We continue with a result that applies for the linearized  equation, and which will be essential in order to obtain local in time bounds for the linearized equation and in turn for well-posedness. For a description of the linearized equation we refer the reader to the next section, and in particular the equation 
\eqref{qnls-lin}. Our bounds for the linearized equation are as follows:

\begin{theorem}\label{t:linearize-fe}
a) Consider the equation \eqref{qnls} with cubic nonlinearity.
Let $u$ be a solution as in Theorem~\ref{t:local-fe}. Let $v$ be a solution to the associated linearized equation 
 with $L^2$ initial data $v_0$ and  with frequency envelope $d_\lambda$. Then $v$ satisfies 
the following bounds:
\begin{enumerate}[label=(\roman*)]
\item Uniform frequency envelope bound:
\begin{equation}\label{uk-ee-lin}
\| v_\lambda\|_{L^\infty_t L^2_x} \lesssim d_\lambda ,
\end{equation}
\item Balanced bilinear $(v,v)$-$L^2$ bound:
\begin{equation} \label{vvab-bi-bal-lin}
\| \partial_x(v_\lambda \bv_\mu^{x_0})  \|_{L^2_{t,x}} \lesssim d_\lambda d_\mu  \lambda^\frac12(1+ \lambda |x_0|)
, \qquad \lambda\approx \mu ,
\end{equation}
\item Unbalanced bilinear $(v,v)$-$L^2$ bound:
\begin{equation} \label{vvab-bi-unbal-lin}
\| \partial_x(v_\lambda \bv_\mu^{x_0})  \|_{L^2_{t,x}} \lesssim d_\lambda d_\mu (\lambda+\mu)^\frac12
, \qquad \mu \not\approx \lambda,
\end{equation}
\item Balanced bilinear $(u,v)$-$L^2$ bound:
\begin{equation} \label{uvab-bi-bal-lin}
\| \partial_x(u_\lambda \bv_\mu^{x_0})  \|_{L^2_{t,x}} \lesssim c_\lambda d_\mu  \lambda^{\frac12-s}(1+ \lambda |x_0|)
, \qquad \lambda\approx \mu ,
\end{equation}
\item Unbalanced bilinear $(u,v)$-$L^2$ bound:
\begin{equation} \label{uvab-bi-unbal-lin}
\| \partial_x(u_\lambda \bv_\mu^{x_0})  \|_{L^2_{t,x}} \lesssim c_\lambda d_\mu \lambda^{-s} (\lambda+\mu)^\frac12
, \qquad \mu \not\approx \lambda ,
\end{equation}
\item Strichartz estimates:
\begin{equation}\label{v-se}
\|v_\lambda \|_{L^6_{t,x}} \lesssim d_\lambda.    
\end{equation}
 \end{enumerate}

b) The same result holds for \eqref{dqnls}  but with $s$ increased by one unit, $s > 2$.
\end{theorem}

This theorem in particular yields $L^2$ well-posedness for the linearized equation.

\bigskip

\subsection{The bootstrap hypotheses}

To prove the above theorems, we make a bootstrap assumption where we assume the same bounds but with a larger constant $C$, as follows:

\begin{enumerate}[label=(\roman*)]
\item Uniform frequency envelope bound:
\begin{equation}\label{uk-ee-boot}
\| u_\lambda \|_{L^\infty_t L_x^2} \leq C\epsilon c_\lambda \lambda^{-s},
\end{equation}
\item Balanced bilinear $L^2$ bound:
\begin{equation} \label{uab-bi-bal-boot}
\| \partial_x(u_\lambda \bu_\mu(\cdot+x_0))  \|_{L^2_{t,x}} \leq C^2\epsilon^2 c_\lambda c_\mu \lambda^{-s} \mu^{-s} \lambda^\frac12 (1+ \lambda |x_0|)
, \qquad \lambda \approx \mu, 
\end{equation}

\item Unbalanced bilinear $L^2$ bound:
\begin{equation} \label{uab-bi-unbal-boot}
\| \partial_x(u_\lambda \bu_\mu(\cdot+x_0))  \|_{L^2_{t,x}} \leq C^2 \epsilon^2 c_\lambda c_\mu \lambda^{-s} \mu^{-s} \lambda^\frac12 
, \qquad \mu \ll \lambda,
\end{equation}

\item Localized Strichartz bound, short time version:
\begin{equation}\label{uk-se-short-boot}
\| u_\lambda \|_{L_{t,x}^6}^6 \leq C^6 (\epsilon c_\lambda)^6 \lambda^{-6s},
\end{equation}

\item Localized Strichartz bound, long time version:
\begin{equation}\label{uk-se-boot}
\| u_\lambda \|_{L_{t,x}^6}^6 \leq C^4 (\epsilon c_\lambda)^4 \lambda^{-4s-1}.
\end{equation}
\end{enumerate}

Then we seek to improve the constant in these bounds. The gain will come from the fact that the in the proof $C$'s will always come paired with extra $\epsilon$'s. This will imply that the bounds we prove have universal implicit constants, independent of $C$, provided that  $\epsilon \ll_C 1$.

Precisely, we claim that

\begin{proposition}\label{p:boot}
a) It suffices to prove Theorem~\ref{t:local-fe}
under the bootstrap assumptions \eqref{uk-ee-boot}, \eqref{uab-bi-bal-boot},
\eqref{uab-bi-unbal-boot}
and \eqref{uk-se-short-boot}.

b) It suffices to prove Theorem~\ref{t:long-fe}
under the bootstrap assumptions \eqref{uk-ee-boot}, \eqref{uab-bi-bal-boot}, \eqref{uab-bi-unbal-boot}  and \eqref{uk-se-boot}.

c) It suffices to prove Theorem~\ref{t:global-fe}
under the bootstrap assumptions \eqref{uk-ee-boot}, \eqref{uab-bi-bal-boot}, \eqref{uab-bi-unbal-boot} and 
\eqref{uk-se-boot}.
\end{proposition}

Similar bootstrap assumptions are made for $v$ in the proof of Theorem~\ref{t:linearize-fe} in Section~\ref{s:lwp}.

\subsection{The continuity argument}
To conclude this section, we provide the continuity argument which shows that it suffices  to prove Theorem~\ref{t:local-fe} under the bootstrap assumptions \eqref{uk-ee-boot}-\eqref{uk-se-short-boot}. Similar arguments apply for Theorem~\ref{t:long-fe}
and Theorem~\ref{t:global-fe}.

We start with a regular initial data $u_0 \in H^k$ with large $k$, and we denote by $T_0$ the maximal time for which the  corresponding solution $u$ exists in $[0,T_0]$ and satisfies the
bounds 
\eqref{uk-ee-boot}-\eqref{uk-se-short-boot} in $[0,T_0]$. By the local well-posedness result in Theorem~\ref{t:regular} we have $T_0 > 0$. Assume by contradiction that $T_0 < T$.
Then the bootstrap version of the theorem implies that the bounds \eqref{uk-ee} -\eqref{uk-se} hold in $[0,T_0]$. In particular, $u(T_0)$ will also be controlled 
by the same frequency envelope $c_\lambda$ coming from the initial data, therefore 
it also belongs to $H^k$.
By Theorem~\ref{t:regular}, this implies in turn that 
the solution exists in a larger interval $[0,T_0+\delta]$
as a regular solution. Furthermore, if $\delta$ is small enough
then the bounds  \eqref{uk-ee-boot}-\eqref{uab-bi-unbal-boot} will hold in $[0,T_0+\delta]$. thereby contradicting the maximality of $T_0$.

\section{A paradifferential/resonant  expansion of the equation}
\label{s:menagerie}

Our goal here is to favourably recast our quasilinear Schr\"odinger evolution \eqref{qnls} or \eqref{dqnls} in a paradifferential fashion. To fix the notations, we will work with \eqref{qnls}; the analysis is essentially the same for \eqref{dqnls}.

Thus our starting point is the equation \eqref{qnls}. Given a solution $u$, we also consider the linearized equation around $u$. For convenience we write it in divergence form 
\begin{equation}\label{qnls-lin}
 i \partial_t v + \partial_x g(u) \partial_x v 
= N^{lin}_u(v,\partial_x v).  
\end{equation}
Here the nonlinearity $N^{lin}_u(v,\partial_x v)$ is at least cubic, and each term contains at most two derivatives, of which at most one is applied to $v$. 

We will expand the full equation \eqref{qnls} in a dual fashion, separating two principal components:
\begin{enumerate}[label=(\roman*)]
\item The paradifferential part, which accounts for the quasilinear character of the problem. This is essential for all of our results.

\item The doubly resonant part, which accounts for the nonperturbative semilinear part of the nonlinearity. This is critical for the long-time results.
\end{enumerate}
The linearized equation will be similarly expanded paradifferentially, but there we omit the second step. This suffices because the linearized equation is only needed for local well-posedness.

\subsection{The quasilinear Schr\"odinger equation and related flows}
In a first approximation, one may view the paradifferential part as simply obtained by truncating the coefficients in the principal part to lower frequencies,
\begin{equation}\label{paraT}
i \partial_t v + \partial_x T_{g(u)} \partial_x v 
= f.
\end{equation}
While this might suffice for the local well-posedness result in Theorem~\ref{t:local}, it is not precise enough for the long-time results, as the low frequencies of $g(u)$  also may include some doubly resonant $ high\times high$ quadratic contributions. This motivates us to make a better choice for the paracoefficients. Algebraically it will be easier to deal with  the corresponding frequency $\lambda$ evolution, which may be taken as
\begin{equation}\label{para}
i \partial_t v_\lambda + \partial_x g_{[<\lambda]} \partial_x v_\lambda 
= f_\lambda, \qquad v_\lambda(0) = v_{0,\lambda},
\end{equation}
where the truncated metric is defined as 
\begin{equation}
g_{[<\lambda]} := P_{<\frac{\lambda}{4}} g (u_{\ll\lambda}).
\end{equation}
Here we carefully localize  $u$ first to lower frequencies, 
rather than $g(u)$ directly; this is essential later on so that we do not include the doubly resonant interactions in the paradifferrential flow. On the other hand, using the second order elliptic operator in divergence form is more a convenience, as the corresponding commutator terms only play  a perturbative role.
\bigskip

One may express both the full equation and the linearized equation in terms of the paradifferential flow, in the form
\begin{equation}\label{para-full}
i \partial_t u_\lambda + \partial_x g_{[<\lambda]} \partial_x u_\lambda 
= N_\lambda(u), \qquad v_\lambda(0) = u_{0,\lambda},
\end{equation}
respectively
\begin{equation}\label{para-lin}
i \partial_t v_\lambda + \partial_x g_{[<\lambda]} \partial_x v_\lambda 
= N^{lin}_\lambda v, \qquad v_\lambda(0) = v_{0,\lambda},
\end{equation}
The source terms $N_\lambda(u)$, respectively $N^{lin}_\lambda$ will be written explicitly later when needed. They will be expected to play a perturbative role for the short-time results.

\bigskip

For the long-time results there is one additional portion of the nonlinearity
which plays a nonperturbative role, namely the doubly resonant part which corresponds to trilinear interactions of nearly equal frequencies.  To capture those, we introduce a symmetric symbol $c_{diag}$, smooth on the corresponding dyadic scale,
so that
\[
c_{diag}(\xi_1,\xi_2,\xi_3) := \left\{
\begin{aligned}
1, & \qquad \mbox{ when }\sum_{i,j=\overline{1,3}} |\xi_i - \xi_j| \ll \la \xi_1\ra + \la \xi_2 \ra 
+ \la \xi_3 \ra 
\\
0, & \qquad  \mbox{ when } \sum_{i,j=\overline{1,3}} |\xi_i - \xi_j| \gtrsim \la \xi_1\ra + \la \xi_2 \ra 
+ \la \xi_3 \ra .
\end{aligned}
\right.
\]
The cubic doubly resonant part of the nonlinearity is defined in terms of the symbol $c$ of the form in \eqref{eq-cubic} as 
\begin{equation}\label{defcl}
c_\lambda(\xi_1,\xi_2,\xi_3) := p_\lambda(\xi_1-\xi_2+\xi_3) \,c_{diag}(\xi_1,\xi_2,\xi_3) \,c(\xi_1,\xi_2,\xi_3).
\end{equation}
Here $p_\lambda$ is the symbol associated to the Littlewood-Paley projector $P_\lambda$.

Correspondingly, we rewrite the frequency localized evolution \eqref{para-full} in the form
\begin{equation}\label{full-lpara}
i \partial_t u_\lambda + \partial_x g_{[<\lambda]} \partial_x u_\lambda = C_\lambda(u,\bu, u) + N_\lambda^{tr}(u) ,
\end{equation}
where $N_\lambda^{tr}(u)$ contains the remaining terms, which 
exclude doubly resonant interactions and can be thought as transversal interactions, hence the  shorthand notation \emph{tr}. At leading order these interactions
can be classified as follows:

\begin{enumerate}[label=(\roman*)]
\item cubic $ high \times high$ terms of the form 
\[
\mu^2 P_{\lambda}L(u_{< \mu}, \bu_\mu, u_\mu) \qquad \mu \gtrsim \lambda.
\]

\item cubic $low \times high$ commutator terms of the form
\[
L(\partial_x u_{< \lambda}, \bu_{<\lambda}, \partial_x u_\lambda),
\qquad L( u_{< \lambda}, \partial_x \bu_{<\lambda}, \partial_x u_\lambda).
\]

\item quintic and higher  $ high \times high$ terms of the form 
\[
\mu^2 L(u_{< \mu}, \bu_{< \mu}, u_{< \mu},\bu_\mu,  u_\mu) \qquad \mu \gtrsim \lambda.
\]

\item quintic and higher $low \times high$ terms of the form
\[
\lambda^2 L(u_{\leq  \lambda}, \bu_{\leq  \lambda},
u_{\leq  \lambda}, \bu_{\leq  \lambda},u_\lambda) .
\]
\end{enumerate}

% To separate these four types of terms we will sometimes write \red{check if we actually use this}
% \begin{equation}
% N_{\lambda}^{\mg{tr}}(u) = N_\lambda^{3,hh} + N_{\lambda}^{3,lh}
% + N_\lambda^{5,hh} + N_{\lambda}^{5,lh}.
% \end{equation}
In our analysis the $N_\lambda^{tr}$ terms will play a perturbative role, though establishing that is not at all immediate, and  is instead part of the challenge. 

We do not repeat this step for the linearized equation, as it lacks some of the structure of the full equation and consequently we do not prove long-time bounds for it. Fortunately such bounds are not needed for the global dispersive estimates we aim to show.

\subsection{Source term bounds}

Now that we have the paradifferential representation \eqref{full-lpara} for \eqref{qnls}, the next step is to prove bounds for the source terms $C_\lambda(u,\bar u,u)$ and $N_\lambda^{tr}(u)$.
The question to ask at this point is, what is a good function space in which to measure these source terms ?

A naive attempt would be to try to use an $L^1_t L^2_x$  norm, as might be suggested by the energy estimates. But this falls short at low regularity for multiple reasons, as (i) it does not take into account the derivatives in the nonlinearity and (ii) our nonlinearity is merely cubic, so we  only have a limited supply of integrated decay to use. So this would only work at high regularity 
and locally in time. A second attempt would be to use a dual Strichartz norm $S'$; but even with this we end up facing the same obstacles.

Another idea introduced in \cite{MMT1,MMT2}  is to use instead local energy spaces. This had the merit of capturing the derivatives in a translation invariant fashion but it is not the most efficient at low regularity.

Here we introduce another strategy, which is somewhat indirect, but it is custom designed to make the most of the bilinear $L^2_{t,x}$ bounds. The starting point is the observation that in order  to prove $L^2$ energy estimates for $u_\lambda$ it would 
suffice to control the integral
\[
\iint N_\lambda(u) \bu_\lambda \, dx dt,
\]
over any time interval. This might motivate one to simply ask that $N_\lambda(u) \bu_\lambda \in L^1_{t,x}$. Such a condition would be enough 
in order to prove energy estimates, but little else. However, as it turns out, this can be made to work if we keep the principle but enlarge somewhat the class of functions to be measured in $L^1_{t,x}$. We will do this in two stages:

\begin{enumerate}[label = (\roman*)]
\item We add some  invariance with respect to translations to the above requirement, and look  at bounds for 
\[
 \| N_\lambda(u) \bu_\lambda^{x_0}\|_{L^1_{t,x}}.
\]
This will suffice in Section~\ref{s:para} for the proof of the bilinear $L^2_{t,x}$ estimates.

\item We enlarge the class of functions for the second factor, and look  instead at bounds for 
\[
 \| N_\lambda(u) \bv_\lambda^{x_0}\|_{L^1_{t,x}},
\]
where $v_\lambda$ belongs to a suitable class of solutions to the linear paradifferential equation \eqref{para}.  This will  be needed later on in Section~\ref{s:str} for the proof of the Strichartz estimates.
\end{enumerate}

Here we carry out the first stage of the two above. The proof is sufficiently robust so that it will carry through  to the second stage without changes when we finally need it.

While for the local well-posedness result one may think of $N_\lambda$ as a whole, for later use in the proof of the global result we will estimate the two components $C_\lambda(u,\bar u,u)$ and $N_\lambda^{tr}(u)$ separately. The bound for  $N_\lambda^{tr}$  is complex but more robust, and does not require the strong Strichartz estimate \eqref{uk-se-short-boot} or time-scale  restrictions. The bound for $C_\lambda$, on the other hand, is far simpler but does need the Strichartz bound \eqref{uk-se-short-boot} and time-scale restrictions. 
We will prove the following estimates:

\begin{proposition}\label{l:N-lambda}
 a) Let $s \geq 1$ and $x_0 \in \R$. Assume that the function $u$ satisfies   the bounds \eqref{uk-ee-boot}-\eqref{uab-bi-unbal-boot} and  \eqref{uk-se-boot}  in a time interval $[0,T]$. Then for $\epsilon$ small enough, the functions $N^{tr}_\lambda(u)$ in \eqref{para-full} satisfy 
 \begin{equation}\label{good-nl}
\| N^{tr}_\lambda(u) \bu_\lambda^{x_0}\|_{L^1_{t,x}} \lesssim C^4 \epsilon^4 c_\lambda^2 \lambda^{-2s}.
 \end{equation}
b) Assume in addition that $u$ satisfies \eqref{uk-se-short-boot} and that $T$ is as in \eqref{small-t-loc}. Then we also have 
 \begin{equation}\label{res-nl}
\| C_\lambda(u,\bu,u) \bu_\lambda^{x_0}\|_{L^1_{t,x}} \lesssim C^4 \epsilon^4 c_\lambda^2 \lambda^{-2s}.
 \end{equation}
\end{proposition}

As a preliminary step in the proof of the proposition, we consider nonlinear expressions of the form $g(u_{<\lambda})$, and show that they are essentially localized at frequencies $\lesssim \lambda$, in other words they satisfy very favourable bounds at frequencies $\gg \lambda$.

\begin{lemma}\label{l:Moser-lh}
Assume that \eqref{uk-ee-boot} and \eqref{uk-se-boot} hold. 
Then for $\mu \gg \lambda$ and large $N$ we have
\begin{equation}\label{a:Moser-tail}
\| P_\mu g(u_{<\lambda}) \|_{L^2_{t,x}} \lesssim C^4 \epsilon^2 c_\lambda^2
\lambda^{-3s} \left(\frac{\lambda}{\mu}\right)^N.
\end{equation}
\end{lemma}
We remark that here one could go further and estimate
other $L^p$ norms up to $L^1$. One could also  arbitrarily lower the power of $\lambda$. We stop at $p=2$ 
and the power $-3s$ because this is all that will be needed in the sequel.
\begin{proof}
We first use \eqref{uk-ee}, Bernstein's inequality
and dyadic summation to see that 
\begin{equation}
\|\partial^j u_{<\lambda} \|_{L^\infty_{t,x}} \lesssim C \epsilon \max\{ 1, c_\lambda \lambda^{j-s+\frac12}\},
\end{equation}
while from \eqref{uk-se} and dyadic summation we similarly have
\begin{equation}
 \|\partial^j u_{< \lambda}\|_{L^6_{t,x}} \lesssim C\epsilon^\frac23 \max\{ 1, c_\lambda^{\frac{2}{3}}\lambda^{j-\frac{4s+1}{6}}\} \end{equation}
 both with an additional $\log \lambda$ factor if 
 the exponent of $\lambda$ vanishes.

Combining the two, by chain rule one easily sees that similar bounds must hold for $g(u_{<\lambda})$. Choosing $j$ arbitrarily large, this in turn implies that for $\mu \gtrsim \lambda$ we have 
\begin{equation}\label{Pg}
\|P_{\mu}  g(u_{<\lambda}) \|_{L^\infty_{t,x}} \lesssim C \epsilon c_\lambda \mu^{-s+\frac12}
\left(\frac{\lambda}{\mu}\right)^N, \qquad \|P_{\mu} g(u_{< \lambda})\|_{L^6_{t,x}} \lesssim C \epsilon c_\lambda  \mu^{-\frac{4s+1}{6}}
\left(\frac{\lambda}{\mu}\right)^N. 
\end{equation}
Next we observe that the expression in \eqref{a:Moser-tail} vanishes if $g$ is a polynomial, so without loss of generality we can assume that $g$ is at least quintic at $0$.

Then we use the expansion in \eqref{para-expansion-n} for $u_{<\lambda}$ with $n_0 = 4$, estimating each of the terms. For this purpose we harmlessly switch from continuous to discrete summation, and also to dyadic frequency notation. For the first term we write 
\[
g(u_{0}) = u_{0}^4 h(u_{0}), 
\]
with smooth $h$ with $h(0)= 0$. After localization to frequency $\mu \gg 1$ this yields
\[
P_{\mu} g(u_{0}) = P_\mu (u_{0}^3 h(u_{0})) = 
P_\mu (u_{0}^3 P_{\approx \mu} h(u_{0})).
\]
Then we use \eqref{Pg} to estimate the first three factors in $L^6$, the next in $L^\infty$ and then the last one, localized to frequency $\approx \mu$,  also in $L^\infty$,
\[
\|P_\mu g(u_0)\|_{L^2_x} 
\lesssim \|u_0\|_{L^6_x}^3 \|u_0\|_{L^\infty_x} 
\| P_{\approx \mu}  h(u_{0})\|_{L^\infty_x} \lesssim C^5 \epsilon^4 
\mu^{-N}.
\]
The following terms are similar, so we move on to the last one, where,  with exponents satisfying $1 \leq \mu_4 \leq \mu_3 \leq \mu_2 \leq \mu_1 < \lambda$ we estimate 
\[
\begin{aligned}
\|P_{\mu} (u_{\mu_1} u_{\mu_2} u_{\mu_3} u_{\mu_4}g(u_{<\mu_4}))\|_{L^2} \lesssim & \ \|u_{\mu_1}\|_{L^6} 
\|u_{\mu_2}\|_{L^6} \|u_{\mu_3}\|_{L^6}
\|u_{\mu_4}\|_{L^\infty}
\|P_{\approx \mu}  g^{(4)}(u_{<\mu_4})\|_{L^\infty}
\\
\lesssim & \ C^5 \epsilon^4 (c_{\mu_1} c_{\mu_2} c_{\mu_3})^\frac23 c_{\mu_4}^2  
\mu_1^{-\frac{4s+1}{6}} \mu_2^{-\frac{4s+1}{6}}   \mu_3^{-\frac{4s+1}{6}} \mu_4^{-s+\frac12}
\left(\frac{\mu_4}{\mu}\right)^N.
\end{aligned}
\]
Since $s > 1$, the first four exponents are negative,
therefore the dyadic summation  with respect 
to $\mu_1$, $\mu_2$, $\mu_3$ and $\mu_4$ as above is straightforward, and the bound \eqref{Moser-tail} follows.

\end{proof}

We now return to the proof of the proposition.

\begin{proof}[Proof of Proposition~\ref{l:N-lambda}]
 We begin by computing the source terms $N_\lambda(u)$ in the paradifferential equation \eqref{para-full}. We have 
\[
\begin{aligned}
 N_\lambda(u) = & \  g_{[<\lambda]} \partial_x^2 u_\lambda - P_\lambda ( g(u) \partial_x^2 u)    
+ \partial_x g_{[<\lambda]} \partial_x u_\lambda + P_\lambda N(u,\partial_x u)
\\ 
= & \ 
\left([g_{[<\lambda]},P_\lambda]  \partial_x^2 u  
+ (\partial_x g_{[<\lambda]}) \partial_x u_\lambda\right) + P_\lambda( P_{\gtrsim \lambda} g(u_{\ll \lambda})
\partial_x^2 u)
\\ & \
- P_\lambda ( ( g(u)-g(u_{\ll\lambda})) \partial_x^2 u)    
 + P_\lambda N(u,\partial_x u)
\\
:= &  N^1_\lambda(u) + N^2_\lambda(u) + N^3_\lambda(u) + N^4_\lambda(u).
\end{aligned}
\]
We will successively estimate the contributions of 
these four terms to both \eqref{good-nl} and \eqref{res-nl}.
We remark that only the last two terms  $N^3_\lambda(u)$ and $N^4_\lambda(u)$ contain $C_\lambda$ contributions.
We also note that some $u$'s may be conjugated; this is 
not important for most cases, but will be noted where
it is relevant.

\bigskip

\textbf{I. The contribution of $N_\lambda^1$.}
 Here we will only use bilinear $L^2$ estimates. We can rewrite $N_{\lambda}^1(u)$ as a multilinear form
\[
N^1_\lambda(u) = L( \partial_x g_{[<\lambda]}, \partial_x u_\lambda) 
= L( \partial_x g(u_{\ll \lambda}), \partial_x u_\lambda) 
= L(h(u_{\ll \lambda}), u_{\ll \lambda}, \partial_x u_{\ll \lambda}, \partial_x u_\lambda) ,
\]
where we have used a standard commutator bound and the chain rule.
Here $h$ is a smooth function, and, by a slight abuse of notation,
$u_\lambda$ has a slightly larger Fourier support. Then
\[
N^1_\lambda(u) \bu_\lambda^{x_0} = 
L(h(u_{\ll \lambda}), u_{\ll \lambda}, \partial_x u_{\ll \lambda}, \partial_x u_\lambda,u_\lambda).
\]
The $h(u_{<\lambda})$ factor is trivially estimated in $L^\infty$. For the remaining four entries we may select two unbalanced pairs, and apply 
\eqref{uab-bi-unbal-boot} twice 
to obtain
\[
\| N^1_\lambda(u) \bu_\lambda^{x_0}\|_{L^1_{t,x}} \lesssim 
C^4 \epsilon^4 c_{\lambda}^2 \lambda^{-2s}.
\]

\bigskip

\textbf{II. The contribution of $N_\lambda^2$.}
Here we will use both bilinear $L^2$ estimates
and the long-time $L^6$ bound \eqref{uk-se-boot}\footnote{Even this is too strong here.}. We further decompose  $N_\lambda^2$ based on the frequency of the last factor,
\[
\begin{aligned}
N_\lambda^2 = & \ L(P_{\lambda} g(u_{\ll \lambda}), \partial^2 u_{<\lambda})
+ L(P_\lambda  g(u_{\ll\lambda}), \partial^2 u_{\lambda}) +
\sum_{\mu \gg \lambda} L( P_\mu g(u_{\ll\lambda}),\partial^2 u_\mu)
\\
:= & \ N_\lambda^{2,lo} + N_\lambda^{2,med} + N_\lambda^{2,hi}.
\end{aligned}
\]
Here we can use  Lemma~\ref{l:Moser-lh} to estimate
\begin{equation}\label{Moser-tail}
\| P_\mu h(u_{\ll \lambda}) \|_{L^2_{t,x}} \lesssim C^5 \epsilon^5 c_\lambda^5
\mu^{-3s}, \qquad \mu \gtrsim \lambda. 
\end{equation}
for any smooth function $h$. Now we proceed to estimate the contributions of the three terms.
\smallskip

The contribution of $ N_\lambda^{2,lo}$ can be rewritten in the form
\[
 N_\lambda^{2,lo} u_\lambda^{x_0} =   \lambda L(P_{\lambda} g(u_{\ll\lambda}), \partial u_{< \lambda}, u_\lambda).
\]
This can be estimated using \eqref{Moser-tail} for the first factor 
and \eqref{uab-bi-unbal-boot} for the remaining two,
\[
\| N_\lambda^{2,lo} u_\lambda^{x_0}\|_{L^1_{t,x}}
\lesssim \lambda \cdot C^5 \epsilon^5 c_\lambda^5 \lambda^{-3s} 
\cdot C^2 \epsilon^2 c_\lambda \lambda^{-s-\frac12} 
= C^7 \epsilon^7 c_\lambda^6 \lambda^{-4s +\frac12}, 
\]
which is better than needed for $s \geq 1$.
\smallskip

Using the chain rule, the contribution of $ N_\lambda^{2,mid}$ can be rewritten in the form
\[
 N_\lambda^{2,med} u_\lambda^{x_0} =  \lambda L(P_{\lambda} \partial g(u_{\ll\lambda}), \partial u_{\lambda}, u_\lambda) = \lambda L(P_{\lambda} g'(u_{\ll \lambda}), \partial_x u_{\ll \lambda}, \partial u_{\lambda}, u_\lambda).
\]
This is of the same type as the contribution of $N^1_\lambda$ but with an added $\partial u_\lambda$ 
factor, which we simply bound in $L^\infty$ using Bernstein's inequality and the energy bound \eqref{uk-ee-boot}:
\[
\| N_\lambda^{2,med} u_\lambda^{x_0}\|_{L^1_{t,x}}
\lesssim \lambda \cdot C^5 \epsilon^5 c_\lambda^5 \lambda^{-3s} \cdot \lambda^{-s+\frac32}
\cdot C^2 \epsilon^2 c_\lambda \lambda^{-s-\frac12} 
= C^7 \epsilon^7 c_\lambda^6 \lambda^{-5s +2}, 
\]
again much better than needed.
\smallskip

Finally the contribution of the summands in $ N_\lambda^{2,hi}$ can be rewritten in the form
\[
N_\lambda^{2,hi} u_\lambda^{x_0} = \mu^2 L(P_{\mu} g(u_{\ll\lambda}),  u_{\mu}, u_\lambda),
\]
and can be estimated again using \eqref{Moser-tail} for the first factor and \eqref{uk-se-short-boot} for the remaining two,
\[
\| N_\lambda^{2,hi} u_\lambda^{x_0}\|_{L^1_{t,x}}
\lesssim \mu^2 \cdot C^5 \epsilon^5 c_\mu^5 \mu^{-3s} 
\cdot C^2 \epsilon^2 c_\lambda c_\mu  \lambda^{-s} \mu^{-s-\frac12} 
= C^7 \epsilon^7 c_\lambda c_\mu^5 \mu^{-4s +\frac32} \lambda^{-s}.
\]
Here the power of $\mu$ is negative when $s \geq 1$, which guarantees the $\mu$ summation and yields agan a bound which is better then needed. 
\bigskip 

\textbf{III. The contribution of $N_\lambda^3$.}
We begin by writing a paradifferential expansion for the $g$ difference. Since $g-1$ 
is at least quadratic, we have
\begin{equation}\label{g-para}
g(u) - g(u_{\ll \lambda}) = \sum_{\mu \gtrsim \lambda} u_{\mu} g'(u_{\ll \lambda})+ \sum_{  \mu_1 \geq \mu_2 \gtrsim \lambda}
u_{\mu_1} u_{\mu_2} g''(u_{<\mu_2}).
\end{equation}
Now we differentiate the terms in $N_\lambda^3$ depending on the size of the $\mu$'s
and also on the frequency of the last $u$. There are many cases, but most can be easily 
dispensed with.
\begin{description}
\item[a) The contribution of $P_{\lambda} (u_\mu g'(u_{\ll \lambda}) \partial_x^2 u_{< \lambda})$] Here we obtain an expression of the form
\[
L(h(u_{\ll \lambda}), u_{\ll\lambda},u_\mu, \partial^2 u_{<\lambda}, u_{\lambda}),
\]
where we identify two unbalanced pairs of entries therefore we can apply the bilinear bound \eqref{uab-bi-unbal-boot} twice.

\item[b) The contribution of $P_{\lambda} (u_\mu g'(u_{\ll \lambda}) \partial_x^2 u_{ \lambda})$ with $\mu \gg \lambda$] Here we have an expression of the form
\[
L(h(u_{\ll \lambda}), u_{\ll\lambda},u_\mu, \partial^2 u_{\lambda}, u_{\lambda}),
\]
where we can again apply the bilinear bound \eqref{uab-bi-unbal-boot} twice.

\item[c) The contribution of $P_{\lambda} (u_\lambda g'(u_{\ll \lambda}) \partial_x^2 u_{ \lambda})$] 
This can be split into two terms
\[
P_{\lambda} (u_\lambda  P_{\gtrsim \lambda} g'(u_{\ll \lambda}) \partial_x^2 u_{ \lambda})+ 
P_{\lambda} (u_\lambda  P_{\ll \lambda} g'(u_{\ll \lambda}) \partial_x^2 u_{ \lambda}).
\]
The contribution of the first term has the form
\[
L(  P_{\gtrsim \lambda} g'(u_{\ll \lambda}), u_\lambda,u_\lambda,u_\lambda),
\]
and can be estimated using the  $L^2$ bound \eqref{Moser-tail} for the first entry, respectively the $L^6$ bound \eqref{uk-se-boot} for the other three entries.
The second term, on the other hand, yields the  expression
\[
P_{\lambda} (u_\lambda  P_{\ll \lambda} g'(u_{\ll \lambda}) \partial_x^2 u_{ \lambda}) 
u_{\lambda}^{x_0}.
\]
Now we consider the effect of the $P_\lambda$ projection.
If exactly one of the frequency $\lambda$ factors is conjugated, this guarantees that they  are $\lambda$ separated. Otherwise, they do not have to be separated,
but in that case they must be roughly at half frequency compared with $u_\lambda^{x_0}$. In both cases
we can again apply the bilinear bound \eqref{uab-bi-unbal-boot} twice.

\item[d) The contribution of $P_{\lambda} (u_\mu g'(u_{\ll\lambda}) \partial_x^2 u_{ \mu_1})$
with $\mu_1 \gg \lambda$] If $\mu \gtrsim \mu_1$ then we 
obtain 
\[
\mu_1^2 L( u_{\mu}, u_{\ll \lambda}, u_{\mu_1},u_\lambda),
\]
where we directly apply 
\eqref{uab-bi-unbal-boot} twice. This is no longer 
enough if $\mu_1 \gg \mu$, as we get two many $\mu_1$ factors. But then we can insert an extra projector
\[
P_{\lambda} (u_\mu g'(u_{<\lambda}) \partial_x^2 u_{ \mu_1})
= P_{\lambda} (u_\mu P_{\gtrsim \mu_1} g'(u_{\ll \lambda}) \partial_x^2 u_{ \mu_1}),
\]
in which case we can use \eqref{Moser-tail} and one instance of \eqref{uab-bi-unbal-boot}.

\item[e) The contribution of $P_{\lambda} (u_{\mu_1} u_{\mu_2}  g''(u_{\leq \mu_2}) \partial_x^2 u)$ with $\mu_1 \gg  \mu_2$] Here we must also have $\mu_1 \gg \lambda$. We separate in cases depending on the frequency $\mu_3$ of the last $u$. If $\mu_3 \approx \mu_1$ then we get 
\[
\mu_1^2 L( u_{\mu_1}, u_{\mu_2},  g''(u_{\leq \mu_2}), u_{\mu_1},u_\lambda),
\]
so we can directly apply \eqref{uab-bi-unbal-boot} twice.
If $\mu_3 \ll \mu_1$ then we can insert a projector on $g''$, and consider the expression 
\[
P_{\lambda} (u_{\mu_1} u_{\mu_2}  P_{\approx \mu_1} g''(u_{\leq \mu_2}) \partial_x^2 u_{\mu_3}),
\]
where we can use \eqref{Moser-tail} and one instance of \eqref{uab-bi-unbal-boot}. The situation is similar
if $\mu_3 \gg \mu_1$, where inserting the appropriate 
projector we need to consider 
 the expression 
\[
P_{\lambda} (u_{\mu_1} u_{\mu_2}  P_{\approx \mu_3} g''(u_{\leq \mu_2}) \partial_x^2 u_{\mu_3}).
\]

\item[f) The contribution of $P_{\lambda} (u_{\mu} u_{\mu}  g''(u_{\leq \mu}) \partial_x^2 u_{\not \approx \mu})$ with $\mu \gg  \lambda$] Here we consider the frequency of the last $u$.
If it is $\ll \mu$ then we directly apply \eqref{uab-bi-unbal-boot} twice.
If it is $\gg \mu$ then we insert another projector on $g''(u_{\leq \mu})$ and use 
\eqref{Moser-tail}.

\item[g) The contribution of $P_{\lambda} (u_{\mu} u_{\mu}  g''(u_{\leq \mu}) \partial_x^2 u_{\mu})$ with $\mu \gg  \lambda$] Here we consider the $g''$ term.
If it is constant then the three $\mu$ frequencies cannot be all balanced so we are allowed to use 
\eqref{uab-bi-unbal-boot} once for the unbalanced $\mu$ pair, and a second time for the $(\lambda,\mu)$ pair.
Else we can pull another $u_{\leq \mu}$ factor from $g''$. If this term is lower frequency,  $u_{\ll \mu}$, then we can apply the bilinear bound \eqref{uab-bi-unbal-boot} twice. Finally, we are left with a contribution
of the form
\[
L(u_\mu,u_\mu,u_\mu, \partial^2 u_\mu, u_\lambda),
\]
where we can use a bilinear $L^2$ bound \eqref{uab-bi-unbal-boot} and three $L^6$ bounds. However, we have no need of the strong $L^6$ bound \eqref{uk-se-short-boot}, and instead it suffices to use \eqref{uk-se-boot}.

\item[h) The contribution of $P_{\lambda} (u_{\lambda} u_{\lambda}  g''(u_{\leq \lambda}) \partial_x^2 u_{\mu})$, with $\mu \gg \lambda$] Here we can freely insert a $P_{\approx \mu}$ projector  on the $g''$ 
factor and then apply \eqref{Moser-tail}
and \eqref{uab-bi-unbal-boot}.

\item[k) The contribution of $P_{\lambda} (u_{\lambda} u_{\lambda}  g''(u_{\leq \lambda}) \partial_x^2 u_{\ll \lambda})$] 
If $g''$ is constant, then we argue as in case (c). Either the two $u_\lambda$'s are unbalanced 
(which is guaranteed if exactly one of them is conjugated) and then we can use 
\eqref{uab-bi-unbal-boot} twice, or they are balanced, but then they cannot also be balanced 
with the additional $u_\lambda^{x_0}$ and again 
we can use \eqref{uab-bi-unbal-boot} twice.

Else, we can pull another $u_{\leq \lambda}$ factor
from $g''$. If this factor is lower frequency  $u_{\ll \lambda}$ then using \eqref{uab-bi-unbal-boot} twice is enough.
Otherwise we have a contribution
\[
L(u_\lambda,u_\lambda,u_\lambda, \partial^2 u_{\ll \lambda}, u_\lambda),
\]
and we can use the weaker $L^6$ bound \eqref{uk-se-boot} together with one instance of \eqref{uab-bi-unbal-boot}.

\item[l) The contribution of $P_{\lambda} (u_{\lambda} u_{\lambda}  g''(u_{\leq \lambda}) \partial_x^2 u_{\lambda})$, nonconstant part of $g''$] Here we  pull another $u_{\leq \lambda}$ factor from $g''$. If  this factor is lower frequency, $u_{\ll\lambda}$, then we can use thrice the weaker $L^6$ bound \eqref{uk-se-boot} together with one instance of \eqref{uab-bi-unbal-boot}. On the other hand if it is at comparable frequency,
$u_\lambda$, then  we arrive at
\[
P_{\lambda} (h(u_{<\lambda}),u_\lambda, u_\lambda,u_\lambda,
\partial^2 u_\lambda) u_\lambda^{x_0},
\]
with a smooth function $h$.

If $h$ is constant then, regardless of the distribution of the complex conjugates,  the four frequencies under the first projection cannot be all balanced, so the same argument as above applies.

On the other hand if $h$ is not constant then we can pull another $u_{<\lambda}$ factor.  If this is 
at frequency $\ll \lambda$ then we argue as above. Else 
we have obtained an expression of the form
\[
P_{\lambda} (h(u_{<\lambda}),u_\lambda, u_\lambda, u_\lambda,u_\lambda,
\partial^2 u_\lambda) u_\lambda^{x_0},
\]
where we simply use  the $L^6$ bound \eqref{uk-se-boot} 
six times.

\item[m) The contribution of $P_{\lambda} (u_{\lambda} u_{\lambda}  g''(0) \partial_x^2 u_{\lambda})$] This is a quartic expression
\[
 g''(0) P_{\lambda} (u_{\lambda} u_{\lambda}  \partial_x^2 u_{\lambda}) u_\lambda^{x_0}.
\]
Here the frequencies in the trilinear form are only comparable to $\lambda$,
so they might still be either balanced or unbalanced. 
Hence we split into 

(i) an unbalanced part (with frequencies $O(\lambda)$ separated),
where we can use two instances of \eqref{uab-bi-unbal-boot},
and 

(ii) a balanced contribution which is included in $C_\lambda$, and is discussed later.  Such balanced contributions can occur only for trilinear forms which have phase rotation symmetry.
\end{description}

\bigskip 

\textbf{IV. The contribution of $N_\lambda^4$.}
Here we recall that $N$ is cubic, and at most quadratic in $\partial u$. The nontrivial case is when $N$ is exactly quadratic in $\nabla u$,
so we write schematically 
\[
N(u,\partial u) = h(u) (\partial u)^2. 
\]
Expanding $h$ paradifferentially to first order,
we write the contribution of $N_\lambda^4$ as a sum of terms of the form
\[
L( P_\lambda (h_1'(u_{<\mu}) u_\mu \partial u \partial u), u_\lambda).
\]
This expression is similar to $N_\lambda^3$, in that it is al least quartic the total number of derivatives is the same, namely two. But it is also better, as the two derivatives are now distributed to different factors.

 To estimate the contribution of $N_\lambda^4$ 
 one should consider again a number of cases depending on the size of the input frequencies relative to $\lambda$. In view of the above comparison, these cases exactly mirror the ones for $N^3_\lambda$, with bounds which are either similar or better  than in the $N_\lambda^3$ case. There is one exception to this,
namely the $N_\lambda^4$ terms which correspond to the 
excluded interactions in the contribution of $N^3_\lambda$, namely the paradifferential ones of the form
\[
L( P_\lambda (h(u_{\ll\lambda}) u_{\ll \lambda}  u_{\ll \lambda} \partial^2 u_\lambda), u_\lambda).
\]
For this expression the bound \eqref{good-nl} fails.
But its counterpart corresponding to $N_\lambda^4$
is 
\[
L( P_\lambda (h(u_{\ll\lambda}) u_{\ll \lambda}  \partial u_{\ll \lambda} \partial u_\lambda), u_\lambda),
\]
which is better balanced and can be estimated easily 
with two applications of the bilinear $L^2$ bound
\eqref{uab-bi-unbal-boot}.
 
 Given the above discussion, there is no need to enumerate again all the cases. However, for 
 clarity we list the key points.

\begin{enumerate}[label=(\roman*)]
     \item Due to \eqref{Moser-tail}, we can freely localize $h'(u_{<\mu})$  at frequency $\lesssim  \mu$.
     
    \item If we have two pairs of unbalanced frequencies then it suffices to use \eqref{uab-bi-unbal-boot} twice.

\item If we have one frequency $\lambda$ input  
and exactly three higher, then the three higher 
must be unbalanced so  \eqref{uab-bi-unbal-boot} twice suffices.

\item If we have one frequency $\lambda$ input  
and at least four higher, then the higher frequency factors 
may be balanced but then we can use 
\eqref{uk-se-boot} thrice and one instance 
of \eqref{uab-bi-unbal-boot}.

\item If we have three frequency $\lambda$ inputs
and a lower one then the three  
must be unbalanced so  \eqref{uab-bi-unbal-boot} twice suffices again.

\item If we have four frequency $\lambda$ inputs
and a lower one then we can use 
\eqref{uk-se-boot} thrice and one instance 
of \eqref{uab-bi-unbal-boot}. The same applies 
with five frequency $\lambda$ inputs, which 
must be unbalanced.

\item If we have exactly four frequency $\lambda$ 
inputs, then we separate into an unbalanced 
term and a balanced one which is included in $C_\lambda$.
\end{enumerate}

 To summarize, at the conclusion of this step we have considered all components in $N_\lambda(u)$ except for the cubic balanced contribution in $C_\lambda$, and we have proved the bound \eqref{good-nl}.
  The $C_\lambda$ term is considered in the last step.

\bigskip 

\textbf{IV. The bound for $C_\lambda$.}
For this term we have four doubly resonant frequencies, where we have no choice but to use the strong Strichartz bound \eqref{uk-se-short-boot} three times, and energy once,
\[
\| C_\lambda(u,u,u) u_\lambda^{x_0}\|_{L^1} \lesssim T^\frac12 
\| u_\lambda\|_{L^6}^3 \| u_\lambda\|_{L^\infty L^2} \lesssim C^4 T^\frac12\epsilon^4 c_\lambda^{4} \lambda^{2-4s},
\]
where the $T^\frac12$ factor comes from using H\"older's inequality in time.
Since in ths case we are assuming that $T \ll \epsilon^{-4}$, the bound 
\eqref{res-nl} follows.

\end{proof}

\section{Density-flux relations for mass and momentum} \label{s:df}

While our problem admits no exact conservation laws in general, propagating energy bounds plays an important role in this paper
so we are interested in deriving efficient almost conservation laws. But in addition to that, following the ideas introduced 
in \cite{IT-global}, the almost conservation laws also play a key role in the proof of the bilinear $L^2_{t,x}$ bounds. \cancel{However} In order for this to work, we need to think of our almost conservation laws
in a more accurate fashion, namely as density-flux relations.
This is the goal of this section.

\subsection{Conservation laws for the linear problem}
For clarity we begin our discussion with the linear Schr\"odinger equation
\begin{equation}
i u_t + u_{xx} = 0, \qquad u(0) = u_0,
\end{equation}
following the exposition in \cite{IT-global}.

For this we consider the following three conserved quantities, the mass
\[
\bM(u) = \int |u|^2 \,dx,
\]
the momentum 
\[
\bP(u) = 2 \int \Im ( u \partial_x \bar u) \,dx,
\]
as well as the energy
\[
\bE(u) = 4 \int |\partial_x u|^2\, dx. 
\]

To these quantities we associate corresponding densities 
\[
M(u) = |u|^2, \qquad P(u) = i ( \bar u \partial_x u - u \partial_x \bar u), \qquad E(u) = - \bar u \partial_x^2 u   + 
2  |\partial_x u|^2 -  u \partial_x^2 \bar u.
\]
The densities here are not uniquely determined; and our choices are
 motivated by the conservation law computations, known in the literature as \emph{density-flux identities}:
\begin{equation}\label{df-lin}
\partial_t M(u) = \partial_x P(u), \qquad \partial_t P(u) = \partial_x E(u).
\end{equation}
The symbols of these densities viewed as bilinear forms are
\[
m(\xi,\eta) = 1, \qquad p(\xi,\eta) = -(\xi+\eta), \qquad e(\xi,\eta) = (\xi+\eta)^2.
\]

We are also interested in frequency localized versions of these objects. 
Given a dyadic frequency $\lambda$, we start with a  smooth bounded symbol $a_\lambda(\xi,\eta)$ which is localized at frequency $\lambda$ and which is symmetric, in the sense that 
\[
a_\lambda(\eta,\xi) = \ol{a_\lambda(\xi,\eta)}.
\]
 A convenient choice is to 
take for instance
\begin{equation}\label{choose-a-loc}
a_\lambda(\xi,\eta) = \phi_\lambda(\xi) \phi_\lambda(\eta).
\end{equation}
where $\phi_\lambda$ is the symbol of $P_\lambda$, see \eqref{def-phi}. 
Then we define an associated weighted mass density by 
\[
M_\lambda(u) := A_\lambda(u,\bar u).
\]
We also define corresponding momentum and energy symbols $p_\lambda$ and $e_\lambda$ by 
\[
p_\lambda(\xi,\eta) = -(\xi+\eta)a_\lambda(\xi,\eta), \qquad e_\lambda(\xi,\eta) = (\xi+\eta)^2a_\lambda(\xi,\eta).
\]
Then a direct computation yields the density-flux relations
\begin{equation}
\partial_t M_\lambda(u,\bar u) = \partial_x P_\lambda(u,\bar u), \qquad \partial_tP_\lambda(u,\bar u) = \partial_x E_\lambda(u,\bar u).
\end{equation}

\subsection{ Density-flux identities for the paradifferential problem}
\
We begin with solutions $v_\lambda$ to the paradifferential equation \eqref{para}. For these we have 
\[
\begin{aligned}
\partial_t M(v_\lambda) = &\ 2 \Re ( \partial_t v_\lambda \cdot \bar v_\lambda) 
\\
= &\ - 2 \Im (\partial_x g_{[<\lambda]} \partial_x v_\lambda   \cdot \bar v_\lambda) 
+ 2 \Im ( f_\lambda \bar v_\lambda)
\\
= & \ - 2 \partial_x [g_{[<\lambda]} \Im (\partial_x v_\lambda   \cdot \bar v_\lambda)] 
+ 2 \Im ( f_\lambda \bar v_\lambda),
\end{aligned}
\]
which we rewrite in the form 
\begin{equation}\label{dens-flux-param}
\partial_t M(v_\lambda) =  \partial_x [g_{[<\lambda]} P(v_\lambda)] + F^{para}_{\lambda,m}, \qquad F^{para}_{\lambda,m} = 2 \Im ( f_\lambda \bar v_\lambda).
\end{equation}
Similarly, one computes 
\begin{equation}\label{dens-flux-parap}
\partial_t P(v_\lambda) =  \partial_x [g_{[<\lambda]} E(v_\lambda)] + F^{para}_{\lambda,p}, \qquad F^{para}_{\lambda,p} = 2 \Re( \partial_x f_\lambda \bar v_\lambda - f_\lambda \partial_x \bar v_\lambda).
\end{equation}

\bigskip

\subsection{Nonlinear density-flux identities for the frequency localized mass and momentum} 
Here we develop the counterpart of the linear analysis above for the nonlinear problem \eqref{qnls}. 
We start from the frequency localized mass and momentum densities as above, still denoted by $M_\lambda$ and $P_\lambda$. For these we derive appropriate density-flux relations.

We begin with a simpler computation for the mass density $M_\lambda$, using the equation \eqref{full-lpara}. We have 
\[
\begin{aligned}
\partial_t M_\lambda(u) = &\ 2 \Re ( \partial_t u_\lambda \cdot \bar u_{\lambda}) 
\\
= &\ 2 \Im (\partial_x g_{[<\lambda]} \partial_x u_\lambda   \cdot \bar u_\lambda) 
+ 2 \Im ( (C_\lambda(u,\bu, u) + N_\lambda(u,\bu,u)) \bar u_\lambda)
\\
= & \ - 2 \partial_x [g_{[<\lambda]} \Im (\partial_x u_\lambda   \cdot \bar u_\lambda)] 
+ 2 \Im ( (C_\lambda(u,\bu, u)  \bu_\lambda)   + 2 \Im (N^{tr}_\lambda(u,\bu,u) \bar u_\lambda),
\end{aligned}
\]
which we rewrite in the form
\begin{equation}\label{dens-flux-m0}
\partial_t M_\lambda(u) =  2 \partial_x [g_{[<\lambda]} P_\lambda(u)] 
+ C^4_{\lambda,m}(u,\bu,u,\bu)  + 2 F^4_{\lambda,m}(u,\bu,u,\bu),
\end{equation}
where 
\[
C^4_{\lambda,m} (u,\bu,u,\bu) = 2 \Im  (C_\lambda(u,\bu, u)  \bu_\lambda), 
\qquad F^4_{\lambda,m} (u,\bu,u,\bu) = 2 \Im (N^{tr}_\lambda(u,\bu, u)  \bu_\lambda).
\]
Here the term $C^4_{\lambda,m}$ contains only interactions 
which are localized at frequency $\lambda$, and  which  therefore contain the doubly resonant regime. For the local well-posedness
result, we will still treat this term perturbatively, though
doing this will require better (lossless) short-time Strichartz
estimates.

However, the above strategy no longer works for the long-time results. Instead, following \cite{IT-global}, we will remove the term $C^4_{\lambda,m}$
via a quartic correction to the mass
density, and a quartic correction to the mass flux.
This is where the phase rotation symmetry and conservative assumptions play a key role. Unlike in \cite{IT-global}, where this computation was exact, in  our case it also produces some higher order perturbative errors.

The term $F^4_{\lambda,m}$, on the other hand, only involves 
transversal interactions and will be treated perturbatively.
Such contributions also did not exist in \cite{IT-global},
but are introduced here as a way to streamline the exposition.

We now discuss the correction associated to $C^4_{\lambda,m}$. By the phase rotation symmetry, $C^4_{\lambda,m}$ has arguments 
$C^4_{\lambda,m}(u,\bu,u,\bu)$. A-priori its symbol, defined on the diagonal $\Delta^4 \xi = 0$,
is given by
\[
c^4_{\lambda,m}(\xi_1,\xi_2,\xi_3,\xi_4) = 
 i c_{\lambda}(\xi_1,\xi_2,\xi_3)  - i \bar c_{\lambda}(\xi_2,\xi_3,\xi_4).
\]
However, we can further symmetrize separately in $(\xi_1,\xi_3)$ and $(\xi_2,\xi_4)$ and replace it by 
\[
\begin{aligned}
c^4_{\lambda,m}(\xi_1,\xi_2,\xi_3,\xi_4) = \frac{i}{2} \left[ -  (c_{\lambda}(\xi_1,\xi_2,\xi_3) + c_{\lambda}(\xi_1,\xi_4,\xi_3))
%\right.\\ \left.
+ (\bar c_{\lambda}(\xi_2,\xi_3,\xi_4)+ 
\bar c_{\lambda}(\xi_2,\xi_1,\xi_4))  \right].
\end{aligned}
\]
In particular we see that our conservative assumption guarantees that the symmetrized symbol $c^4_{\lambda, m}(\xi_1,\xi_2,\xi_3,\xi_4)$ vanishes of second order on the doubly resonant set
\[
\calR_2 = \{(\xi_1, \xi_2, \xi_3, \xi_4)\in \mathbb{R}^4 \, / \, \xi_1 = \xi_3 = \xi_2 = \xi_4\}.
\]

For comparison purposes, we recall here that in our previous work \cite{IT-global} we had a slightly  stronger conservative assumption, which was sufficient to guarantee that $c$ vanishes on the full 
resonant set
\[
\calR = \{(\xi_1, \xi_2, \xi_3, \xi_4)\in \mathbb{R}^4 \, / \,\Delta^4 \xi = 0, \,  \Delta^4 \xi^2 = 0\} = \{ \{ \xi_1,\xi_3\} = \{\xi_2,\xi_4\} \}.
\]
That in turn implied (see Lemma~4.1 in \cite{IT-global}) that we could 
 express it smoothly in the form
\begin{equation}\label{choose-R4m}
c^4_{\lambda,m} +  i  \Delta^4 \xi^2 \,  b^4_{\lambda,m} = i \Delta^4 \xi\, r^4_{\lambda,m} .
\end{equation}
Here the quartic form $B^4_{\lambda,m}$ was viewed as a 
correction to the localized mass density, while $R^4_{\lambda,m}$ 
was the corresponding flux correction. 

In our context here, we cannot hope to have the representation \eqref{choose-R4m}, but we will instead prove that we have a weaker replacement, which is just as good for our purposes. Our new smooth representation has the form
\begin{equation}\label{choose-R4m-new}
c^4_{\lambda,m} +  i  \Delta^4 \xi^2 \,  b^4_{\lambda,m} = i \Delta^4 \xi\, r^4_{\lambda,m} 
+ i f^{4,bal}_{\lambda,m},
\end{equation}
where the additional source term has the redeeming feature that can be represented in the form
\begin{equation}\label{fbal-m}
   f^{4,bal}_{\lambda,m} =  q^{4,bal}_{\lambda,m}[ (\xi_1-\xi_2)(\xi_3-\xi_4)+ (\xi_1-\xi_4)(\xi_2-\xi_3)],
\end{equation}
and in particular vanishes to second order on the doubly resonant set.
Finding such a representation falls into the class of \emph{division problems}, with the additional twist that here we have two divisors, namely $\Delta^4 \xi$ and $\Delta^4 \xi^2$.

Arguing exactly as in \cite{IT-global},  we
define the modified
mass density
\begin{equation}\label{m-sharp}
\ms_\lambda(u) = M_\lambda(u) + B^4_{\lambda,m}(u,\bar u, u,\bar u),
\end{equation}
and rewrite the relation \eqref{dens-flux-m0} in the 
better form
\begin{equation}\label{dens-flux-m}
\begin{aligned}
\partial_t \ms_\lambda(u) = & \  \partial_x [g_{[<\lambda]} P_\lambda(u) 
+ R^4_{\lambda,m}(u,\bu,u,\bu)]  + 
 R^6_{\lambda,m}(u,\bu,u,\bu,u,\bu)
 \\ & + F^4_{\lambda,m}(u,\bu,u,\bu)+F^{4,bal}_{\lambda,m}(u,\bu,u,\bu),
 \end{aligned}
\end{equation}
where $R^6_{\lambda,m}$ represents the  $6$-linear and higher form 
arising from the time derivative of $B^4_{\lambda,m}(u,\bu,u,\bu)$
via the equation \eqref{qnls}.

This will be our main density-flux relation for mass. We argue in a similar fashion for the momentum. The symbol of the source term in the momentum equation is
\[
c^4_p(\xi_1,\xi_2,\xi_3,\xi_4) =  i (\xi_1-\xi_2+\xi_3+\xi_4) c(\xi_1,\xi_2,\xi_3) - i (\xi_1+\xi_2-\xi_3+\xi_4)\bar c(\xi_2,\xi_3,\xi_4),
\]
which is then symmetrized to
\[
\begin{aligned}
c^4_p(\xi_1,\xi_2,\xi_3,\xi_4) = & \  \frac{i}2 [(\xi_1-\xi_2+\xi_3+\xi_4) c(\xi_1,\xi_2,\xi_3)  + (\xi_1+\xi_2+\xi_3-\xi_4) c(\xi_1,\xi_4,\xi_3) 
\\
& - (\xi_1+\xi_2-\xi_3+\xi_4)\bar c(\xi_2,\xi_3,\xi_4)-(-\xi_1+\xi_2+\xi_3+\xi_4)\bar c(\xi_2,\xi_1,\xi_4)],
\end{aligned}
\]
and also vanishes of second order on the doubly resonant set $\calR_2$.
This will yield a representation
\begin{equation}\label{choose-R4p-new}
c^4_{\lambda,p} +  i  \Delta^4 \xi^2 \,  b^4_{\lambda,p} = i \Delta^4 \xi\, r^4_{\lambda,p} 
+ i f^{4,bal}_{\lambda,p},
\end{equation}
so that we have a similar correction for the momentum
\begin{equation}\label{p-sharp}
\ps_\lambda(u) = P_\lambda(u) + B^4_{\lambda,p}(u,\bar u, u,\bar u),
\end{equation}
as well as the associated density-flux relation
\begin{equation}\label{dens-flux-p}
\begin{aligned}
\partial_t \ps_\lambda(u) =  & \ \partial_x [g_{[<\lambda]} E_\lambda(u) 
+ R^4_{\lambda,p}(u,\bu,u,\bu)]  + 
 R^6_{\lambda,p}(u,\bu,u,\bu,u,\bu) 
\\ & \  
+ F^4_{\lambda,p}(u,\bu,u,\bu)+  F^{4,bal}_{\lambda,p}(u,\bu,u,\bu).
\end{aligned}
\end{equation}
We summarize the main properties of the symbols in the above 
decompositions in the next proposition:

\begin{proposition}\label{p:dens-flux-alg}
a) The mass density-flux relation \eqref{dens-flux-m} holds with symbols
$b^4_{\lambda,m}, r^4_{\lambda,m}$ and  $f^{4,bal}_{\lambda,m}$ which have the following properties:

\begin{enumerate}[label=(\roman*)]
\item Division relation: $b^4_{\lambda,m}, r^4_{\lambda,m}$  and  $f^{4,bal}_{\lambda,m}$ are localized at frequency $\lambda$ and satisfy \eqref{choose-R4m} and \eqref{fbal-m}.

\item Size and regularity: for each multiindex $\alpha$ we have
\begin{equation}\label{div-m}
|\partial_\xi^\alpha  b^4_{\lambda,m}| \lesssim \lambda^{-|\alpha|},
\qquad 
|\partial_\xi^\alpha  r^4_{\lambda,m}| \lesssim \lambda^{1-|\alpha|}
\qquad
|\partial_\xi^\alpha  q^{4,bal}_{\lambda,m}| \lesssim \lambda^{-|\alpha|}.
\end{equation}
\end{enumerate}

b) The same property holds for the momentum density-flux relation \eqref{dens-flux-p}, where the corresponding symbols satisfy
\begin{equation}\label{div-p}
|\partial_\xi^\alpha  b^4_{\lambda,p}| \lesssim \lambda^{1-|\alpha|},
\qquad 
|\partial_\xi^\alpha  r^4_{\lambda,p}| \lesssim \lambda^{2-|\alpha|}
\qquad
|\partial_\xi^\alpha  q^{4,bal}_{\lambda,p}| \lesssim \lambda^{1-|\alpha|}.
\end{equation}

\end{proposition}

\begin{proof}[Proof of Proposition~\ref{p:dens-flux-alg}]
It suffices to consider part (a), as part (b) is similar.
If we knew that $c^4_{\lambda,m}$ vanishes on the full resonant set $\calR$, then 
it would suffice to take $f^{4,bal}_{\lambda,m}=0$, and then
this result would be a localized and rescaled version of Lemma~4.1 in \cite{IT-global}. So our objective will be to reduce the problem to this case. 

In this proposition $\lambda$ plays the role of a scaling parameter, so without any loss in generality 
we can set $\lambda=1$, in which case all our functions will be supported in a unit size cube.
We begin by examining the behavior of $c^4_{\lambda,m}$ on $\calR$.
For this we parametrize $\calR$ with two frequencies $\xi$ and $\eta$,
so that 
\[
\{ \xi_1,\xi_3\} = \{\xi_2,\xi_4\} = \{ \xi,\eta\}.
\]
With this parametrization $\calR$ can be seen as the union 
of two transversal planes intersecting on a line in $\R^4$.
Restricted to this set, we can view $c^4_{\lambda,m}$ as a smooth
symmetric function of $\xi$ and $\eta$,
\[
c^4_{\lambda,m}(\xi_1,\xi_2,\xi_3,\xi_4) = h(\xi,\eta) \qquad \text{on } 
\calR.
\]
Here $c^4_{\lambda,m}$ vanishes of second order on $\calR_2$, so $h$ 
vanishes of second order when $\xi=\eta$. Then we can smoothly divide, 
\[
h(\xi,\eta) = h_1(\xi,\eta)(\xi-\eta)^2
\]
with $h_1$ still smooth and symmetric.
On $\calR$ we have  
\[
(\xi-\eta)^2 = -[ (\xi_1-\xi_2)(\xi_3-\xi_4)+ (\xi_1-\xi_4)(\xi_3-\xi_2)],
\]
which easily extends, so it remains to extend $h_1$ from $\calR$ to $\R^4$.
With a linear change of variable we have
\[
h_1(\xi,\eta) = h_2(\xi+\eta,\xi-\eta),
\]
where $h_2$ is smooth and even in the second variable. Then we can further write
\[
h_1(\xi,\eta) = h_3(\xi+\eta,(\xi-\eta)^2).
\]
But this last expression clearly admits a smooth extension, e.g.
\[
q^4_{\lambda,m}:= h_3\left(\frac12(\xi_1+\xi_2+\xi_3+\xi_4),\frac12((\xi_1-\xi_3)^2+(\xi_2-\xi_4)^2)\right).
\]
Now we subtract the contribution of $q^4_{\lambda,m}$ from $c^4_{\lambda,m}$, which reduces the problem to the case when $c^4_{\lambda,m}$ is zero on $\calR$, as desired.
\end{proof}

%%%%%%%%%%%%%%%%%%%%%%%%%%%%%%%%%%%%%%%%%%%%%%%%%%%%%%%%%%%%%%%%%%%%%%%%%%%%%%%%%%%%%%%%%%%%%%%%%%%%%%%%%%%%%%%%%%%%%%%%%%%%%%%%%%%%%%%%%%%%%%%%%%%%%%%%%%%

\section{Interaction Morawetz bounds for the linear paradifferential flow}
\label{s:para}

In this and the next section we study the linear paradifferential equation
associated to a solution $u$ to the quasilinear Schr\"odinger flow \eqref{qnls}. This is a key step in the proof of the local well-posedness
result in Theorem~\ref{t:local}, which is carried out in Section~\ref{s:lwp}. Here we prove bilinear $L^2_{t,x}$ bounds
 for the linear paradifferential equation. In the next section we prove Strichartz estimates for the linear paradifferential equation.

To start with, we consider a one parameter family of functions $v_\lambda$ which are localized at frequency $\lambda$ and solve the linear paradifferential equations \eqref{para} with source terms $f_\lambda$
in a time interval $[0,T]$. Due to the frequency localization, both $v_\lambda$ and $f_\lambda$ are smooth in $x$. In order to insure that 
the equations \eqref{para} are meaningful, we make the qualitative assumptions
\[
v_\lambda(0) \in L^2_x, \qquad f_\lambda \in L^2_{x,t}\left([0,T] \times \R\right),
\]
which in turn, by Gronwall's inequality, imply that 
\[
v_\lambda \in C(0,T;L^2_x),
\]
though with weak quantitative bounds which are allowed to depend on $\lambda$.

For now we assume no connection between these functions. Later we will apply these bounds in the case when  $v_\lambda=P_\lambda u$ or $v_\lambda = P_\lambda v$ where $v$ solves the linearized equation.

We will measure each $v_\lambda$ based on the size of the initial data
and of the source term $f_\lambda$. But rather than measuring $f_\lambda$ 
directly, we will measure its interaction with $v_\lambda$ and its translates. Precisely, we denote
\begin{equation}\label{dl}
d_\lambda^2 := \| v_\lambda(0)\|_{L^2_x}^2 + \sup_{x_0 \in \R} \| v_\lambda f_\lambda^{x_0} \|_{L^1_{t,x}}.
\end{equation}
While this might seem somewhat indirect, it will allow us later 
to apply this result in multiple 
settings and on different time-scales, with minimal 
assumptions on the source terms.

Our main result here asserts that we can obtain energy and bilinear $L^2_{t,x}$
bounds for $v_\lambda$:

\begin{theorem}\label{t:para}
Let $s > 1$. Assume that $u$ solves \eqref{qnls} in a time interval $[0,T]$, with $T$ as in \eqref{small-t-loc}, and satisfies the bounds \eqref{uk-ee},\eqref{uab-bi-bal}, \eqref{uab-bi-unbal} and \eqref{uk-se-short}. 
Assume that $v_\lambda$ is localized at frequency $\lambda$ and solves \eqref{para}, and let $d_\lambda$ be as in \eqref{dl}. Then the following bounds hold for the functions $v_\lambda$ in $[0,T]$, uniformly in $x_0 \in \R$:

\begin{enumerate}
    \item Uniform energy bounds:
\begin{equation}\label{v-ee}
\| v_\lambda \|_{L^\infty_tL^2_x} \lesssim d_\lambda .
\end{equation}

\item Balanced bilinear $(u,v)$-$L^2$ bound:
\begin{equation} \label{uv-bi-bal}
\| \partial_x(v_\lambda \bu_\mu^{x_0})  \|_{L^2_{t,x}} \lesssim \epsilon d_{\lambda} c_\mu \lambda^{-s} \lambda^{-\frac12} (1+ \lambda |x_0|),
\qquad \mu \approx \lambda.
\end{equation}

 \item Unbalanced bilinear $(u,v)$-$L^2$ bound:
\begin{equation} \label{uv-bi-unbal}
\| \partial_x(v_\lambda \bu_\mu^{x_0})  \|_{L^2_{t,x}} \lesssim \epsilon d_{\lambda} c_\mu \mu^{-s} (\lambda+\mu)^{\frac12} 
 \qquad \mu \not \approx \lambda .
\end{equation}

\item Balanced bilinear $(v,v)$-$L^2$ bound:
\begin{equation} \label{vv-bi-bal}
\| \partial_x(v_\lambda \bv_\mu^{x_0})  \|_{L^2_{t,x}} \lesssim d_{\lambda} d_\mu  \lambda^{-\frac12} (1+ \lambda |x_0|),
\qquad \mu \approx \lambda.
\end{equation}

 \item Unbalanced bilinear $(v,v)$-$L^2$ bound:
\begin{equation} \label{vv-bi-unbal}
\| \partial_x(v_\lambda \bv_\mu^{x_0})  \|_{L^2_{t,x}} \lesssim  d_{\lambda} d_\mu  (\lambda+\mu)^{\frac12} 
 \qquad \mu \not \approx \lambda.
\end{equation}
\end{enumerate}
\end{theorem}

Before proving this result, we 
have two remarks on $d_\lambda$:
\begin{itemize}
\item  Here the parameters $d_\lambda$ are independent, and we could freely take $d_\lambda = 1$ for the proof. But later we will think 
of $d_\lambda$ as a frequency envelope
for $L^2$ solutions to the linearized equation.

\item Another use of $d_\lambda$ is in the context of the solution $u$ itself, where we take $d_\lambda = \lambda^{-s} c_\lambda$.
\end{itemize}

\begin{proof}
We first note that in order to prove \eqref{v-ee},
\eqref{uv-bi-bal} and \eqref{uv-bi-unbal} it suffices 
to work with a fixed $v_\lambda$, while for the bilinear
$v$ bounds we need to work with exactly two $v_\lambda$'s.
To prove the theorem it is convenient to make the following bootstrap assumptions in $[0,T]$, with a large universal constant $C$:

\begin{enumerate}
    \item Uniform energy bounds:
\begin{equation}\label{v-ee-boot}
\| v_\lambda \|_{L^\infty _tL^2_x} \leq C d_\lambda.
\end{equation}

\item Balanced bilinear $(u,v)$-$L^2$ bound:
\begin{equation} \label{uv-bi-bal-boot}
\| \partial_x(v_\lambda \bu_\mu^{x_0})  \|_{L^2_{t,x}} \leq C \epsilon d_{\lambda} c_\mu \lambda^{-s} \lambda^{-\frac12} (1+ \lambda |x_0|),
\qquad \mu \approx \lambda.
\end{equation}

 \item Unbalanced bilinear $(u,v)$-$L^2$ bound:
\begin{equation} \label{uv-bi-unbal-boot}
\| \partial_x(v_\lambda \bu_\mu^{x_0})  \|_{L^2_{t,x}} \leq C \epsilon d_{\lambda} c_\mu \mu^{-s} (\lambda+\mu)^{-\frac12} 
 \qquad \mu \ll  \lambda. 
 \end{equation}

\item Balanced bilinear $(v,v)$-$L^2$ bound:
\begin{equation} \label{vv-bi-bal-boot}
\| \partial_x(v_\lambda \bv_\mu^{x_0})  \|_{L^2_{t,x}} \leq   C^2\epsilon d_{\lambda} d_\mu \lambda^{-s} \lambda^{-\frac12} (1+ \lambda |x_0|),
\qquad \mu \approx \lambda.
\end{equation}

 \item Unbalanced bilinear $(v,v)$-$L^2$ bound:
\begin{equation} \label{vv-bi-unbal-boot}
\| \partial_x(v_\lambda \bv_\mu^{x_0})  \|_{L^2_{t,x}} \leq C^2 \epsilon d_{\lambda} d_\mu \mu^{-s} (\lambda+\mu)^{-\frac12} 
 \qquad \mu \not \approx \lambda. 
\end{equation}
\end{enumerate}
We note that when the dyadic indices for the $v$'s are fixed, these represent
finitely many bootstrap assumptions, though 
we need to be careful about the $x_0$ dependence.

The objective will then be to prove that the bounds \eqref{v-ee}-\eqref{vv-bi-unbal} hold with implicit constants which do not depend on either $C$ or $\epsilon$. To insure that, in the proofs we will verify that the constant $C$ always appears together with extra powers of $\epsilon$. Then the independence of the implicit constants on $C$ is guaranteed by making an appropriate smallness
assumption on $\epsilon$, namely 
\[
\epsilon \ll_C 1.
\]
To avoid a circular argument, here we first choose $C$ larger than the universal implicit constant in \eqref{v-ee}-\eqref{vv-bi-unbal}, and then 
impose the above smallness assumption on $\epsilon$.

To close the bootstrap, we use a straightforward continuity argument. 
We choose $T_0 \in [0,T]$ maximal so that the  bootstrap bounds \eqref{v-ee-boot}-\eqref{vv-bi-unbal-boot} hold in $[0,T_0]$.
Then our bootstrap bounds imply that \eqref{v-ee}-\eqref{vv-bi-unbal}
hold in $[0,T_0]$. But our qualitative assumptions on $v_\lambda$ and $f_\lambda$ guarantee that the norms estimated in \eqref{v-ee}-\eqref{vv-bi-unbal} in $[0,T_0]$ are continuous as a function of $T_0$; this is also uniform in $x_0$ since we only use a finite range of dyadic frequencies for $u$.  Hence, if $T_0 < T$ then this implies that the bootstrap bounds \eqref{v-ee-boot}-\eqref{vv-bi-unbal-boot} hold in a slightly larger interval $[0,T_0+\delta]$, contradicting the maximality of $T_0$.

\medskip

We now proceed to prove \eqref{v-ee}-\eqref{vv-bi-unbal}
under the bootstrap assumptions \eqref{v-ee-boot}-\eqref{vv-bi-unbal-boot}.
The energy bound \eqref{v-ee} for $v_\lambda$ is straightforward, and follows 
directly by integrating \eqref{dens-flux-param}.

In order to treat the
dyadic bilinear $(u,v)$ and $(v,v)$ bounds at the same time,
we write the equation for $u_\lambda$ in the form 
\eqref{para-full} where we estimate favourably the source term $N_\lambda$. Precisely, given  the bounds \eqref{good-nl} and \eqref{res-nl}, we can interpret the functions $u_\lambda$ as solutions for the paradifferential equation \eqref{para}, with $d_\lambda$ replaced by $C \epsilon \lambda^{-s} c_\lambda$. Then the $(u,v)$ bilinear bounds 
and the $(v,v)$ bilinear bounds are absolutely identical. So from here on we focus on 
\eqref{vv-bi-bal} and \eqref{vv-bi-unbal}.

\bigskip

To prove  \eqref{vv-bi-bal} and \eqref{vv-bi-unbal} we will use the conservation laws
for the mass and momentum for $v_\lambda$ written in density-flux form, namely the identities \eqref{dens-flux-param} and \eqref{dens-flux-parap}. Next we use the definition of $d_\lambda$ in \eqref{dl}  in order to obtain $L^1_{t,x}$ bounds for the sources $F^{para}_{\lambda,m}$
and $F^{para}_{\lambda,p}$ in the density-flux energy identities for $v_\lambda$:

\begin{lemma}\label{l:F-para}
Assume that \eqref{uv-bi-unbal-boot} holds. Then we have
\begin{equation}
\| F^{para}_{\lambda,m}\|_{L^1_{t,x}} \lesssim C^2  d_\lambda^2,
\qquad 
\|  F^{para}_{\lambda,p}\|_{L^1_{t,x}} \lesssim  \lambda C^2  d_\lambda^2.
\end{equation}
\end{lemma}
The proof is immediate and is omitted.

\bigskip

To prove the bilinear $L^2_{t,x}$ bounds we use interaction Morawetz identities.
For each two $v_\lambda$'s we define the interaction Morawetz functional 
\begin{equation}\label{Ia-sharp-def-vv}
\bI^{x_0}(v_\lambda,v_\mu) :=   \iint_{x > y} M(v_\lambda)(x) P(v_\mu^{x_0}) (y) -  
P(v_\lambda)(x) M(v_\mu^{x_0}) (y) \, dx dy.
\end{equation}

To compute the time derivative of $\bI^{x_0}(v_\lambda,v_\mu)$ 
we use the density-flux identities \eqref{dens-flux-param} and \eqref{dens-flux-parap} and integrate by parts. This yields
\begin{equation}\label{I-vv}
\frac{d}{dt} \bI^{x_0}(v_\lambda,v_\mu) =  \bJ^4(v_\lambda,v_\mu^{x_0}) +  \bK(v_\lambda,v_\mu^{x_0}), 
\end{equation}
where
\[
\bJ^4(v_\lambda,v_\mu^{x_0}) := \int  g_{[<\mu]}^{x_0}\left(M(v_\lambda) E(v_\mu^{x_0}) - P(v_\lambda) P(v_\mu^{x_0})\right)
+ g_{[<\lambda]}
\left(M(v_\lambda) E(v_\mu^{x_0})
-  P(v_\lambda) P(v_\mu^{x_0})\right)\, dx,
\]
respectively
\begin{equation}\label{K8-def-ab-vv}
\begin{aligned}
\bK (v_\lambda,v_\mu^{x_0}):= \iint_{x > y} & \ M(v_\lambda)(x)   F^{para,x_0}_{\lambda,p}(y) - P(v_\lambda)(y)  F^{para,x_0}_{\lambda,m}(y) 
\\ & \!\!\!\!\!\! - 
M(v_\mu^{x_0})(y)  F^{para}_{\lambda,p}
(x)  - P(v_\mu^{x_0})(y)  F^{para}_{\lambda,m}(x)\,
dx dy.
\end{aligned}
\end{equation}
We will use the (integrated) interaction Morawetz identity \eqref{I-vv} to produce  $L^2_{t,x}$ bilinear bounds as follows:
\begin{itemize}
    \item The spacetime term $\bJ^4$ contains at leading order the squared bilinear $L^2_{t,x}$ norm we aim to bound.
    \item The fixed-time expression $\bI$ gives the primary bound for 
    $\bJ^4$.
    \item The space-time term $\bK$ will be estimated perturbatively. 
\end{itemize}

\bigskip

\textbf{I. The bound for $\bI^{x_0}(v_\lambda,v_\mu)$.}
Here we observe that we have the straightforward bound
\begin{equation} \label{I-lm}
 |    \bI^{x_0}(v_\lambda,v_\mu) | \lesssim (\lambda+\mu) d_\lambda^2 d_\mu^2,
\end{equation}
as a direct consequence of the energy bound \eqref{v-ee}.
\bigskip

\textbf{II. The contribution of $\bJ^4$.}
Here we will show that
\begin{lemma}\label{l:J4}
Under our assumptions on $u$ in Theorem~\ref{t:para} and our bilinear bootstrap assumptions 
\eqref{v-ee-boot}-\eqref{vv-bi-unbal-boot} we have 
\begin{equation}\label{j4-unbal}
\int_0^T \bJ^4(v_\lambda,v_\mu^{x_0})\, dt \approx \| \partial_x (v_\lambda \bv_\mu^{x_0})\|_{L^2_{t,x}}^2
+ O(C^6 \epsilon^2 \lambda d_\lambda^2 d_\mu^2), \qquad \mu \ll \lambda, 
\end{equation}
respectively
\begin{equation}\label{j4-bal}
\int_0^T \bJ^4(v_\lambda,v_\mu^{x_0})\, dt \approx \| \partial_x (v_\lambda \bv_\mu^{x_0})\|_{L^2_{t,x}}^2
+ O(C^6 \epsilon^2 \lambda d_\lambda^2 d_\mu^2(1+\lambda|x_0|)), \qquad \mu \approx \lambda.
\end{equation}

\end{lemma}
For later use we remark that the proof does not use at all the Strichartz bootstrap assumption for $u$.
\begin{proof}
We will separate the proof into two subcases, depending on whether
$\lambda$ and $\mu$ are balanced or not. We begin with the more difficult 
balanced case.
\medskip

\textbf{ II(a) The balanced case $\mu \approx \lambda$.} 
We split $J^4(v_\lambda,v_\mu^{x_0})$ it into 
\[
\begin{aligned}
J^4(v_\lambda,v_\mu^{x_0}) = & \  g_{[<\lambda]} J^4_{main}(v_\lambda,v_\mu^{x_0})
+ (g_{[<\mu]}^{x_0}- g_{[<\lambda]}) 
\left(M(v_\mu^{x_0}) E(v_\lambda)
-  P(v_\lambda) P(v_\mu^{x_0})\right)
\\
: = & \  
J^{4,a}(v_\lambda,v_\mu^{x_0})+
J^{4,b}(v_\lambda,v_\mu^{x_0}),
\end{aligned}
\]
where the leading expression $J^4_{main}$ has the form
\[
J^4_{main}(u,v) = M(u) E(v) 
+ M(v) E(u)
- 2 P(u) P(v).
\]
This is  a quadrilinear form in $(u,\bu,v,\bv)$ with symbol
\[
\begin{aligned}
(\xi_1+\xi_2)^2 + (\xi_3 +\xi_4)^2 - 2(\xi_1+\xi_2)(\xi_3+\xi_4)
= & \ (\xi_1+\xi_2-\xi_3-\xi_4)^2
\\ = & \ 4( \xi_1 - \xi_4)(\xi_2-\xi_3) + (\Delta^4 \xi)^2.
\end{aligned}
\]
This relation leads to the identity
\begin{equation}\label{J4-main}
J^4_{main}(u,v) = 4 | \partial_x (u \bar v)|^2 + 
\partial_x^2 (|u|^2 |v|^2),
\end{equation}
which in our case yields
\[
\int_0^T \bJ^{4,a}(v_\lambda,v_\mu^{x_0})\, dt = 
4\| g_{[<\lambda]}^\frac12
\partial_x (v_\lambda \bar v_\mu^{x_0})\|_{L^2}^2 + Err,
\]
where, integrating by parts, we have 
\[
Err := \int_0^T \int_\R \partial_x^2  g_{[<\lambda]} |v_\lambda|^2
|v_\mu^{x_0}|^2 \,dx dt.
\]
The key feature of the above expression is that $g$ is at least quadratic
in $u_{\ll \lambda}$, which allows us to use two bilinear $L^2_{t,x}$ bounds
pairing $u_{\ll \lambda}$ and $v_\lambda$. Precisely, neglecting complex conjugates of $u_{\ll \lambda}$ we can write
\[
\begin{aligned}
\partial_x^2  g_{[<\lambda]} |v_\lambda|^2 = & \  L(\partial_x^2 g(u_{\ll \lambda}), |v_\lambda|^2) 
\\ = & \  L( h(u_{\ll \lambda}),\partial_x u_{\ll \lambda},\partial_x u_{\ll \lambda},v_\lambda,v_\lambda )+ L( h(u_{\ll \lambda}),u_{\ll \lambda},\partial_x^2 u_{\ll \lambda},v_\lambda,v_\lambda ).
\end{aligned}
\]
Then, denoting by $\mu_1$ and $\mu_2$ the dyadic frequencies 
of the two $u_{\ll \lambda}$ entries, we can estimate $Err$ using two bilinear $L^2_{t,x}$ bootstrap bounds \eqref{uv-bi-unbal-boot} for the pairs
$(u_{\mu_1},v_\lambda)$ and $(u_{\mu_2},v_\lambda)$, combined 
with an $L^\infty$ bound for $v_\mu^{x_0}$ obtained from \eqref{v-ee} using Bernstein's inequality:
\[
|Err|\lesssim \epsilon^2 C^6 \sum_{\mu_1,\mu_2 \ll \lambda}
 (\mu_1+\mu_2)^2 \lambda^{-1}  \mu_1^{-s} \mu_2^{-s} d_\lambda^2
\cdot \mu  d_\mu^2 \lesssim \epsilon^2 C^6 \mu \, d_\lambda^2 \, d_\mu^2,
\]
as needed. For the last step we have used $ s \geq 1$ in order to handle the $\mu_1$ and $\mu_2$ summation.
\bigskip

We now consider the second term $J^{4,b}$. 
We split the difference between metrics as
\begin{equation}\label{dg-x0}
\begin{aligned}
g_{[<\lambda]} - g_{[<\mu]}^{x_0}
= & \ g_{[<\lambda]} - g_{[<\mu]} + g_{[<\mu]} - g_{[<\mu]}^{x_0}
\\ = &\
\ P_{<\lambda} (g(u_{\ll \lambda}) - g(u_{\ll \mu}))
+  (P_{<\lambda}-P_{<\mu}) g(u_{\ll \mu})+ P_{< \mu} (g(u_{\ll \mu}) - g(u_{\ll \mu}^{x_0})).
\end{aligned}
\end{equation}

The first term has contributions arising from the difference 
$u_{\ll\lambda} - u_{\ll \mu}$ which is just below frequency $\lambda$.
Since we are in the case $\mu \approx \lambda$ and $g$ is at least quadratic,  we represent this schematically as
\[
g(u_{\ll \lambda}) - g(u_{\ll \mu}) =  g'(u_{\ll\lambda}) u_{\approx c \lambda} =  L( h(u_{\ll \lambda}), u_{\ll \lambda}, u_{\approx c \lambda}) 
\]
with a small universal constant $c \ll 1$.
Then we can use two bilinear unbalanced bootstrap bounds \eqref{uv-bi-unbal-boot} for the pairs $(u_{\ll \lambda},v_\lambda)$, respectively $(u_{\approx c\lambda},v_\lambda)$ and $L^\infty$ Bernstein bounds for $v_\mu$ to estimate
\[
\left|\iint P_{<\lambda} (g(u_{\ll \lambda}) - g(u_{\ll \mu})) \left(M(v_\mu^{x_0}) E(v_\lambda)
-  P(v_\lambda) P(v_\mu^{x_0})\right) dx dt\right| 
\lesssim \epsilon^2 C^2 \lambda^{2-s}  d_\lambda^2 d_\mu^2, 
\]
which suffices.
\medskip

For the second difference in \eqref{dg-x0} we write
\[
(P_{<\lambda}-P_{<\mu}) g(u_{\ll \mu}) = P_{\approx \lambda} 
g(u_{\ll \mu}) = \lambda^{-1} \partial_x \tilde P_{\approx \lambda} 
g(u_{\ll \mu}),
\]
where the multiplier $\tilde P_{\approx \lambda}$ has symbol size, regularity and localization similar to $P_{\approx \lambda}$.

Then we can integrate by parts in the corresponding integrand in $J^{4,b}$,
where the derivative can fall on $v_\lambda$ or on $v_\mu$ terms,  yielding an expression of the form
\[
\begin{aligned}
\iint (P_{<\lambda}-P_{<\mu}) g(u_{\ll \mu})
\left(M(v_\mu^{x_0}) E(v_\lambda)
-  P(v_\lambda) P(v_\mu^{x_0})\right) \, dxdt = 
\\ \iint 
\lambda  \tilde P_{\approx \lambda} 
g(u_{\ll \mu}) ( \partial_x L(v_\lambda,\bv_\lambda) L(v_\mu,\bv_\mu)
+  L(v_\lambda,\bv_\lambda) \partial_x L(v_\mu,\bv_\mu))\, dxdt.
\end{aligned}
\]
Here we bound the integrand in $L^1_{t,x}$ using \eqref{Moser-tail} for the first factor, the balanced bootstrap bound \eqref{vv-bi-bal-boot} for the differentiated $v$ factor and an $L^\infty$ Bernstein bound for the undifferentiated $v$ factor. This yields
\[
\left|\iint (P_{<\lambda}-P_{<\mu}) g(u_{\ll \mu})
\left(M(v_\mu^{x_0}) E(v_\lambda)
-  P(v_\lambda) P(v_\mu^{x_0})\right) \, dxdt\right | \lesssim 
\epsilon^5 C^4 \lambda^{-3s+\frac32} d_\lambda^2 d_\mu^2 ,
\]
which suffices. We remark that, alternatively, here one can avoid the integration by parts and the use of the balanced bootstrap bound \eqref{vv-bi-bal-boot} by improving \eqref{Moser-tail} to an $L^1_{t,x}$ bound.
\medskip

Finally we consider the third difference in \eqref{dg-x0}, which, using the fact that $g$ is at least quadratic,  we expand as
\[
g_{[< \mu]} - g_{[< \mu]}^{x_0} = \int_0^{x_0} \partial_x  g_{[<\mu]}^y \, dy = \int_0^{x_0} P_{<\mu} (g'(u_{\ll \mu}^y) \partial_x u_{\ll\mu}^y)dy
= |x_0| L(h(u_{\ll \mu}), u_{\ll \mu}, \partial_x u_{\ll\mu}).
\]
Now we are able to estimate the corresponding integrand in $J^{4,b}$ using  two bilinear bootstrap  bounds \eqref{uv-bi-unbal-boot} estimates, where we pair $v_\mu$ once with 
$u_{\ll\mu}$ and secondly with $\partial_x u_{\ll\mu}$. This gives
\[
\begin{aligned}
\left| \iint  (g_{[<\mu]} - g_{[<\mu]}^{x_0}) \left(M(v_\mu^{x_0}) E(v_\lambda)
-  P(v_\lambda) P(v_\mu^{x_0})\right) dx dt \right|
\lesssim & \ \epsilon^2 C^2 \lambda^2 |x_0|  
\, d_{\lambda}^2\, d_{\mu}^2,
\end{aligned}
\]
which again suffices. We note that this is the place in the proof 
where  we need to take into account the difference between the original and the translated metric.
\medskip

\textbf{ II(a) The unbalanced case $\mu \ll \lambda$.} 
The argument here is similar, but with a major simplification
in the analysis of $J^{4,b}$. Precisely, 
if $\mu \ll \lambda$ 
then we have 
\[
\| v_\lambda \bar v_\mu^{x_0}\|_{L^2_{t,x}} \lesssim \lambda^{-1}
\| \partial_x (v_\lambda \bar v_\mu^{x_0})\|_{L^2_{t,x}},
\]
and, using the unbalanced bilinear bootstrap bound \eqref{vv-bi-unbal-boot} twice, it suffices to estimate the difference of $g$'s uniformly,
\[
\| g_{[<\lambda]} - g_{[<\mu]}^{x_0}\|_{L^\infty_{t,x}} \lesssim \epsilon^2 C^2.
\]
This concludes the proof of Lemma~\ref{l:J4}.

\end{proof}

\bigskip

\textbf{III. The contribution of $\bK$.} Here we estimate directly
\begin{equation*}
\begin{aligned}
\int_0^T |\bK (v_\lambda,v_\mu^{x_0})|\, dt \lesssim & \  \|M(v_\lambda)(x)\|_{L^\infty_t L^1_x}    \|F^{para,x_0}_{\mu,p}\|_{L^1_{t,x}} +  \|P(v_\lambda)\|_{L^\infty_t L^1_x} \|F^{para,x_0}_{\mu,m}\|_{L^1_{t,x}} 
\\ &  + 
\|M(v_\mu^{x_0})\|_{L^1_{t,x}}  \|F^{para}_{\lambda,p}\|_{L^1_{t,x}}
+ \|P(v_\mu^{x_0})\|_{L^\infty_t L^1_x}  \|F^{para}_{\lambda,m}\|_{L^1_{t,x}},
\end{aligned}
\end{equation*}
and then use the energy bounds for mass and momentum 
together with Lemma~\ref{l:F-para} for the source terms.
This yields
\begin{equation}\label{K-lm}
   \int_0^T |\bK (v_\lambda,v_\mu^{x_0})|\, dt \lesssim \lambda \, d_\lambda^2\, d_\mu^2.
\end{equation}

\bigskip

\textbf{ IV. The frequency localized bilinear $L^2$ bounds.}
Combining the bounds 
\eqref{I-lm}, \eqref{j4-bal} and \eqref{K-lm}
for $\bI$, $\bJ^4$ and $\bK$ we obtain \eqref{vv-bi-bal}. Similarly, \eqref{vv-bi-unbal} follows by combining he bounds 
\eqref{I-lm}, \eqref{j4-unbal} and \eqref{K-lm}. By replacing one of the $v_\lambda$'s  with $u_\lambda$ 
these also imply 
\eqref{uv-bi-bal} and \eqref{uv-bi-unbal}. 
\end{proof}

%%%%%%%%%%%%%%%%%%%%%%%%%%%%%%%%%%%%%%%%%%%%%%%%%%%%%%%

\section{Low regularity Strichartz for the paradifferential flow} \label{s:str}

%%%%%%%%%%%%%%%%%%%%%%%%%%%%%%%%%%%%%%%%%%%%%%%%%%%%%

Bilinear $L^2_{t,x}$ bounds suffice in order to understand the 
linear paradifferential flow \eqref{para}, where only unbalanced interactions occur between $u$ and $v$. However, in order to 
study the full linearized equation we also need to be able 
to estimate balanced interactions. The key to doing that at low regularity 
($s >1$ for \eqref{qnls}) is to have (nearly) lossless Strichartz estimates
for the linear paradifferential flow.

For convenience we recall the linear paradifferential equation:
\begin{equation}\label{para-re}
i \partial_t v_\lambda + \partial_x g_{[<\lambda]} \partial_x v_\lambda 
= f_\lambda, \qquad v_\lambda(0) := v_{0,\lambda}.
\end{equation}
The simplest Strichartz estimate one could write for this problem would be 
\begin{equation}\label{se-simple}
 \| v_\lambda \|_{L^\infty_t L^2_x \cap L^6_{t,x}} \lesssim \| v_{\lambda,0}\|_{L^2_x}
 + \|f_\lambda\|_{L^1_tL^2_x+L^{\frac65}_{t,x}}.
\end{equation}
While such an estimate will follow from our results in this section,
it is not robust enough, neither for our applications later on, 
nor for the proof. We point out two more specific closely related issues:

\begin{itemize}
    \item For the applications, the typical source term $f_\lambda$ will 
 a trilinear form which generally contains unbalanced interactions, and which, even using bilinear $L^2_{t,x}$ bounds, we cannot estimate in any $L^p$ norm without losing derivatives.
 \item For the proof, localizing both the source term $f_\lambda$ 
 and the initial data $v_\lambda$ at frequency $\lambda$ does not a-priori guarantee a similar localization for the solution. Furthermore, if we take the next obvious step and relocalize the solution at frequency $\lambda$,
 the commutators arising as truncation errors have a similar unbalanced 
 trilinear structure.
 \end{itemize}

For these reasons, we need to enlarge the space of allowed source terms.
At first, one might hope that working with source terms such as in \eqref{dl}
might suffice. While that has turned out to suffice for energy-type bounds,
the Strichartz estimates seem far less robust.  To bridge this gap,
we use a definition which involves duality. Precisely, for smooth, 
frequency localized functions $w_\lambda$ solving the inhomogeneous equation
\begin{equation}\label{dual}
\partial_t w_\lambda + \partial_x g_{[<\lambda]} \partial_x w_\lambda 
= e_\lambda, \qquad w_\lambda(T) = w_{T,\lambda},
\end{equation}
we define a homogeneous functional
\begin{equation}\label{X-lambda}
\|w_\lambda \|_{X_\lambda}^2 := \|w_\lambda(T)\|_{L^2_x}^2 + \sup_{x_0 \in \R} \| w^{x_0}_\lambda e_\lambda\|_{L^1_{t,x}},   
\end{equation}
which is akin to the $d_\lambda$'s in \eqref{dl}, and in particular allows us to use the bilinear estimates in Theorem~\ref{t:para}. Here it does not matter 
whether we measure the $L^2$ data at the initial or the final time;
we have chosen the time $T$ in order to emphasize the role of $w_\lambda$ 
as an adjoint variable. Also \eqref{X-lambda} is not a norm, and it does not even define a linear space, though one may define an associated norm  $X_\lambda^\sharp$ by 
\[
\| w_\lambda\|_{X_\lambda^\sharp} := \inf_{w_\lambda = \sum w_\lambda^j} 
\sum \|  w_\lambda^j\|_{X_\lambda},
\]
whose unit ball is the convex closure of the unit $X_\lambda$ ball.

Now we introduce a dual norm to be used for source terms $f_\lambda$ which are localized at frequency $\lambda$,
\begin{equation}
\| f_\lambda\|_{X_\lambda^*} := \sup_{\mu \approx \lambda} \sup_{\|w_\mu\|_{X_\mu} \leq 1} \| f_\lambda w_\mu\|_{L^1_{t,x}} .   
\end{equation}
With these notations, we are ready to state our main Strichartz estimate:

\begin{theorem}\label{t:para-se}
Assume that $u$ solves \eqref{qnls} in a time interval $[0,T]$ 
with $T$ as in \eqref{small-t-loc}, and satisfies the bounds \eqref{uk-ee},\eqref{uab-bi-bal}, \eqref{uab-bi-unbal} and \eqref{uk-se-short} for some $s \geq 1$. Then Strichartz estimates hold for frequency $\lambda$ localized solutions to the linear paradifferential equation \eqref{para-re},
\begin{equation}\label{para-se}
 \| v_\lambda \|_{L^\infty_t L^2_x \cap L^6_{t,x}} \lesssim \| v_{\lambda,0}\|_{L^2_x}
 + \|f_\lambda\|_{L^1_t L^2_x + L^{\frac65}_{t,x}+X_\lambda^*}.
\end{equation}
\end{theorem}

We note that the simpler Strichartz estimate \eqref{se-simple} is
a direct consequence of the theorem.
The rest of this section is devoted to the proof of this theorem.
In a first step, we successively replace equation
\eqref{para-re} with two carefully chosen equations with better properties. The first step will help with 
the frequency localization issue. This is
\begin{equation}\label{para-rea}
i \partial_t v_\lambda + \partial_x g_{[<\lambda^\sigma]} \partial_x v_\lambda 
= f_\lambda, \qquad v_\lambda(0) = v_{0,\lambda},
\end{equation}
where $\sigma < 1$ is close to $1$. 

The second step will be to add some ``good lower order terms'' in the above equation,
\begin{equation}\label{para-reb}
i \partial_t v_\lambda + \partial_x g_{[<\lambda^\sigma]} \partial_x v_\lambda + V_{<\lambda^\sigma} \partial_x v_\lambda + W_{<\lambda^\sigma}  v_\lambda 
= f_\lambda, \qquad v_\lambda(0) = v_{0,\lambda}.
\end{equation}
The potentials $V_{<\lambda^\sigma}$ and and $W_{<\lambda^\sigma} $
are assumed to be localized at  frequency $< \lambda^\sigma$.
Their form and role will become clear shortly. But for now what is important is that they satisfy good pointwise bounds which make their contributions perturbative. From the perspective of classical Strichartz estimates, it would be natural to place them into mixed norm spaces such as $L^1_t L^\infty_x$ or $L^{\frac32}_{t,x}$. That would make them directly perturbative, but it would be too restrictive in our case. Instead, we will assume they have a favourable bilinear structure,
\begin{equation}\label{which-V}
 |V_{<\lambda^\sigma}| \lesssim L(|\partial u_{<\lambda^\sigma}|,|\partial u_{<\lambda^\sigma}|) ,\qquad  |W_{<\lambda^\sigma}| \lesssim \epsilon^2  L(|\partial u_{<\lambda^\sigma}|,|\partial u_{<\lambda^\sigma}|).
\end{equation}

One may formulate corresponding versions of Theorem~\ref{t:para-se} 
associated to the equations \eqref{para-rea}, respectively \eqref{para-reb}.
In addition to changing the equation, this also entails changes to the 
space $X_\lambda$. The modified versions of $X_\lambda$, which we denote by 
$X_\lambda^1$ and $X_\lambda^2$, are defined exactly as in \eqref{X-lambda} but with respect to the inhomogeneous adjoint equations to  \eqref{para-rea}, respectively \eqref{para-reb}, which are 
\begin{equation}\label{para-duala}
i \partial_t w_\lambda + \partial_x g_{[<\lambda^\sigma]} \partial_x w_\lambda 
= e_\lambda^1, \qquad w_\lambda(T) = w_{T,\lambda},
\end{equation}
respectively
\begin{equation}\label{para-dualb}
i \partial_t w_\lambda + \partial_x g_{[<\lambda^\sigma]} \partial_x w_\lambda - \partial_x \bar V_{<\lambda^\sigma}  v_\lambda + \bar  W_{<\lambda^\sigma}  w_\lambda 
= e_\lambda^2, \qquad w_\lambda(T) = w_{T,\lambda}.
\end{equation}

We will now show that it is equivalent to use either of the three 
equations:

\begin{proposition}\label{p:equiv}
The estimate \eqref{para-se} for the equation \eqref{para-re} is equivalent to the similar bound for the equation \eqref{para-rea} or \eqref{para-reb}.
\end{proposition}

\begin{proof}
The differences in the three formulations of the theorem have two sources:
\begin{enumerate}[label=(\roman*)]
\item the changes in the $X_\lambda$ functional associated to the adjoint equation, and
\item the changes in the source term for a given solution $v$. 
\end{enumerate}
We will successively address these two issues.
\medskip

(i). The first step in the proof is to show that the $X_\lambda$ functionals associated to the three equations stay equivalent,
\begin{equation}\label{X-equiv}
\|w \|_{X_\lambda} \approx \|w\|_{X_\lambda^1} \approx \|w\|_{X_\lambda^2}.   \end{equation}

Suppose for instance that $w_\lambda$ solves the second adjoint equation
\eqref{para-duala}. Then it also solves \eqref{dual} with 
\[
e_\lambda = e_\lambda^1 + \partial_x (g_{[<\lambda^{\sigma}]}-g_{[<\lambda]}) \partial_x w_\lambda.
\]

This yields a difference in the associated $X_\lambda$ functionals which is given by
\[
\left| \|w\|_{X_\lambda}^2 - \|w\|_{X_\lambda^1}^2\right| \lesssim \sup_{x_0\in \R} \| w_\lambda^{x_0} \partial_x (g_{[<\lambda^{\sigma}]}-g_{[<\lambda]}) \partial_x w_\lambda \|_{L^1_{t,x}}.
\]
We estimate the $g$ difference as in the proof of Lemma~\ref{l:J4},
writing
\[
 w_\lambda^{x_0} \partial_x (g_{[<\lambda^{\sigma}]}-g_{[<\lambda]}) \partial_x w_\lambda = \lambda^2 L(h(u_{\ll \lambda},u_{\ll \lambda^\sigma}), w_\lambda,w_\lambda,u_{\ll\lambda},u_{\ll\lambda^\sigma}-u_{\ll \lambda}).
\]
Here we can apply twice the unbalanced $L^2_{t,x}$ bound \eqref{uv-bi-unbal-boot}
for the pairs $(w_\lambda,u_{\ll \lambda})$ respectively $(w_\lambda, u_{\ll\lambda^\sigma}-u_{\ll \lambda})$
in order to estimate the $L^1_{t,x}$ norm on the right, with $d_\lambda =  \|w\|_{X_\lambda}$. This yields
\begin{equation}\label{err1}
\| w_\lambda^{x_0} \partial_x (g_{[<\lambda^{\sigma}]}-g_{[<\lambda]}) \partial_x w_\lambda \|_{L^1_{t,x}} \lesssim \epsilon^2 \lambda^{1-\sigma s}\| w \|_{X_\lambda^2} \ll \| w \|_{X_\lambda^2},
\end{equation}
where we have chosen $\sigma < 1$ so that $\sigma s > 1$.  Hence
\begin{equation}\label{err2}
\left| \|w\|_{X_\lambda}^2 - \|w\|_{X_\lambda^1}^2\right| \ll \| w \|_{X_\lambda}^2,
\end{equation}
 from which we conclude that 
\[
\|w \|_{X_\lambda} \approx \|w\|_{X_\lambda^1}.
\]

A similar argument allows us to compare the $X_\lambda^1$ and $X_\lambda^2$ functionals for \eqref{para-rea} and \eqref{para-reb}. Here we need to estimate 
\[
\| w_\lambda^{x_0} (\partial_x \bar V_{<\lambda^\sigma}   + \bar  W_{<\lambda^\sigma}) w_\lambda\|_{L^1_{t,x}} \ll \|w\|_{X_\lambda}^2,
\]
which is obtained by using the bound  \eqref{which-V} and then applying twice 
the unbalanced $L^2_{t,x}$ bound \eqref{uv-bi-unbal-boot}. 

This concludes the proof of \eqref{X-equiv}, which in turn shows the equivalence of the dual norms $X_\lambda^*$, $X_\lambda^{1,*}$ and $X_\lambda^{2,*}$. At this point we can drop the 
superscripts $1$ and $2$ and work simply with $X_\lambda$ and $X_\lambda^*$.  \medskip

(ii) The second step is to compare source terms in \eqref{para-re}, \eqref{para-rea} and \eqref{para-reb}, and show that their difference
is perturbative, i.e. small in $X_\lambda^*$. As a preliminary step, we 
verify that the $X_\lambda$ norm of a function $v_\lambda$  solving \eqref{para-re} can be estimated in terms 
of the quantities in \eqref{para-se}, namely
\begin{equation}\label{vl-in-X}
\| v_\lambda \|_{X_\lambda} \lesssim \|v_{\lambda,0}\|_{L^2_x} + \|v_{\lambda}\|_{L^6_{t,x}}   + \|f_\lambda\|_{L^{\frac65}_{t,x} + X_\lambda^*}  .
\end{equation}
This bound is important in that it allows us to use Theorem~\ref{t:para} for
$v_\lambda$. To see this we decompose 
\[
f_\lambda = :f_\lambda^1 + f_\lambda^2, \qquad \|f_\lambda^1 \|_{L^{\frac65}_{t,x}}
+ \|f_\lambda^2\|_{X_\lambda^*} \lesssim \|f_\lambda\|_{L^{\frac65}_{t,x} + X_\lambda^*}.
\]
Then we estimate
\[
\begin{aligned}
\| v_\lambda\|^2_{X_\lambda} \lesssim & \   \|v_{\lambda,0}\|_{L^2_x}^2
+ \| v_\lambda^{x_0} f_\lambda^1\|_{L^1_{t,x}} + \| v_\lambda^{x_0} f_\lambda^2\|_{L^1_{t,x}}
\\
\lesssim  & \   \|v_{\lambda,0}\|_{L^2_{x}}^2
+ \| v_\lambda\|_{L^6_{t,x}} \|f_\lambda^1\|_{L^\frac65_{t,x}} + 
\| v_\lambda\|_{X_\lambda}  \|f_\lambda^2\|_{X_\lambda^*},
\end{aligned}
\]
after which \eqref{vl-in-X} follows by the Cauchy-Schwartz inequality.

We now turn our attention to the estimate for the difference between the source terms \eqref{para-re} and \eqref{para-rea}, which we seek to place into $X_\lambda^*$. This requires us to obtain the  $L^1_{t,x}$ bound
\[
\| w_\mu \partial_x (g_{[<\lambda^{\sigma}]}-g_{[<\lambda]}) \partial_x v_\lambda\|_{L^1_{t,x}} \ll \| w_\mu \|_{X_\mu} \|v_\lambda\|_{X_\lambda}, \qquad \mu \approx \lambda.
\]
But this is obtained in the same way as \eqref{err1}, by applying twice the unbalanced bilinear $L^2$ bound \eqref{uv-bi-unbal-boot}. Similarly, for the second difference we need to estimate 
\[
\| w_\mu (V_{<\lambda^\sigma} \partial_x  + W_{<\lambda^\sigma})  v_\lambda\|_{L^1_{t,x}}\ll \| w_\mu \|_{X_\mu} \|v_\lambda\|_{X_\lambda}, \qquad \mu \approx \lambda.
\]
This uses  \eqref{which-V}, which allows us once more apply twice the unbalanced bilinear $L^2_{t,x}$ bound \eqref{uv-bi-unbal-boot}, exactly as in the proof of \eqref{err2}.
This concludes the proof of the proposition.
\end{proof}

Now we show how to choose frequency localized potentials $V_{<\lambda^\sigma}$ and $W_{<\lambda^\sigma}$ so that the equation 
\eqref{para-reb} has good dispersive properties after a favourable change of coordinates. 
Our change of coordinates is chosen so that 
it flattens the principal part of the equation.
Such a strategy has been used before, see \cite{BP}, but only in a situation where the coefficients are time independent; there, the  change of coordinates completely solves the problem, as it produces first order terms which are shown to be perturbative via local smoothing estimates.
 In our case, the difficulty lies in the 
fact that the coefficients are inherently time 
dependent. Because of this, the change of coordinates merely shifts the problem, rather than solving it.

Our change of the spatial variable $x\rightarrow y$ is chosen so that
\[
\frac{\partial }{\partial x} =\frac{1}{\sqrt{ g_{[<\lambda^{\sigma}]}}}\frac{\partial}{\partial y},
\]
which gives
\[
\frac{\partial y}{\partial x}=\frac{1}{\sqrt{g_{[<\lambda^{\sigma}]}}}.
\]
To this we add the initialization $y(t,0) = 0$.
Such a transformation removes the nonlinear term in the principal part of \eqref{para-rea}; specifically, the second order operator in the equation 
becomes
\[
\partial_xg_{[<\lambda^{\sigma}]}\partial_x \rightarrow \partial^2_y.
\]
The time variable $t$ remains the same; for clarity we record the change of variable below
\begin{equation}
\label{change var}
\left\{
\begin{aligned}
&y=y(t,x)\\
&\tau=t.
\end{aligned}
\right.
\end{equation}
However, the expression for $\partial/\partial t$ will depend on $\partial/\partial {\tau}$, via the chain rule, as follows
\[
\frac{\partial}{\partial t} =\frac{\partial {\tau}}{\partial t} \cdot \frac{\partial}{\partial {\tau}}   + 
\frac{\partial y}{\partial t} \cdot \frac{\partial}{\partial y} .
\]
So, \eqref{para-rea} can now be  equivalently written as
\begin{equation}
\label{para-rea-var}
i\partial_{\tau} v_{\lambda}+ \left(i\frac{\partial y}{\partial t} +\partial_x \sqrt{g_{[<\lambda^{\sigma}]}}     \right)\partial_y v_{\lambda}+\partial^2_y v_{\lambda} = \tilde{f}_\lambda.
\end{equation}
Here $\tilde{f}_{\lambda}$  is $f_{\lambda}$ expressed in the new coordinates.

It remains to compute the expression $\partial y/\partial t$. For this we recall that 
\[
y=\int_{0}^{x} \frac{1}{\sqrt{g_{[<\lambda^{\sigma}]}}}\, dx',
\]
where $g$ depends on both $u$ and $\bar{u}$. Thus, via the  chain rule, the time derivative can be computed as follows,
\[
\begin{aligned}
\frac{\partial y}{\partial t}=&\int_{0}^x \partial_t  \frac{1}{\sqrt{g_{[<\lambda^{\sigma}]}}}\, dx' = \int_0^x  - \frac12 (g_{[<\lambda^{\sigma}]})^{-\frac32} P_{<\lambda^\sigma} \partial_t g(u_{\ll \lambda^{\sigma}}) \,dx' \\
= & \ \Re \int_0^x   (g_{[<\lambda^{\sigma}]})^{-\frac32} P_{<\lambda^\sigma} [h_1(u_{\ll \lambda^\sigma}) P_{\ll \lambda^\sigma}(g(u) \partial_x^2 u + h_1(u) (\partial_x u)^2)]\, dx'  ,
\end{aligned}
\]
where $h_1$ represents smooth functions vanishing at zero. We then integrate by parts to eliminate the second order derivative and to rewrite it in the form
\begin{equation}\label{vlong}
\begin{aligned}
\frac{\partial y}{\partial t} = & \ \left.  \Re  (g_{[<\lambda^{\sigma}]})^{-\frac32} P_{<\lambda^\sigma} [h_1(u_{\ll \lambda^\sigma}) P_{\ll \lambda^\sigma}(g(u) \partial_x u) ]\right|_0^x
\\
& \ +  \Re \int_0^x   (g_{[<\lambda^{\sigma}]})^{-\frac32} P_{<\lambda^\sigma} [h_1(u_{\ll \lambda^\sigma}) P_{\ll \lambda^\sigma}( h_1(u) (\partial_x u)^2)] \, dx'  
\\ & \ +  \Re \int_0^x   (g_{[<\lambda^{\sigma}]})^{-\frac52}P_{<\lambda^\sigma} [h_1(u_{\ll \lambda^\sigma}) \partial_x u_{\ll \lambda^\sigma}] P_{<\lambda^\sigma} [h_1(u_{\ll \lambda^\sigma}) P_{\ll \lambda^\sigma}( g(u) \partial_x u)] \,dx'  
\\ & \ + \Re \int_0^x   (g_{[<\lambda^{\sigma}]})^{-\frac32} P_{<\lambda^\sigma} [h_0(u_{\ll \lambda^\sigma}) \partial_x u_{\ll \lambda^\sigma}) P_{\ll \lambda^\sigma}( g(u) \partial_x u)] \,dx'  ,
\end{aligned}
\end{equation}
where $h_0$ represents smooth functions possibly not vanishing at zero.

Equation \eqref{para-rea-var} can now be reexpressed as
\begin{equation}
\label{para-rea-var-new}
i\partial_{\tau} v_{\lambda}+ \left(V_1+ iV_2    \right)\partial_y v_{\lambda}+\partial^2_y v_{\lambda} = \tilde{f}_\lambda,
\end{equation}
where $V_1$ and $V_2$ are complex valued, respectively real valued, and are chosen as follows:
\begin{itemize}
\item $V_1$ contains the term $\partial_x \sqrt{g_{[<\lambda^{\sigma}]}}$
in \eqref{para-rea-var}, as well as the $x$-dependent contribution 
on the first line of \eqref{vlong}.

\item $V_2$ contains the constant function on the first line of \eqref{vlong}, as well as the integral contributions on the last three
lines in \eqref{vlong}.
\end{itemize}
We will replace the equation \eqref{para-rea-var-new} by the truncated version
\begin{equation}
\label{para-rea-final}
i\partial_{\tau} v_{\lambda}+ iZ    \partial_y v_{\lambda}
+ \frac{i}2 \partial_y Z    v_{\lambda}
+\partial^2_y v_{\lambda} 
= \tilde{f}_\lambda, \qquad Z := P_{\lesssim \epsilon^2}^yV_2,
\end{equation}
where the zero order term was added simply in order to gain symmetry, and the spectral projectors are now taken in the $y$ coordinates. This equation  will play two roles:
\begin{itemize}
    \item We will prove Strichartz estimates for the flow \eqref{para-rea-final} in the $(\tau,y)$ coordinates, and
    \item This will be the equation  \eqref{para-reb} in the $(t,x)$ coordinates, after some mild  adjustments which are perturbative in a strict Strichartz fashion. 
\end{itemize}
We note that, reversing the above change of coordinates, it is easily seen that  in the $(t,x)$ coordinates the equation \eqref{para-rea-final}
takes the form 
\begin{equation}\label{para-reb-bis}
i \partial_t v_\lambda + \partial_x g_{[<\lambda^\sigma]} \partial_x v_\lambda + V \partial_x v_\lambda + W  v_\lambda 
= f_\lambda, \qquad v_\lambda(0) =: v_{0,\lambda},
\end{equation}
where the potentials $V$ and $W$ in \eqref{para-reb} will be given by
\[
V:=\sqrt{g_{<\lambda^\sigma}} (V_1+iP_{\gg\epsilon^2 }^yV_2), \qquad W :=\frac{i}2 \partial_y P_{\lesssim \epsilon^2}^yV_2.
\]
In order to achieve the objectives outlined above, we need to take a closer look at the potentials $V$, $W$ and $Z$ above:
\begin{proposition}\label{p:potentials}
Under the same assumptions as in Theorem~\ref{t:para-se}, the potentials 
$V$, $W$ and $Z$ have the following properties:

a) The potentials $V$ and $W$ admit decompositions  
\begin{equation}
V = V^b + V^g, \qquad W = W^b+ W^g    
\end{equation}
so that the following estimates hold:
\begin{equation}\label{which-Vg}
 |V^b| \lesssim L(|\partial u_{<\lambda^\sigma}|,|\partial u_{<\lambda^\sigma}|) ,\qquad  |W^b| \lesssim \epsilon^2  L(|\partial u_{<\lambda^\sigma}|,|\partial u_{<\lambda^\sigma}|),
\end{equation}
respectively 
\begin{equation}\label{which-Vb}
\begin{aligned}
 \|V^g\|_{L^1_t L^\infty_x + L^\frac32} + \|P_{\gg\lambda^\sigma} V^{b}\|_{L^1_t L^{\infty}_x + L^{\frac32}_{t,x}}   \lesssim 
 & \
 \epsilon^2 \lambda^{-1}, \\
 \|W^g\|_{L^1_t L^{\infty}_x + L^{\frac32}_{t,x}} + \|P_{\gg\lambda^\sigma} W^{b}\|_{L^1_{t} L^{\infty}_{x} + L^{\frac32}_{t,x}}   \lesssim
 & \ 
 \epsilon^2. 
 \end{aligned}
\end{equation}

b) The low frequency component $Z := P_{\lesssim \epsilon^2}^y V_2$ satisfies the 
uniform bound
\begin{equation}\label{Z}
\| Z \|_{L^{\infty}_x} \lesssim \epsilon^2.
\end{equation}
In addition, its derivative admits a decomposition
\begin{equation}
\label{Z-bg}
\partial_y Z = Z^b + Z^g , 
\end{equation}
where 
\begin{equation} \label{Zb}
|Z^b| \lesssim L(|\partial u_{<\lambda^\sigma}|,  |\partial u_{<\lambda^\sigma}|)  ,
\end{equation}
and 
\begin{equation}\label{Zg}
\|Z^g\|_{L^{1}_t L^{\infty}_x} \lesssim \epsilon^2 \lambda^{-1}  . 
\end{equation}
\end{proposition}

Given this proposition, it is clear that we should choose the potentials $V_{<\lambda^\sigma}$ and $W_{<\lambda^\sigma}$ in \eqref{para-reb} as 
\begin{equation}
 V_{<\lambda^\sigma} := P_{\lesssim\lambda^\sigma} V^b, \qquad     W_{<\lambda^\sigma} := P_{\lesssim\lambda^\sigma} W^b.
\end{equation}
Then the bound \eqref{which-Vb} insures that 
\begin{equation}\label{which-V-diff}
\lambda  \|V - V_{<\lambda^\sigma}\|_{L^1_t L^\infty_x + L^{\frac32}_{t,x}} + \|W - W_{<\lambda^\sigma}\|_{L^1 _tL^{\infty}_x + L^{\frac32}_{t,x}}   \lesssim 
 \epsilon^2, 
\end{equation}
which implies that these differences are directly perturbative in 
the Strichartz estimate \eqref{para-se}. Hence, it is equivalent to work with the equation \eqref{para-reb} or with \eqref{para-rea-final}.

\begin{proof}
a) We begin by considering  the contribution of $V_1$, which, 
after peeling off the constant term in $g$,
can be expanded as
\[
\begin{aligned}
V_1 = & \ (g_{[<\lambda^{\sigma}]})^{-\frac12} P_{< \lambda^\sigma}
[h_1(u_{\ll \lambda^\sigma}) \partial_x u_{\ll \lambda^\sigma}]
+ (g_{[<\lambda^{\sigma}]})^{-\frac32} P_{<\lambda^\sigma} [h_1(u_{\ll \lambda^\sigma})  \partial_x u_{\ll \lambda^\sigma}  ]
\\ 
& \ + (g_{[<\lambda^{\sigma}]})^{-\frac32} P_{<\lambda^\sigma} [h_1(u_{\ll \lambda^\sigma}) P_{\ll \lambda^\sigma} (h_2(u) \partial_x u_{< \lambda^\sigma})  ]
\\ & \ + (g_{[<\lambda^{\sigma}]})^{-\frac32} P_{<\lambda^\sigma} [h_1(u_{\ll \lambda^\sigma}) P_{\ll \lambda^\sigma} (h_2(u) \partial_x u_{> \lambda^\sigma})  ],
\end{aligned}
\]
where $h_1$ and $h_2$ represent smooth functions vanishing simply, respectively doubly at zero.
Taking absolute values and using the fact that change of coordinates is bilipschitz, the first three terms can be directly estimated by 
$L(|u_{<\lambda^\sigma}|, |\partial_x u_{<\lambda^\sigma}|)$, and thus
will be denoted by $V_1^b$ and will be included in $V^b$. This is no longer possible for the last term, which also contains some high frequency contributions. However, instead this term is both quartic and unbalanced,
so we will estimate it in $\epsilon^4 \lambda^{-1} L^\frac32_{t,x}$.  We can discard the $g$  factor and the first spectral projector, and denote by $\mu > 
\lambda^\sigma$ the frequency of the last factor. Then $h_2(u)$ must also be at frequency $\mu$, so we are left with an expression of the form
\[
L(u_{\ll \lambda^\sigma}, P_{\mu}h_2(u), \partial_x u_{\mu}).
\]
We pair the first and last entry and use an unbalanced  $L^2_{t,x}$ bound,
while the middle term is estimated in $L^6_{t,x}$. We get 
\[
\| L(u_{\ll \lambda^\sigma}, P_{\mu}h_2(u), \partial_x u_{\mu})\|_{L^{\frac32}_{t,x}} \lesssim  \epsilon^2 \mu^{-s+\frac12} \cdot \epsilon^2 \mu^{-s}
= \epsilon^4 \mu^{-2s +\frac12},
\]
which is better than needed after the $\mu$ summation.

We still need to estimate the high frequencies of $V_1^b$. We consider the first term in $V_1$, as the second and the third are similar. For this term we can write
\[
P_{\gg \lambda^\sigma} [ (g_{[<\lambda^{\sigma}]})^{-\frac12} P_{< \lambda^\sigma}
[h_1(u_{\ll \lambda^\sigma}) \partial_x u_{\ll \lambda^\sigma}]]= 
P_{\gg\lambda^\sigma} [ P_{\gg \lambda^\sigma} [(g_{[<\lambda^{\sigma}]})^{-\frac12}] P_{< \lambda^\sigma}
[h_1(u_{\ll \lambda^\sigma}) \partial_x u_{\ll \lambda^\sigma}]],
\]
which we can estimate in $L^\frac32_{t,x}$ by Lemma~\ref{l:Moser-lh} and an $L^6_{t,x}$ Strichartz bound,
\[
\lesssim \|P_{\gg \lambda^\sigma} [(g_{[<\lambda^{\sigma}]})^{-\frac12}]\|_{L^2_{t,x}} \| \partial_x u_{\ll \lambda^\sigma}]\|_{L^6_{t,x}}
\lesssim \epsilon^6 \lambda^{-3\sigma s},
\]
and which is stronger than we need.

\bigskip

We next consider the contribution of $V_2$ to $V$. 
The kernel for $\partial_y^{-1} P^y_{\gg \epsilon^{2}}$ is bounded and rapidly decreasing on the $\epsilon^{-2}$ scale, and thus has $L^1_{t,x}$ norm $\ll \epsilon^{-2}$. It remains to estimate the  integrands in \eqref{vlong}.
We distinguish two cases based on the frequency of the $\partial_x u$
factors:
\medskip

\emph{Case 1:} Both $\partial_x u$ frequencies are less than $\lambda^\sigma$. Then we can bound all the integrands in \eqref{vlong} by $L(|\partial_x u_{<\lambda^\sigma}|,|\partial_x u_{<\lambda^\sigma}|)$,
and include their contributions into $V^b$. To estimate the high frequencies in $V^b$ as in \eqref{which-Vb} it suffices to consider $\partial_x^2 V_b$ and show that
\begin{equation}\label{V2-high}
\|\partial_x^2  V^{b}\|_{L^1_t L^\infty_x + L^{\frac32}_{t,x}}   \lesssim  \epsilon^2 \lambda^{2\sigma-1} .
\end{equation}
The frequency localization of $\partial_x u$ is not needed here.
To show this, we simply convert $\partial_x = (g_{[<\lambda^\sigma]})^{-\frac12} \partial_y$, where the $y$ differentiation eliminates the integration in 
\eqref{vlong}. The first $y$ differentiation is then commuted past $P_{\lesssim \epsilon^2}^y$ in order to eliminate the integration, and the second is combined with $P_{\lesssim \epsilon^2}^y$ for an $\epsilon^2$ factor. The common computation is 
\[
\begin{aligned}
\partial_x^2  P_{\lesssim \epsilon^2}^y \int_0^x f(x') \, dx' = &\ 
((g_{[<\lambda^\sigma]})^{-\frac12} \partial_y)^2  P_{\lesssim \epsilon^2}^y \int_0^x f(x')\, dx'
\\ = & \ 
(g_{[<\lambda^\sigma]})^{-\frac12} \partial_y \left((g_{[<\lambda^\sigma]})^{-\frac12}   P_{\lesssim \epsilon^2}^y \partial_y \int_0^x f(x') \,dx'\right)
\\ = & \ 
(g_{[<\lambda^\sigma]})^{-\frac12} \partial_y \left((g_{[<\lambda^\sigma]})^{-\frac12}   P_{\lesssim \epsilon^2}^y \left[(g_{[<\lambda^\sigma]})^{\frac12} f\right]\right)
\\ = & \ 
(g_{[<\lambda^\sigma]})^{-1} \partial_y   P_{\lesssim \epsilon^2}^y [(g_{[<\lambda^\sigma]})^{\frac12} f] +  \left[\partial_x (g_{[<\lambda^\sigma]})^{-\frac12})\right]   P_{\lesssim \epsilon^2}^y \left[(g_{[<\lambda^\sigma]})^{\frac12} f\right].
\end{aligned}
\]
Here $f$ is any of the integrands in \eqref{vlong}, which contains two copies of $\partial_x u$. 

In the first term above, we simply use the energy bound for $\partial_x u$ to get $\|f\|_{L^\infty_t L^1_x} \lesssim \epsilon^2$.
Since $T \lesssim \epsilon^{-4}$, by H\"older's inequality in $t$ and Bernstein's inequality in $x$ we get
\[
\| \partial_y   P_{\lesssim \epsilon^2}^y [(g_{[<\lambda^\sigma]})^{\frac12} f] \|_{L^1_t L^\infty_x} \lesssim \epsilon^{-4} \|\partial_y   P_{\lesssim \epsilon^2}^y [(g_{[<\lambda^\sigma]})^{\frac12} f] \|_{L^\infty_t L^\infty_x}
\lesssim \| [(g_{[<\lambda^\sigma]})^{\frac12} f] \|_{L^\infty_t L^1_x}
\lesssim \epsilon^2,
\]
which is better than needed in \eqref{V2-high}.

In the second term above 
we have three copies of $\partial_x u$ and at least one copy of $u$, which are all bounded in $L^6_{t,x}$, obtaining a favorable $L^\frac32_{t,x}$ estimate with an $\epsilon^4$ factor.

\medskip

\emph{Case 2:} We have at least one copy of  $\partial_x [u_{> \lambda^\sigma}]$.  These terms are placed in $V^g$.
We consider each of the integrands on the last three lines in \eqref{vlong}, neglecting bounded factors.
\medskip

The contribution of line $2$  of \eqref{vlong} is estimated by 
\[
L(|u_{\ll \lambda^\sigma}|,|u|, |\partial_x u_{>\lambda^\sigma}|, |\partial_x u_{>\lambda^\sigma}|) + \sum_{\mu >\lambda^\sigma} L(|u_{\ll \lambda^\sigma}|, |\partial_x u_{\ll \lambda^\sigma}|, |\partial_x u_\mu|, |h_1(u)_\mu|).
\]
For the first term we distinguish two cases, depending on the frequency of $u$. 
In the low frequency case we use two unbalanced bilinear $L^2_{t,x}$ bounds,
\[
\| L(|u_{\ll \lambda^\sigma}|,|u_{\ll \lambda^{\sigma}}|, |\partial_x u_{>\lambda^\sigma}|, |\partial_x u_{>\lambda^\sigma}|)\|_{L^1_{t,x}}
\lesssim \epsilon^4 \lambda^{\sigma( 1-2s)},
\]
followed by Bernstein's inequality in $y$ to get to $L^1 _tL^\infty_x$. In the high-frequency case we use one unbalanced bilinear $L^2_{t,x}$ bound and one $L^6_{t,x}$ bound for the remaining $\partial_x u_{>\lambda^\sigma}$, and estimate the high frequencies of $u$ in $L^\infty_x$, to obtain
\[
\| L(|u_{\ll \lambda^\sigma}|,|u_{\gtrsim \lambda^{\sigma}}|, |\partial_x u_{>\lambda^\sigma}|, |\partial_x u_{>\lambda^\sigma}|)\|_{L^\frac32_{t,x}}
\lesssim \epsilon^4 \lambda^{\sigma( \frac12-2s)},
\]
which is satisfactory.
The second term is similar to the last case above,
\[
\|L(|u_{\ll \lambda^\sigma}|, |\partial_x u_{\ll \lambda^\sigma}|, |\partial_x u_\mu|, |h_1(u)_\mu|)\|_{L^\frac32_{t,x}} \lesssim \epsilon^4 \mu^{\sigma(\frac12 -2s)}.
\]

\medskip

The contributions of lines $3$ and $4$ of \eqref{vlong} are estimated by 
\[
\sum_{\mu >\lambda^\sigma} L( |\partial_x u_{\ll \lambda^\sigma}|, |\partial_x u_\mu|, |g(u)_\mu|).
\]
Here $g$ is at least quadratic after subtracting the constant term.
Hence we can estimate
\begin{equation}
\label{gu-mu}
|g(u)_\mu| \lesssim | P_\mu g(u)_{\ll \mu}| + L(|u|, |u_{\gtrsim \mu}|).
\end{equation}
For the first term in \eqref{gu-mu} we use a bilinear $L^2_{t,x}$ bound and Lemma~\ref{l:Moser-lh},
\[
\| L( |\partial_x u_{\ll \lambda^\sigma}|, |\partial_x u_\mu|, | P_\mu g(u)_{\ll \mu}|)\|_{L^1_{t,x}} \lesssim \epsilon^6 \lambda^\sigma \mu^{\frac12-3s}.
\]
For the second  term in \eqref{gu-mu} we need to estimate
\[
 L( |\partial_x u_{\ll \lambda^\sigma}|, |\partial_x u_\mu|, |u|, |u_{> \mu}|).
\]
Here we split the frequencies of $u$ into low and high relative to $\mu$,
and repeat the computation above.

Finally, to conclude the proof of part (a) of the proposition we remark that the same argument applies for $W$, which has an additional $\epsilon^2$ factor from the differentiation.

\bigskip

b) We now turn our attention to $Z$. Since we can bound all integrands in \eqref{vlong} by $L(|\partial_x u|,|\partial_x u|)$, the bound \eqref{Z} follows directly from the energy bound for $\partial_x u$. We next consider $\partial_y Z$, where the $y$ differentiation  cancels the integration in \eqref{vlong} modulo harmless $g_{[<\lambda^\sigma]}$ factors. We are left  with the integrands on the last three lines  of \eqref{vlong}, which we need to decompose as in \eqref{Z-bg}.which

The $Z^b$ component is easily identified as corresponding to the case when each instance of $\partial_x u$ is replaced by its lower frequency part 
$\partial_x [u_{< \lambda^{\sigma}}]$, and then the bound \eqref{Zb}
directly follows.

The $Z^g$ component is what is left. We begin with several simplifying observations:
\begin{itemize}
    \item We can harmlessly discard all $g_{[< \lambda^\sigma]}$ factors, which are bounded.
    \item The projector $P_{\lesssim \epsilon^2}^y$ has an $\epsilon^{2}$ 
    bound from $L^1_y$ to $L^\infty_y$.
    \item we can replace $g(u)$ by $g(u) - 1:=h_2(u)$, which is at least quadratic, in the last two lines of \eqref{vlong}; this insures that all terms in $Z^g$ are at least quartic.
\end{itemize}

Given these observations, it suffices to estimate the  remaining part of the integrands in \eqref{vlong} in $L^1_{t,x}$ by $\lambda^{-1}$. 
Taking absolute values and discarding bounded factors, this leaves us with the following contributions:

(i) from line 2 of \eqref{vlong} with two high frequencies on $\partial_x u$ we pull an $u$ factor from the second $h_1(u)$, and separate two cases depending on the frequency of this factor, obtaining
\[
L(|u_{\ll \lambda^\sigma}|, |u_{\ll \lambda^\sigma}| , |\partial_x u_{>\lambda^\sigma}|,|\partial_x u_{>\lambda^\sigma}|),
\]
respectively
\[
L(|u_{\ll \lambda^\sigma}|, |u_{\gtrsim \lambda^\sigma}| , |\partial_x u_{>\lambda^\sigma}|,|\partial_x u_{>\lambda^\sigma}|).
\]

(ii) From line  2 of \eqref{vlong}, with only one high frequency on $\partial_x u$, we must also have a high frequency on $h_1(u)$
\[
L(|u_{\ll \lambda^\sigma}|, |h_1(u)_{> \lambda^\sigma}| , |\partial_x u_{>\lambda^\sigma}|,|\partial_x u_{\lesssim \lambda^\sigma}|).
\]

(iii) From lines 3,4 of \eqref{vlong}:
\[
L(|\partial_x u_{\ll \lambda^\sigma} |, |h_2(u)_{>\lambda^\sigma}|, |\partial_x u_{>  \lambda^\sigma}|).
\]
These expressions are estimated in the same way as in part (a) of the 
proof. For the first expression in (i) we use two bilinear unbalanced $L^2_{t,x}$ 
bounds,
\[
\| L(|u_{\ll \lambda^\sigma}|, |u_{\ll \lambda^\sigma}| , |\partial_x u_{>\lambda^\sigma}|,|\partial_x u_{>\lambda^\sigma}|)\|_{L^1_{t,x}}
\lesssim \epsilon^4 \lambda .%{\sigma(1-2s)}.
\]

For the second expression in (i) we use one unbalanced bilinear $L^2_{t,x}$ bound, 
one energy bound for the $\partial_x u_{ \ll \lambda^\sigma }$ factor, and $L^\infty_{t,x}$ for the remaining $u_{\gtrsim \lambda^\sigma}$ factor, to obtain
\[
\| L(|u_{\ll \lambda^\sigma}|, |u_{\gtrsim \lambda^\sigma}| , |\partial_x u_{>\lambda^\sigma}|,|\partial_x u_{>\lambda^\sigma}|)\|_{L^2_t L^1_x}
\lesssim \epsilon^4 \lambda^{\sigma(\frac12-2s)},
\]
after which we use H\"older's inequality in time losing an $\epsilon^{-2}$ factor.

For the expression in (ii) we begin with 
\[
|P_{> \lambda^\sigma} h_1(u)| \lesssim |P_{> \lambda^\sigma} h_1(u_{\ll \lambda^\sigma})|
+ L(|u_{\gtrsim \lambda^\sigma}|).
\]
For the contribution of the first term we use an unbalanced bilinear $L^2_{t,x}$ bound, Lemma~\ref{l:Moser-lh} and a pointwise bound on $\partial_x u_{\lesssim \lambda^\sigma}$ to write 
\[
\| L(|u_{\ll \lambda^\sigma}|, |P_{> \lambda^\sigma} h_1(u_{\ll \lambda^\sigma})| , |\partial_x u_{>\lambda^\sigma}|,|\partial_x u_{\lesssim \lambda^\sigma}|)\|_{L^1_{t,x}} \lesssim \epsilon^7 \lambda^{\sigma( \frac32 -4s)}.
\]
For the contribution of the first term we use an unbalanced bilinear $L^2_{t,x}$ bound, a pointwise bound on $u_{\approx \lambda^\sigma}$ and an energy bound on $\partial_x u_{[\lesssim \lambda^\sigma ]}$ to write 
\[
\| L(|u_{\ll \lambda^\sigma}|, |u_{\gtrsim\lambda^\sigma}| , |\partial_x u_{>\lambda^\sigma}|,|\partial_x u_{\lesssim \lambda^\sigma}|)\|_{L^2_t L^1_x} \lesssim \epsilon^4 \lambda^{\sigma( \frac12 -2s)},
\]
after which we use H\"older's inequality in time losing an $\epsilon^{-2}$ factor.

Finally, the expression in (iii) is estimated in the same way as the 
one in (ii), with the only difference that $h_2$ is at least quadratic, 
so we have
\[
|P_{> \lambda^\sigma} h_2(u)| \lesssim |P_{> \lambda^\sigma} h_1(u_{\ll \lambda^\sigma})|
+ L(|u_{\ll \lambda^\sigma}|, |u_{\gtrsim \lambda^\sigma}|)
+ L(|u_{\gtrsim \lambda^\sigma}|, |u_{\gtrsim \lambda^\sigma}|).
\]
where we can use Lemma~\ref{l:Moser-lh} to bound the first term in $L^2_{t,x}$. 
\end{proof}

Our next step is to prove Strichartz estimates for the transformed equation \eqref{para-rea-final}.
We take the initial data $v_{0\lambda}$ and the source 
term $\tilde f_\lambda$ at frequency $\lambda$, and prove a classical Strichartz bound, and also $\lambda^{-N}$ decay away from frequency $\lambda$. 

One minor issue 
we encounter is that the flow in 
\eqref{para-rea-final} does not 
preserve the frequency localization at frequency $\lambda$.
To rectify this, we consider a slight modification of this problem
where the frequency localization is 
preserved. Then we show that the error we have introduced is negligible, i.e. $O(\lambda^{-N})$.
Our modified equation will be
\begin{equation}
\label{para-rea-final+}
i\partial_{\tau} v_{\lambda}+ \tilde P^y_\lambda(iP_{\lesssim \epsilon^2}^yV_2 \partial_y+\frac{i}2  P_{\lesssim \epsilon^2}^y \partial_y V_2)  \tilde P^y_\lambda v_{\lambda}+\partial^2_y v_{\lambda} = \tilde{f}_\lambda,
\end{equation}
where
$\tilde P^y_\lambda$ is a Littlewood-Paley multiplier in the $y$ coordinates which selects a larger 
frequency region than $P^y_\lambda$.

For this problem, we will prove the 
following:

\begin{proposition}\label{p:se-loc}
Let $u$ be a solution to \eqref{qnls}
which satisfies the bootstrap assumptions \eqref{uk-ee-boot}-\eqref{uk-se-short-boot}. Then for every frequency $\lambda$ localized data $v_{0\lambda} \in L^2$ and every source term $f_\lambda \in L^1_t L^2_x + L^\frac65_{t,x}$
the solution to \eqref{para-rea-final+} satisfies the Strichartz bound 
\begin{equation}\label{Str-y}
 \| v_\lambda \|_{L^\infty_t L^2_x \cap L^6_{t,x}} \lesssim \| v_{\lambda,0}\|_{L^2_x}
 + \|\tilde f_\lambda\|_{L^1_t L^2_x+ L^{\frac65}_{t,x}},  
\end{equation}
as well as 
\begin{equation}\label{cut-tP}
\| (1-\tilde P^y_\lambda) v \|_{L^\infty_t L^2_x} \lesssim \epsilon^N\lambda^{-N}.   
\end{equation}
\end{proposition}

We remark that our proof yields the full Strichartz bound
\begin{equation}
  \| v_\lambda \|_{S} \lesssim \| v_{\lambda,0}\|_{L^2_x}
 + \|\tilde f_\lambda\|_{S'}.
\end{equation}
However, only the $L^6_{t,x}$ and $L^\frac65_{t,x}$ components are used in the sequel.

\begin{proof} \
Denote by $S(t,s)$ the $L^2$ evolution operator associated to the flow \eqref{para-rea-final+}. By definition these are $L^2$ isometries, so by standard arguments it remains
to prove the $L^1$ to $L^\infty$ dispersive
decay bound
\begin{equation}\label{decay}
    \| S(t,s) v \|_{L^\infty_x} \lesssim |t-s|^{-\frac12}\|v\|_{L^1_x}, \qquad t,s \in [0,T_\epsilon = \epsilon^{-4}].
\end{equation}
Here without loss in generality we can assume that  $v$ is supported at frequency $\lambda$, as away from this frequency the evolution is governed by the flat Schr\"odinger flow. We can also freely assume that $s=0$ in order to simplify the notations.

For this, we will use the wave packet 
parametrix of \cite{MMT-param}, which 
implies the decay bound \eqref{decay}
by decomposing the initial data into coherent states, whose evolution is localized as a wave packet moving along the associated Hamilton flow.
This decomposition depends on the timescale $T = t-s$. For simplicity we set $s=0$ and then $T \lesssim \epsilon^{-4}$. Since we only consider waves at frequency $\lambda$, the bound \eqref{decay} is trivial, by Bernstein's inequality and $L^2$ conservation, provided that $T \lesssim \lambda^{-2}$.
So in what follows we assume that 
\begin{equation}\label{which-T}
\lambda^{-2} \lesssim T \lesssim \epsilon^{-2}.
\end{equation}

The first step is to understand what is the correct scale for the wave packets which remain coherent up to time $T$. Since our operator is quite similar to the flat Schr\"odinger operator, it is natural to use the same scales. Precisely, given a time $T$, the scale which yields coherent packets up to time $T$
is
\[
\delta y = \sqrt{T}, \qquad \delta \xi= \frac{1}{\sqrt{T}}.
\]
To use the wave packet parametrix of \cite{MMT-param}
relative to these scales, we need to show that 
our symbol\footnote{Here we simply take the principal symbol, as the lower order terms are perturbative.}
\[
a(y,\xi) = (P^y_{\ll \epsilon^2} V_2) \, \xi \, \tilde p_\lambda^2(\xi)  + \xi^2,
\]
verifies the conditions in Theorem~2 of \cite{MMT-param}. 
These are stated relative to the Hamilton flow 
$\chi(s,t)$ associated to the symbol $a$, which is defined by the ode
\[
\left\{
\begin{aligned}
 \dot y = & \ a_\xi(y,\xi) \\
 \dot \xi = & \  -a_y(y,\xi).
\end{aligned}
\right.
\]
Given an initial data $(y_0,\xi_0)$ at time $0$,
we denote by $(y_t,\xi_t)$ the associated trajectory.
Then the condition we need to verify is as follows:
\begin{equation}\label{good-symbol}
\int_0^T | \partial_y^\alpha \partial_\xi^\beta a(y_t,\xi_t)| \, dt \lesssim  (\delta y)^{-\alpha} (\delta \xi)^{-\beta}, \qquad \alpha+\beta \geq 2.
\end{equation}
By \eqref{which-T} we have $\delta y \lesssim \epsilon^{-2}$ and $\delta \xi \lesssim \lambda$. But every spatial derivative gains an $\epsilon^2$ factor due to the projector $P_{\ll \epsilon^2}$, and every $\xi$ derivative gains a factor
of $\lambda^{-1}$. Given these properties, 
it suffices to verify \eqref{good-symbol} when $\alpha+\beta = 2$. 

Since $u_x \in L^2$ with norm $\leq \epsilon$, we 
have the uniform bounds
\[
|P^y_{\ll \epsilon^2} V_2| \lesssim \epsilon^2,
\qquad | \partial_y P^y_{\ll \epsilon^2} V_2| \lesssim \epsilon^4,
\]
so the cases $\alpha = 0$ and $\alpha = 1$ of \eqref{good-symbol} are straightforward. But if $\alpha = 2$ then we need to take advantage of the integral in \eqref{good-symbol}. 

To prove \eqref{good-symbol} in this last case, we start 
with the observation that $|\dot x| \approx \lambda$
when $|\xi| \approx \lambda$. This simply says that waves with frequency $\lambda$ propagate with velocity 
$\lambda$. This is all the information we will use about trajectories, i.e. we will prove \eqref{good-symbol} holds uniformly for all trajectories with velocity $O(\lambda)$. Our starting point is the decomposition
\[
 \partial_y P^y_{\ll \epsilon^2} V_2 = Z^b+Z^g
\]
in Proposition~\ref{p:potentials}, where $Z^b$ and $Z^g$ can be freely assumed to also be localized at frequency $\lesssim \epsilon^2$.

By \eqref{Zg}, the good component satisfies 
\[
\| \partial_y Z^g\|_{L^1_t L^\infty_x} \lesssim \epsilon^{4} \lambda^{-1}, 
\]
exactly as needed for the $(\alpha,\beta) = (2,0)$ case of \eqref{good-symbol}.

On the other hand for $\partial_y Z^b$  we use \eqref{Zb}, which gives the bound 
\[
| \partial_y^2 P^y_{\ll \epsilon^2} V_2| \lesssim \epsilon^2 L(|\partial_x u_{<\lambda^\sigma}|,|\partial_xu_{<\lambda^\sigma}|).
\]
Since the class of trajectories with velocity $O(\lambda)$
is invariant with respect to translations, it follows that it suffices to show that on all such trajectories 
we have
\begin{equation}\label{wp-need+}
\int_0^T |\partial_x u_{<\lambda^\sigma}(t,x_t)|^2 \lambda \, dt\lesssim \epsilon^2,
\end{equation}
where $\lambda \, dt \approx dy \approx ds$, where $ds$ denotes arc-length.

It is easier to estimate this for dyadic portions of $u$, say $u_\mu$ with $\mu \leq \lambda^\sigma$. There
it suffices to consider the integral 
\[
I_{\mu} = \mu^2 \int_0^T | u_\mu(t,x_t)|^2 \, ds. 
\]

For this we use the mass conservation law \eqref{dens-flux-m0} for $u_\mu$, which gives
\[
\partial_t |u_\mu|^2 = \partial_x (g_{[<\mu]}P(u_\mu)
+ 2\Im(N_\mu(u) \cdot \bar u_\mu),
\]
and which we integrate in one of the strips $S$ in $[0,T] \times \R$ which is bounded by our trajectory $\gamma$.
This gives
\[
\int_\gamma |u_\mu|^2\, ds \lesssim \|u_\mu\|_{L^\infty L^2}^2 + \lambda^{-1} \int_\gamma|P(u_\mu,u_\mu)| \,ds
+ \int_S |N_\mu(u) \cdot \bar u_\mu| \, dx dt.
\]
For the first term on the right we use \eqref{uk-ee-boot}, and for the last we use the bounds \eqref{good-nl} and \eqref{res-nl} with $C=1$. We also note that
\[
|P(u_\mu)| \lesssim \mu L(|u_\mu|,|u_\mu|).
\]
Then we arrive at
\[
\int_\gamma |u_\mu|^2 \, ds \lesssim  \frac{\mu}{\lambda}  \int_\gamma L(|u_\mu|,|u_\mu|) \, ds
+  \epsilon^2 c_\mu^2 \mu^{-2s}.
\]
We now take the supremum over spatial translates of $\gamma$,
\[
\sup_{x_0 \in \R} \int_{\gamma^{x^0}} |u_\mu|^2 ds \lesssim  \frac{\mu}{\lambda}  \sup_{x_0 \in \R} \int_{\gamma^{x_0}} L(|u_\mu|,|u_\mu|) \,ds
+  \epsilon^2 c_\mu^2 \mu^{2-2s}.
\]
This allows us to absorb on the left the first term on the right, to obtain
\[
I_\mu \lesssim  \epsilon^2 c_\mu^2 \mu^{-2s}.
\]
Summing over $\mu$, we obtain the desired bound \eqref{wp-need+}, which completes the proof of \eqref{good-symbol}.
\bigskip

The bound \eqref{good-symbol} guarantees that 
(a rescaled version of) Theorem~2 in \cite{MMT-param}
applies, and yields a wave packet parametrix on the $(\delta y,\delta\xi)$ scale indicated above. Precisely, (the rescaled form of) this result asserts that in the time interval $[0,T]$ the fundamental solution admits a representation of the form
\begin{equation}\label{wp}
 K(t,y,z_0) = (\delta y)^{-1} \int e^{i \xi_0 (z_0-y_0)}e^{-\frac{(z_0-y_0)^2}{(\delta y)^2}}
 e^{-i \eta (y-z)}e^{-\frac{(y-z)^2}{(\delta y)^2}} K^{PS}(t,z,\eta,x_0,\xi_0) 
 \, d\xi_0 dy_0 d\eta dz,
\end{equation}
where the \emph{phase-space kernel\footnote{The name is consistent with the fact that this kernel is obtained by conjugating the flow map with a Bargman transform on the appropriate scales, see e.g. \cite{T-Pisa}.}} $K^{PS}$ decays rapidly away from the 
 graph of the flow map $\chi(t,0)$ of the Hamilton flow on the $(\delta y,\delta \xi)$ scale,
\[
K^{PS}(t,z,\eta,x_0,\xi_0) \lesssim_N (1 + (\delta y)^{-1}|z-x_t| + (\delta \xi)^{-1} |\eta - \xi_t|)^{-N}.
\]
Combined with the bi-Lipschitz property of the Hamilton flow on 
the $(\delta y,\delta \xi)$ scale, this directly yields the kernel bound
\[
|K(t,y,y_0)| \lesssim  \delta y \int (1 + (\delta y)^{-1}|y-y_t|)^{-N} \, d \xi_0.
\]
Finally, in order to obtain the dispersive estimate \eqref{decay} from time $0$ to time $T$ we need to bring into play dispersion, which has not been used so far. In our case this corresponds to the nondegenerate dependence 
of $y_t$ on $\xi_0$, when $y_0$ is fixed. We claim that we have 
\begin{equation}\label{Hlin}
 \frac{ \partial y_t}{\partial \xi_0} = 2t( 1+ O(\epsilon^2+\epsilon^4 T)).   
\end{equation}
For $T \ll \epsilon^{-4}$ the coefficient on the right has size $2t$.
Inserting this into the previous kernel bound, this yields  
\begin{equation} \label{wp-disp}
|K(t,y,y_0)| \lesssim   \int (1 + (\delta y)^{-1}|y-y_t|)^{-N} \, d \xi_0
\lesssim \frac{\delta y}{|t|},
\end{equation}
which at $t = T$ yields the desired bound $|K(t,y,y_0)|\lesssim T^{-\frac12}$.

It remains to see that \eqref{Hlin} holds, which we do by considering 
the linearized Hamilton flow for the variables 
\[
(z,\eta):=\left(  \frac{ \partial y_t}{\partial \xi_0}, \frac{ \partial \xi_t}{\partial \xi_0}   \right).
\]
This has the form
\[
\left\{
\begin{aligned}
 \dot z = & \ a_{\xi y}(y_t,\xi_t) z+ a_{\xi_t\xi_t} \eta \\
 \dot \eta = & \  -a_{yy}(y_t,\xi_t) z - a_{y_t \xi_t} \eta
\end{aligned}
\right., \qquad 
\begin{aligned}
&z(0) =0 \\
&\eta(0) = 1.
\end{aligned}
\]
Here we know that
\[
|a_{\xi\xi} - 2| \lesssim \epsilon^2, \qquad |a_{y\xi}| \lesssim \epsilon^4,
\qquad \int |a_{yy}| \, dt  \lesssim \epsilon^2,
\]
so we easily obtain
\[
z = 2t ( 1+ O(\epsilon^2 + \epsilon^4 T)), \qquad \eta = 1+ O(\epsilon^2 + \epsilon^4 T),
\]
and in particular the bound \eqref{Hlin} is proved. This completes the proof of the Strichartz estimate \eqref{Str-y}.

Finally, the bound \eqref{cut-tP} is an immediate
consequence of the wave packet localization around the 
bicharacteristics, with rapid off-diagonal decay
on the $\delta \xi$ scale, which can be taken to be 
of size $\epsilon^2$ when $T \approx \epsilon^{-4}$.

We conclude the proof of Proposition~\ref{p:se-loc} with a final remark
concerning wave packet scales.

\begin{remark}
The scales $(\delta y,\delta \xi) = (T^{\frac12},T^{-\frac12})$ in the wave packet decomposition corresponding to the representation \eqref{wp} are determined by desired coherence time scale $T$. But using this decomposition only yields the correct dispersive bound if $t \approx T$, while for smaller
$t$ there is a loss, see \eqref{wp-disp}. This is explained by the fact that wave packets on the above scales fully separate by time $O(T)$, while at earlier times they have a substantial overlap. So the wave packet time-scale  has to be adapted to the time $T$ in the dispersive estimate. On the other hand in order to get the best frequency localization in this representation, we need to work with wave packets with the best frequency localization, which corresponds to the longest time-scale, $\epsilon^{-4}$ in our case.
\end{remark}

\end{proof}

To return to the problem \eqref{para-reb}, we change coordinates back.
The goal is to show the following:

\begin{proposition}\label{p:reb}
Let $v_{0\lambda} \in L^2_x$ and $f_\lambda \in L^\frac65_{t,x}$.
Then the evolution \eqref{para-reb} has a frequency localized approximate solution $v_\lambda$ in $[0,T_\epsilon=\epsilon^{-4}]$,
\begin{equation}\label{para-reb-app}
i \partial_t v_\lambda + \partial_x g_{[<\lambda^\sigma]} \partial_x v_\lambda + V_{<\lambda} \partial_x v_\lambda
= f_\lambda + f_\lambda^{err}, \qquad v_\lambda(0) = v_{0,\lambda}+ 
 v_{0,\lambda}^{err},
\end{equation}
with $O(\lambda^{-N})$ errors,
\begin{equation}\label{para-se-err}
 \| v_{0,\lambda}^{err} \|_{L^2_x} 
 + \|f_\lambda^{err}\|_{L^1_t L^2_x + L^{\frac65}_{t,x}} \lesssim 
 \lambda^{-N}( \| v_{\lambda,0}\|_{L^2_x}
 + \|f_\lambda\|_{L^1_t L^2_x + L^{\frac65}_{t,x}})
\end{equation}
and so that
\begin{equation}\label{para-se-app}
 \| v_\lambda \|_{L^\infty_t L^2_x \cap L^6_{t,x}} \lesssim \| v_{\lambda,0}\|_{L^2_x}
 + \|f_\lambda\|_{L^1_t L^2_x + L^{\frac65}_{t,x}}.
\end{equation}
\end{proposition}

\begin{proof}
If our change of coordinates were
to preserve frequency localizations,
then there would be nothing to prove.
It does not do this exactly, but
it does approximatively. 

  We need to show that, with negligible tails, the dyadic localization is preserved by the change of coordinates. Denoting by $P_\lambda$ the Littlewood-Paley projectors in the original coordinates 
   and by $\tilde P_\lambda^y$ a similar projectors in the $y$ coordinates with, say, double support, we have
\begin{lemma} \label{l:switch-FT}
The following estimate holds:
\begin{equation}
\| (1-\tilde P^y_\lambda ) P_\lambda \|_{L^2 \to L^2} \lesssim \lambda^{-N},
\end{equation}
and its analogue with the two sets of coordinates interchanged.
\end{lemma}
\begin{proof}
Composing to the left and to the right with the respective Fourier transforms, we obtain an integral operator with kernel
\[
K(\xi,\eta) = \int e^{ix \xi - y \eta} \, dx.
\]
It suffices to show that for $|\eta| \approx \lambda$ and $|\xi|\not \approx \lambda$
we have 
\begin{equation}\label{two-ft}
|K(\xi,\eta)| \lesssim (|\xi|+|\eta|)^{-N}.
\end{equation}
This is a nonstationary phase bound.
We write $\dfrac{dy}{dx} = 1+ h$ 
where
\[
\| \partial^j h\|_{L^\infty_x} + \|\partial^{j+1} h \|_{L^1_x} \lesssim
\epsilon \lambda^{\sigma j}, \qquad j \geq 0.
\]

One integration by parts yields
\[
K(\xi,\eta) = - \int e^{ix \xi - y \eta} \frac{d}{dx} \left( \frac{1}{\xi - \eta(1+h')}\right) \, dx  = 
- \int e^{ix \xi - y \eta}  \frac{\eta h''}{(\xi - \eta(1+h'))^2} \, dx,  
\]
where we can estimate
\[
|K(\xi,\eta)| \lesssim \epsilon \lambda^{\sigma}
(|\xi|+|\eta|)^{-2}.
\]
Each additional integration by parts yields
an extra factor of $\lambda^{\sigma}
(|\xi|+|\eta|)^{-1}$. Then \eqref{two-ft} follows.

\end{proof}

Having this lemma, the proof of the proposition is straightforward.
 We change $v_{0\lambda}$ and $f_\lambda$ to the $y$ coordinates, 
 and then localize them to frequency
 $\lambda$, with negligible tails.
 We then solve the problem \eqref{para-rea-final+}, with good Strichartz estimates. By \eqref{cut-tP} can remove
the projector $\tilde P_\lambda$, again with negligible tails.
Finally, we switch the solution
back to the $(x,t)$ coordinates, and
relocalize to frequency $\lambda$ using again Lemma~\ref{l:switch-FT}.
   
\end{proof}

Now we complete the proof of Theorem~\ref{t:para-se}. By Proposition~\ref{p:equiv}, it suffices 
to prove the result for the equation~\eqref{para-reb}.
Consider a solution $v_\lambda$ for \eqref{para-reb} with source term 
$f_{\lambda} = f_{\lambda 2} + f_{\lambda 3}$ where $f_{\lambda 2} \in L^1_t L^\infty_x + L^{\frac65}_{t,x}$ and $f_{\lambda 3} \in X_\lambda^*$.

We test the $v_\lambda$ equation with the approximate solution $w_\lambda$ for the adjoint equation given by Proposition~\ref{p:reb} (which applies equally to the adjoint equation).

 The argument goes by duality. Given $w_\lambda$ a solution 
to the adjoint equation with 
initial data $w_\lambda(T)$
source term $g_\lambda \in L^1_t L^\infty_x + L^\frac65_{t,x}$ 

and respective $O(\lambda^{-N})$ errors $w_\lambda^{err}(T)$ and 
$g_\lambda^{err}$, we write the duality relation
\[
\la v_\lambda,w_\lambda \ra|_0^T = \iint v_\lambda g_\lambda - w_\lambda f_\lambda \, dx dt.
\]
We rewrite this as
\begin{equation}
\iint v_\lambda (g_\lambda+g_\lambda^{err}) \, dx dt 
- \langle v_\lambda(T),(w_\lambda(T)+w_\lambda^{err}(T)) \rangle
= \iint  w_\lambda f_\lambda \, dx dt - \langle v_\lambda(0),w_\lambda(0) \rangle.
\end{equation}
We estimate the right hand side directly,
\[
\lesssim \| w_{\lambda}\|_{L^\infty_t L^2_x \cap L^6_{t,x}\cap  X_\lambda}
\| f_\lambda\|_{L^1_t L^2_x + L^{\frac65}_{t,x} + X_\lambda^*} + \|v_\lambda(0)\|_{L^2_x} \| w_\lambda(0)\|_{L^2_x},
\]
which by \eqref{para-se-app} becomes
\[
\lesssim 
(\| f_\lambda\|_{L^1_t L^2_x + L^{\frac65}_{t,x} + X_\lambda^*} + \|v_\lambda(0)\|_{L^2_x}) (\| w_\lambda(T)\|_{L^2_x}
+\|g_\lambda \|_{L^1_t L^2_x + L^\frac65_{t,x}}).
\]
On the left we can neglect the small perturbative error terms, and take 
the supremum over all $g_{\lambda}$ 
and $w_T$ with 
\[
\| w_\lambda(T)\|_{L^2_x}
+\|g_\lambda \|_{L^1_t L^2_x + L^\frac65_{t,x}} \lesssim 1,
\]
where we can harmlessly allow a slightly larger Fourier support than for $v_\lambda$.  Then \eqref{para-se} follows, and the proof of the theorem is concluded.

\section{The local well-posedness result}
\label{s:lwp}

In this section we prove our main local well-posedness result in Theorem~\ref{t:local}.
Our starting point is the local well-posedness result for regular data in Theorem~\ref{t:regular}. For these solutions we will prove the frequency envelope bounds in Theorem~\ref{t:local-fe}, which in turn allow us to continue the solutions up to an $O(\epsilon^{-4})$ time, with uniform bounds. We will then prove Theorem~\ref{t:local} by constructing rough solutions as unique limits of smooth solutions.
We proceed in several steps:

\bigskip

\textbf{ STEP 1: A-priori $H^s$  estimates for regular solutions: Proof of Theorem~\ref{t:local-fe}.}
We recall that we can freely make the bootstrap assumption that 
the bounds \eqref{uk-ee-boot}-\eqref{uk-se-boot} hold with a large universal constant $C$. Then we seek to apply Theorems~\ref{t:para}, \ref{t:para-se} with $v_\lambda=u_\lambda$ and $\epsilon$ replaced by $C\epsilon$, which is always assumed to be $\ll 1$. To achieve this we write the equation for $u_\lambda$ in the paradifferential form \eqref{para-full}, and use Lemma~\ref{l:N-lambda} 
for the components of the source term $N_\lambda$. We can use 
the smallness of $\epsilon$ to absorb the $C$ factors in \eqref{good-nl} and \eqref{res-nl}. Then the $d_\lambda$'s associated to $u_\lambda$  as in \eqref{dl} will be
\[
d_\lambda = \epsilon \lambda^{-s} c_\lambda.
\]
Hence, the conclusion of Theorem~\ref{t:local-fe}
follows directly by applying Theorem~\ref{t:para} for the energy and bilinear estimates, respectively Theorem~\ref{t:para-se} for the Strichartz estimates.
For later use, we record the following:

\begin{corollary}
Under the same assumptions as Theorem~\ref{t:local-fe}, the bounds \eqref{good-nl} for $N^{tr}_\lambda$ 
and \eqref{res-nl} for $C_\lambda$ hold with $C=1$.
\end{corollary}

\bigskip

\textbf{STEP 2: Higher regularity.} 
Here we consider initial data $u_0$ which is not only $\epsilon$-small in $H^s$, but also belongs to a higher Sobolev space $H^{s_1}$, say
\[
\|u_0\|_{H^s} \leq \epsilon, \qquad \| u_0\|_{H^{s_1}} \leq M, \qquad s_1 > s.
\]
Here we do not assume any smallness on $M$.
Then we claim that the solution $u$ remains in $H^{s_1}$ within the time interval in Theorem~\ref{t:local-fe}, with a uniform bound
\begin{equation}
\| u\|_{L^\infty_t H^{s_1}_x} \lesssim M.    
\end{equation}

To see why this is true,  for the initial data $u_0$ we consider a minimal $H^s$ frequency envelope $\epsilon c_\lambda$ with the unbalanced slowly varying condition as in Remark~\ref{r:unbal-fe}. Then by construction we have
\[
\sum_\lambda (\epsilon c_\lambda \lambda^{s_1-s})^2 \lesssim M^2.
\]
By Theorem~\ref{t:local-fe} the frequency envelope $c_\lambda$ is propagated along the flow at least up to time $t \ll \epsilon^{-4}$, and we obtain
\[
\| u(t) \|_{H^{s_1}}^2 \lesssim \sum_\lambda (\epsilon c_\lambda \lambda^{s_1-s})^2 \lesssim M^2, \qquad t \in [0,T].
\]
We remark that not only the uniform $L^2$ higher regularity bounds are propagated along the flow, but also the corresponding bilinear $L^2_{t,x}$ and Strichartz bounds.

\bigskip

\textbf{ STEP 3: A-priori $L^2$ estimates for the linearized equation: Proof of Theorem~\ref{t:linearize-fe}.}

Here we seek to apply Theorem~\ref{t:para} and Theorem~\ref{t:para-se}  by writing the linearized equation \eqref{qnls-lin} in the paradifferential form \eqref{para-lin}. To start with, we take $d_\lambda$ to be the $L^2$ frequency envelope for the initial data $v_0$, and seek to prove that the 
Littlewood-Paley components $v_\lambda$ of $v$ satisfy the bounds 
\eqref{v-ee}-\eqref{vv-bi-unbal}.

To prove this we make a bootstrap assumption for $v_\lambda$, which is the same as in the proof of Theorem~\ref{t:para}, namely the bounds \eqref{v-ee-boot}-\eqref{vv-bi-unbal-boot}, together with a Strichartz bootstrap assumption,
\begin{equation}\label{v-se-boot}
\| v_\lambda \|_{L^6_{t,x}} \leq C d_\lambda.    
\end{equation}
The first step is to prove the following:
\begin{lemma}\label{l:lin-1}
 Let $s > 1$. Assume that the function $u$ satisfies the bounds 
 \eqref{uk-ee}-\eqref{uk-se} in a time interval $[0,T]$ with $T$ as in \eqref{small-t-loc}.
 Assume also that $v$ satisfies the bootstrap assumptions  \eqref{v-ee-boot}-\eqref{vv-bi-unbal-boot}
 and \eqref{v-se-boot}. Then for $\epsilon$ small enough, the functions $N^{lin}_\lambda v$ satisfy the bounds
 \begin{equation}\label{good-nl-lin}
\| N^{lin}_\lambda v  \cdot \bv_\lambda^{x_0}\|_{L^1_{t,x}} \lesssim C^2 \epsilon^2 d_\lambda^2.
 \end{equation}

\end{lemma}
The above Lemma allows us to  apply Theorem~\ref{t:para} to $v_\lambda$,
and conclude that if $\epsilon$ is small enough,
then the bounds \eqref{v-ee}-\eqref{vv-bi-unbal} hold.

\medskip

The second step is the next Lemma:
\begin{lemma}\label{l:lin-2}
 Let $s > 1$. Assume that the function $u$ satisfies the bounds 
 \eqref{uk-ee}-\eqref{uk-se} in a time interval $[0,T]$ with $T$ as in \eqref{small-t-loc}.
 Assume also that $v$ satisfies the bootstrap assumptions  \eqref{v-ee-boot}-\eqref{vv-bi-unbal-boot}
 and \eqref{v-se-boot}. Then for $\epsilon$ small enough, the functions $N^{lin}_\lambda v$ satisfy the bounds
\begin{equation}\label{Nv-Xstar}
 \|   N^{lin}_\lambda v \|_{X_\lambda^*}  \lesssim  C^2 \epsilon^2 d_\lambda^2.
\end{equation}
\end{lemma}
This lemma allows us to  apply Theorem~\ref{t:para-se} to $v_\lambda$,
and conclude that if $\epsilon$ is small enough,
boundthen the $L^6_{t,x}$ bound \eqref{v-se} holds, thereby completing the proof of Theorem~\ref{t:linearize-fe}. It remains to prove the lemmas.

\begin{remark}
We remark on the close connection between the two lemmas. On one hand, 
the proof of Lemma~\ref{l:lin-2} is almost identical to the proof of Lemma~\ref{l:lin-1}. On the other hand, one might be tempted 
to think that Lemma~\ref{l:lin-1} could be viewed as a corollary
of Lemma~\ref{l:lin-2}. However, the function $v_\lambda$ qualifies 
as a test function as in \eqref{X-lambda} only a-posteriori, after 
Lemma~\ref{l:lin-1} has been established. So such an interpretation would lead to a circular argument.  Instead, the logical progression is to prove 
Lemma~\ref{l:lin-1} first, where we only use our bootstrap assumptions. 
This in turn qualifies $v_\lambda$ for the application of Theorem~\ref{t:para}, and that is a key element in the proof of
Lemma~\ref{l:lin-2}.

\end{remark}

\begin{proof}[Proof of Lemma~\ref{l:lin-1}]
The linearized equation has the form 
\[
i \partial_t v + g(u) \partial_x^2 v = 
- g'(u) \partial_x^2 u \, v + N_1(u,\partial u)\, v + N_2(u,\partial u) \, \partial v,
\]
where $N_1$ and $N_2$ denote the derivatives of $N$ with respect to the first, respectively the second argument; they are at most quadratic, respectively linear in the second argument.

Writing the above equation in a paradifferential form we obtain the expression 
for $N_\lambda^{lin} v$  as follows
\[
\begin{aligned}
 N_\lambda^{lin} v = & \  g_{[<\lambda]} \partial_x^2 v_\lambda - P_\lambda ( g(u) \partial_x^2 v)    
+ \partial_x g_{[<\lambda]} \partial_x v_\lambda + P_\lambda 
(N_1(u,\partial u) v + N_2(u,\partial u) \partial v-g'(u) \partial_x^2 u v )
\\ 
= & \ 
\left([g_{[<\lambda]},P_\lambda]  \partial_x^2 v  
+ (\partial_x g_{[<\lambda]}) \partial_x v_\lambda\right) + P_\lambda( P_{\geq \lambda} g(u_{\ll \lambda})
\partial_x^2 v)- P_\lambda ( ( g(u)-g(u_{\ll \lambda})) \partial_x^2 v)
\\ & \
+ P_\lambda (N_1(u,\partial u) v)  
+ P_\lambda (N_2(u,\partial u) \partial v) 
+ P_\lambda (g'(u) \partial^2 u v) 
\\
:= &  \ N^{l,1}_\lambda v + N^{l,2}_\lambda v + N^{l,3}_\lambda v + N^{l,4}_\lambda v + N^{l,5}_\lambda v + N^{l,6}_\lambda v.
\end{aligned}
\]

\smallskip

Here the first three terms closely resemble the similar terms 
in the expansion of $N_\lambda(u)$ in the proof of Lemma~\ref{l:N-lambda}. We consider the contributions
of each of these terms to the expression $ N^{lin}_\lambda v  \cdot \bv_\lambda^{x_0}$:

\medskip

\textbf{ The contribution of $N^{l,1}_\lambda v$.} The estimate here is identical to that for $N_\lambda^1$ in Lemma~\ref{l:N-lambda}.

\medskip

\textbf{ The contribution of $N^{l,2}_\lambda v$.}
Here we argue again as in Lemma~\ref{l:N-lambda}, decomposing
\[
\begin{aligned}
N_\lambda^{l,2}v = & \ L(P_{\lambda} g(u_{\ll \lambda}), \partial^2 v_{<\lambda})
+ L(P_{\lambda}  g(u_{\ll \lambda}), \partial^2 v_{\lambda}) +
\sum_{\mu \gg \lambda} L( P_\mu g(u_{\ll \lambda}),\partial^2 v_\mu)
\\
:= & \ N_\lambda^{l,2,lo} v + N_\lambda^{l,2,med} v + N_\lambda^{l,2,hi} v,
\end{aligned}
\]
where for the $g$ factor we use again \eqref{Moser-tail}.
The contribution of $ N_\lambda^{l,2,lo} v$ can be rewritten in the form 
\[
 N_\lambda^{l,2,lo}  v_\lambda^{x_0} =   \lambda L(P_{\lambda} g(u_{\ll\lambda}), \partial v_{<\lambda}, v_\lambda),
\]
and can be estimated using \eqref{Moser-tail} for the first factor  and the bilinear, unbalanced $L^2_{t,x}$ bound  \eqref{vv-bi-unbal-boot} for the remaining two.

The contribution of $ N_\lambda^{l,2,mid} v$ can be rewritten in the form
\[
 N_\lambda^{l,2,med} v v_\lambda^{x_0} =   L(P_{\lambda} \partial^2 g(u_{\ll \lambda}),  v_{\lambda}, v_\lambda).
\]
Dropping $P_\lambda$  and using chain rule in the first factor, here we can use two instances of \eqref{uv-bi-unbal-boot}.

Finally the contribution of the summands in $ N_\lambda^{2,hi}$ can be rewritten in the form
\[
 N_\lambda^{l,2,high} v v_\lambda^{x_0} = \mu^2  L(P_{\mu} g(u_{\ll \lambda}),  v_{\mu}, v_\lambda),
\]
and can be estimated using \eqref{Moser-tail} for the first factor and \eqref{uv-bi-unbal-boot} for the remaining two.

\medskip

\textbf{ The contribution of $N^{l,3}_\lambda v$.}
The argument is identical to the one for $N^3_\lambda(u)$ 
in Lemma~\ref{l:N-lambda} when the frequency of $v$ is $\lesssim \lambda$, so it remains to examine components of the form
\[
L(P_{\mu}(g(u) - g(u_{\ll \lambda})), \partial^2 v_\mu), \qquad \mu \gg \lambda.
\]
For the $g$ difference we use the expansion \eqref{g-para}.
The first term in \eqref{g-para} gives contributions of the form
\[
L(P_{\mu} (u_{\mu_1}g'(u_{\ll \lambda})), \partial^2 v_\mu) v_\lambda^{x_0},
\qquad \mu_1 \gtrsim \lambda.
\]
Discarding the easier case when $\mu_1 \neq \mu$, where 
we can insert an extra high frequency projector on $g'(u_{\ll\lambda})$ and use \eqref{Moser-tail}, we are left with
the contribution
\[
L(u_\mu,g'(u_{\ll \lambda}),\partial^2 v_\mu,v_\lambda).
\]
Then we can use one instance of \eqref{uab-bi-unbal} and one instance of 
\eqref{vv-bi-unbal-boot}.

The second term in \eqref{g-para} gives contributions of the form
\[
L(P_{\mu} (u_{\mu_1} u_{\mu_2}g''(u_{<\mu_2})), \partial^2 v_\mu) v_\lambda^{x_0}, \qquad \mu_1 \geq \mu_2 \gtrsim \lambda.
\]
where  we can again use \eqref{Moser-tail} unless
either $\mu_1 \approx  \mu$ or $\mu_1 \approx \mu_2 \gg \mu$.
But in both of these cases it suffices to use \eqref{uv-bi-unbal-boot}
twice.

\medskip

\textbf{ The contribution of $N^{l,4}_\lambda v$.}
The argument is identical to the one for $N^4_\lambda(u)$ 
in Lemma~\ref{l:N-lambda} when the frequency of $v$ is $\lesssim \lambda$, so it remains to examine components of the form
\[
L(P_{\mu} N_1(u,\partial u), v_\mu)v_\lambda^{x_0}, \qquad \mu \gg \lambda,
\]
where $N_1$ is at least quadratic and may have at most two $\partial u$ factors. Here we can use a bilinear $L^2_{t,x}$ bound for the $v$ factors, so if $N_1$ is at least cubic then three $L^2_{t,x}$ bounds will suffice. Hence we are left with the case when $N_1$ is exactly quadratic, 
where it suffices to estimate 
\[
L(P_{\mu} ( \partial u \cdot \partial u), v_\mu)v_\lambda^{x_0}.
\]
Here we can immediately use two unbalanced $L^2_{t,x}$ bounds
unless both $\partial u$ factors have $O(\mu)$ frequencies. But even in this last case, the projector $P_{\mu}$ insures that the two $\partial u$ factors are  
either frequency separated, or separated from $v_\mu$,
in which case we can use again two bilinear $L^2_{t,x}$ bounds.

\medskip

\textbf{ The contribution of $N^{l,5}_\lambda v$.}
As in the previous case it suffices to consider the 
scenario where $v$ is at frequency $\mu \gg \lambda$,
and estimate
\[
L(P_{\mu} N_2(u,\partial u), \partial v_\mu)v_\lambda^{x_0}, \qquad \mu \gg \lambda,
\]
where $N_2$ is at least quadratic with at most one instance of $\partial u$. There are several cases to consider:
\begin{enumerate}[label = (\roman*)]
    \item If all $N_2$ entries are at frequency $\ll \mu$
then we can use Lemma~\ref{l:Moser-lh} for $N_2$, and an unbalanced bilinear $L^2_{t,x}$ bound for the $v$ factors.
\item If $N_2$ has an unbalanced pair of frequencies 
with the highest at least $\mu$, then we can use two unbalanced bilinear $L^2_{t,x}$ bounds.
\item If $N_2$ has at least three frequency $\mu$ 
factors then we can use three $L^6_{t,x}$ bounds and an unbalanced bilinear $L^2_{t,x}$ bound for the $v$ factors.
\item  If $N_2$ has at exactly two frequency $\mu$ 
factors then they must be  are  
either frequency separated, or separated from $v_\mu$,
in which case we can use again two bilinear $L^2_{t,x}$ bounds.
\end{enumerate}

\medskip

\textbf{ The contribution of $N^{l,6}_\lambda v$.}
This is similar with $N^{l,5}_\lambda v$ unless the frequency of the $\partial^2 u$ factor is much larger than the frequency of all the other factors. So we still have to consider the expression 
\[
g'(u_{\ll \lambda}) \partial^2 u_{\lambda} v_{\ll \lambda}v_\lambda^{x_0}.
\]
Its contribution can be rewritten in the form
\[
L(h(u_{\ll \lambda}), u_{\ll\lambda},\partial^2 u_\lambda,v_{\ll\lambda}, v_\lambda).
\]
Again, after discarding the first factor, which is bounded, it suffices to use one instance of \eqref{uab-bi-unbal-boot} for the next two factors, and one instance of \eqref{vv-bi-unbal-boot} for the last two. 
\end{proof}

\begin{proof}[Proof of Lemma~\ref{l:lin-1}]
Here we consider an additional function $w_\lambda$ which solves the adjoint equation \eqref{dual}, with $\|w_\lambda\|_{X_\lambda} \leq 1$,
and we seek to show that 
\[
\| N_\lambda^{lin}v \cdot w_\lambda\|_{L^1_{t,x}} \lesssim C^2 \epsilon^2 d_\lambda.
\]
Given the control we have over $w_\lambda$ via the $X_\lambda$ norm,
and the control we have over $v_\lambda$ through Lemma~\ref{l:lin-1},
it follows that we have full access to the bilinear $L^2_{t,x}$ estimates
for both $v_\lambda$ and $w_\lambda$. The only missing piece of information for $w_\lambda$ is the $L^6_{t,x}$ Strichartz bound.

Hence, the proof of Lemma~\ref{l:lin-2} repeats the proof of Lemma~\ref{l:lin-1}, as long as we can avoid using the  $L^6_{t,x}$ Strichartz bound for $w_\lambda$. But this is always the case unless we have at least 
three factors at frequency $\lambda$ and at most one at a different frequency, which then has to be lower. The more factors we have, the better, so it suffices to examine quartic terms. For these we consider two cases:

\begin{enumerate}
    \item Four balanced entries. Then we estimate for instance
    \[
\| u_\lambda \partial^2 u_\lambda v_\lambda w_\lambda\|_{L^1_{t,x}}
\lesssim T^\frac12 \lambda^2 \|u_\lambda \|_{L^6_{t,x}}^2 \|v_\lambda \|_{L^6_{t,x}}
\|w_\lambda \|_{L^\infty_t L^2_x} \lesssim C^3 \epsilon^2 c_\lambda^2 d_\lambda \lambda^{2-2s},
\]
which suffices.
 \item Three balanced entries, and a lower one. Then we would have
to estimate an expression of the form
\[
u_{\ll \lambda} P_\lambda( \partial^2 u_\lambda, v_\lambda) w_\lambda.
\]
However, if we carefully choose the dyadic regions narrow enough, as discussed in Remark~\ref{r:nonres}, then the middle projector will cancel 
the product of two balanced entries. So this case cannot occur at all.
 
\end{enumerate}
The proof of the lemma is concluded.

\end{proof}

\textbf{ STEP 4: Conclusion.}
Here we show that the $H^s$ energy estimates for the full equation, combined with the $L^2$ energy estimates for the linearized equation, both in the frequency envelope formulation, yield the local well-posedness result in Theorem~\ref{t:local}. We have arrived now at a point where a standard argument applies, and for which we outline the steps.  For more details we 
refer the reader to the expository paper~\cite{IT-primer}, where the strategy of the proof is presented in detail. Here it is more convenient to use the dyadic notation for frequencies, also for easier comparison with \cite{IT-primer}.
The steps are as follows:

\begin{enumerate}[label=(\roman*)]
    \item Initial data regularization:
 given $u_0 \in H^s$, with frequency envelope $\epsilon c_k$, we consider 
 the regularized data 
\[
u_0^h = P_{<h} u_0,
\]
where $h \in \R^+$.
where we can use an improved frequency envelope
\[
c^h_j = \left\{ \begin{array}{lc}
c_j & j \leq k,
\cr 
c_h \step^{-N(j-h)} & j > h.
\end{array} 
\right.
\]

For this data we consider the corresponding smooth solutions $u^h$. Using the local well-posedness result for regular solutions in Theorem~\ref{t:regular} and the  bounds in Theorem~\ref{t:local-fe}, a continuity argument shows that these 
solutions extend as regular solutions
up to time $T \ll \epsilon^{-4}$, 
with similar frequency envelope bounds
\begin{equation}\label{ul-fe}
 \|P_{j} u^h \|_{L^\infty L^2}
 \lesssim \epsilon c_j^h 
 \step^{-sj}.
\end{equation}

\item Difference bounds: here we define 
\[
v^h = \frac{d}{dh} u^h,
\]
 which solves the linearized equation around $u^h$,
 and has initial data
\[
v^h(0) = P_h u_0,
\]
which satisfies
\[
\|v^h(0)\|_{L^2_x} \lesssim \epsilon \step^{-sh} c_h.
\]
By the energy estimates in Theorem~\ref{t:linearize-fe}, it follows that 
\begin{equation}\label{vl-fe}
\| v^h\|_{L^\infty_t L^2_x} \lesssim  \epsilon\step^{-sh} c_h.   
\end{equation}

\item Convergence: 
Here we think of the tentative limit $u$ of $u^h$ 
as the telescopic sum
\[
u = u^0 + \sum_{k\geq 0} u^{k+1}-u^k.
\]
By \eqref{vl-fe} we have 
\begin{equation}\label{dul-fe}
\| u^{k+1}-u^k\|_{L^\infty_t L^2_x} \lesssim  \epsilon \step^{-sk} c_k,   
\end{equation}
which implies rapid convergence of the series in $L^2$, and 
\begin{equation}\label{dul-fe-u}
\| u-u^k\|_{L^\infty_t L^2_x} \lesssim  \epsilon \step^{-sk} c_k.   
\end{equation}

On the other hand by \eqref{ul-fe} we have 
\begin{equation}\label{dul-feh}
\| u^{k+1}-u^k\|_{L^\infty_t H^{s+m}_x} \lesssim  \epsilon \step^{sm} c_k.
\end{equation}
Combining the last two relations, it follows that the summands in the series are almost orthogonal in $H^s$, and we have
\begin{equation}\label{ddul-fe-diff}
\| u^{k}-u^j\|_{L^\infty_t H^{s}_x} \lesssim  \epsilon c_{[j,k]}, \qquad j < k.
\end{equation}
This shows convergence in $L^\infty_t H^s_x$, and 
\begin{equation}\label{ddul-fe-ldiff}
\| u-u^j\|_{L^\infty_t H^{s}_x} \lesssim  \epsilon c_{\geq j}. 
\end{equation}
Now it is easily verified that the limit $u$ solves the equation \eqref{qnls}, and also that it satisfies all the bounds in Theorem~\ref{t:local-fe}.

\item Continuous dependence. This again follows 
the argument in \cite{IT-primer}.
\end{enumerate}

\section{Interaction Morawetz for the nonlinear flow and the global/long-time results}

The aim of this section is to complete the proof of the long-time and 
global well-posedness results in Theorem~\ref{t:long} and Theorem~\ref{t:global}. In view of the local well-posedness result in Theorem~\ref{t:local} and its associated frequency envelope  bounds in Theorem~\ref{t:local-fe}, it suffices to establish the frequency envelope bounds in Theorem~\ref{t:long-fe}, respectively Theorem~\ref{t:global-fe}.
Further, it suffices to do this under appropriate bootstrap assumptions,
see Proposition~\ref{p:boot}.

For the proof of the local well-posedness result in Theorem~\ref{t:local}
it was enough to interpret the (QNLS) evolution as a paradifferential equation with a perturbative source term. For the long-time results, on the other hand, we also single out the doubly resonant interactions,
which are handled via density/flux corrections as described in Section~\ref{s:df}.

 We consider a smooth symbol $a_\lambda$ as in \eqref{choose-a-loc}.
Correspondingly, we have the modified localized mass and momentum densities
\[
\ms_\lambda := M_\lambda(u,\bar u) + B^4_{\lambda,m}(u,\bar u, u,\bar u),
\]
\[
\ps_{\lambda} := P_{\lambda}(u,\bar u) + B^4_{\lambda,p}(u,\bar u, u,\bar u),
\]
which satisfy the conservation laws (see  \eqref{dens-flux-m}, respectively  \eqref{dens-flux-p})
\begin{equation}\label{dens-flux-m-re}
\partial_t \ms_\lambda(u) = \partial_x\left( g_{[<\lambda]} P_{\lambda}(u)
+ R^4_{\lambda,m}(u)\right) +  R^6_{\lambda,m}(u) +F^4_{\lambda,m}(u)  +F^{4,bal}_{\lambda,m}(u) ,
\end{equation}
respectively
\begin{equation}\label{dens-flux-p-re}
\partial_t \ps_{\lambda}(u) = \partial_x\left(g_{[<\lambda]} E_{\lambda}(u)
+ R^4_{\lambda,p}(u)\right) +  R^6_{\lambda,p}(u) +F^4_{\lambda,p}(u) +F^{4,bal}_{\lambda,p}(u) .
\end{equation}
Our task is now to use these conservation laws in order to prove 
the linear and bilinear frequency envelope bounds in Theorem~\ref{t:long-fe}, respectively Theorem~\ref{t:global-fe}.

\subsection{ Source term bounds}
Here we prove fixed time estimates for the mass and momentum corrections, as well as space-time bounds for
the source terms $R^6_{\lambda,m}(u)$, $F^4_{\lambda,m}(u)$ $F^{4,bal}_{\lambda,m}(u)$, and their momentum counterparts in the above density-flux relations:

\begin{proposition}\label{p:dens-flux-an}
Assume that our bootstrap assumptions in \eqref{uk-ee-boot}-\eqref{uk-se-boot}
hold in a time interval $[0,T]$ with $s > 1$. 
Then the multilinear forms in the
above mass density-flux relation \eqref{dens-flux-m}
have the following properties:
\begin{enumerate}[label=(\roman*)]
\item Fixed time bounds:
\begin{equation}\label{b4-est}
\| B^4_{\lambda,m}(u,\bar u, u,\bar u) \|_{L^1_x}
\lesssim C^4 \epsilon^4 c_\lambda^4 \lambda^{1-4s}.
\end{equation}

\item Space-time bounds:
\begin{equation}\label{pert-flux-source}
\|  R^6_{\lambda,m}(u,\bu,u,\bu,u,\bu)\|_{L^1_{t,x}} + \|F^4_{\lambda,m}(u,\bu,u,\bu) \|_{L^1_{t,x}} 
+ \|F^{4,bal}_{\lambda,m}(u,\bu,u,\bu) \|_{L^1_{t,x}}
\lesssim \epsilon^4 C^4 c_\lambda^4
\lambda^{-2s}.
\end{equation}
\end{enumerate}
The same bounds hold for the momentum counterparts, with an additional $\lambda$ factor on the right.
\end{proposition}

\begin{proof}[Proof of Proposition~\ref{p:dens-flux-an}]
We successively consider the four bounds above:
\medskip

\textbf{ 1. The bound \eqref{b4-est} for $B^4_{\lambda,m}$.}
By \eqref{div-m},  $B^4_{\lambda,m}$ has a smooth and bounded symbol which is localized at frequency $\lambda$. Then we estimate at fixed time
\[
\| B^4_{\lambda,m}(u,\bar u, u,\bar u) \|_{L^1_x}
\lesssim \|u_\lambda\|_{L^2_x}^2 \|u_\lambda\|_{L^\infty_x}^2,
\]
where we can use the uniform energy bound \eqref{uk-ee-boot}, combined with Bernstein's inequality for the $L^\infty_x$ norm.

\medskip

\textbf{ 2. The bound \eqref{pert-flux-source} for $R^6_{\lambda,m}$.}
We recall that $R^6_{\lambda,m}$
contains the contribution of terms of order three and higher in the equation \eqref{qnls} in the time derivative 
of $ B^4_{\lambda,m}(u,\bar u, u,\bar u)$. Here it is more convenient to 
use the form \eqref{full-lpara} of the 
equation, but with a slightly larger frequency window for $C_\lambda$, in order to guarantee that we have dyadic frequency separation between the input frequencies in $N_\lambda^{tr}$ and the frequencies in the support of $B^4_{\lambda,m}$. Then we need to bound in $L^1_{t,x}$
the following three types of expressions:
\[
B^4_{\lambda,m}(N_\lambda^{tr}(u),\bar u, u,\bar u), \qquad B^4_{\lambda,m}(C_\lambda(u,\bar u,u),\bar u, u,\bar u), \qquad B^4_{\lambda,m}(\partial_x (g_{[<\lambda]}-1)\partial_x u_\lambda,\bar u, u,\bar u).
\]
Here the last three inputs of $B^4_{\lambda,m}$ are all
localized at frequency $\lambda$. So for  the first expression above we may pair 
one of them with $N_\lambda^{tr}$ and use the bound 
\eqref{good-nl} to estimate their product in $L^1_{t,x}$,
while using Bernstein's inequality to place the other two
in $L^\infty_x$. This yields
\[
\|B^4_{\lambda,m}(N_\lambda^{tr}(u),\bar u, u,\bar u)
\|_{L^1_{t,x}} \lesssim C^6 \epsilon^6 c_\lambda^4  \lambda^{-2s} \lambda^{1-2s}, 
\]
which is more than sufficient.

For the second expression above we have six balanced inputs at frequency $\lambda$, so we use \eqref{uk-se-boot} six times to obtain
\[
\| B^4_{\lambda,m}(C_\lambda(u,\bar u,u),\bar u, u,\bar u)
\|_{L^1_{t,x}} \lesssim C^4 \epsilon^4 c_\lambda^4 \lambda^{-4s+1},
\]
same as before.

Finally, for the third expression we can pull two 
$u_{<\lambda}$ factors from $g$, and then use two bilinear $L^2_{t,x}$ bounds \eqref{uab-bi-unbal-boot},
and two $L^\infty_x$ Bernstein inequalities at frequency $\lambda$,
obtaining again the same bound as in the first two cases. 

\medskip

\textbf{ 3. The bound \eqref{pert-flux-source} for $F^4_{\lambda,m}$.} Recalling that 
 $F^4_{\lambda,m} (u,\bu,u,\bu) = 2 \Im ( (N_\lambda^{tr}(u,\bu, u)  \bu_\lambda)$, this bound is an immediate consequence of \eqref{good-nl}.

\medskip

\textbf{ 4. The bound \eqref{pert-flux-source} for $F^{4,bal}_{\lambda,m}$.}
Here we recall that $F^{4,bal}_{\lambda,m}$ is 
a quartic form whose symbol admits the factorization
\[
f^{4,bal}_{\lambda,m}(\xi_1,\xi_2,\xi_3,\xi_4)
=  q^{4,bal}_{\lambda,m}[ (\xi_1-\xi_2)(\xi_3-\xi_4)+ (\xi_1-\xi_4)(\xi_2-\xi_3)]
\]
with $q^{4,bal}_{\lambda,m}$ localized at frequency $\lambda$, bounded and smooth on the dyadic scale.
This implies that $F^{4,bal}_{\lambda,m}$ can be represented as a linear combination of multilinear forms
\[
F^4_{\lambda,m} (u,\bu,u,\bu)
= L( \partial_x L(u,\bar u),\partial_x L(u,\bar u)),
\]
with multilinear forms $L$ localized on the $\lambda^{-1}$ spatial scale\footnote{This is important in order to handle the $x_0$ dependence in the balanced bilinear $L^2$ bounds.}. Then the desired $L^1_{t,x}$ 
bound is obtained by applying twice the balanced $L^2_{t,x}$ bound \eqref{uab-bi-bal-boot}.

\end{proof}

\subsection{The uniform energy bounds}
Here we use the bounds in Proposition~\ref{p:dens-flux-an} in order to prove the dyadic energy estimates \eqref{uk-ee} in our 
main result in Theorem~\ref{t:long-fe} and Theorem~\ref{t:global-fe}, under the appropriate bootstrap assumptions as stated in Proposition~\ref{p:boot}. 

The argument differs slightly for the two theorems.
We begin with the global result in Theorems~\ref{t:global-fe}.
On a time interval $[0,T]$ we integrate the density-flux relation for mass, to obtain
\[
\sup_{t \in [0,T]} \int \ms(u_\lambda)(t,x) \, dx 
\lesssim \int \ms(u_\lambda)(0,x) \, dx
+ \|  R^6_{\lambda,m}(u)\|_{L^1_{t,x}} + \|F^4_{\lambda,m}(u) \|_{L^1_{t,x}} 
+ \|F^{4,bal}_{\lambda,m}(u) \|_{L^1_{t,x}}.
\]
Now we apply the estimate \eqref{b4-est} to switch from $\ms$ to $M$, and \eqref{pert-flux-source} for the remaining terms. We obtain
\[
\|u_\lambda\|_{L^\infty_t L^2_x}^2 \lesssim 
\epsilon^2 c_\lambda^2 \lambda^{-2s} +
C^4 \epsilon^4 c_\lambda^4 \lambda^{1-4s}
+ \epsilon^4 C^4 c_\lambda^4 \lambda^{-2s}.
\]
The constant $C$ is absorbed by the extra $\epsilon$ factors if $\epsilon$ is small enough, so we arrive at
\begin{equation}\label{good-ee}
\|u_\lambda\|_{L^\infty_t([0,T]; L^2_x)}^2 \lesssim 
\epsilon^2 c_\lambda^2 \lambda^{-2s}.
\end{equation}
Since $T$ was arbitrarily large, the global bound \eqref{uk-ee} follows.
\bigskip

We continue with the long-time result in Theorem~\ref{t:long-fe}.  The difference is 
that we now only assume the bootstrap bounds
\eqref{uk-ee-boot}-\eqref{uk-se-boot} on time intervals of length $\ll \epsilon^{-6}$. 
Hence the bounds \eqref{pert-flux-source} also hold only in similar time intervals. Then the argument above only applies directly for $T\ll \epsilon^{-6}$.
For a larger $T$, we split $[0,T]$ in subintervals of size $\epsilon^{-6}$ and add the errors, to obtain
\[
\|u_\lambda\|_{L^\infty_t L^2_x}^2 \lesssim 
\epsilon^2 c_\lambda^2 \lambda^{-2s} +
C^4 \epsilon^4 c_\lambda^4 \lambda^{1-4s}
+ (T \epsilon^{-6}) \epsilon^4 C^4 c_\lambda^4 \lambda^{-2s}.
\]
Then the estimate \eqref{good-ee} will follow provided that $T \ll \epsilon^{-8}$, as needed.

\subsection{ The interaction Morawetz identities}

Overall, we will seek to pair a frequency $\lambda$ 
portion of one solution $u$ with a frequency $\mu$ portion 
of another solution $v$, which will eventually be taken to be 
$v = u^{x_0}$, i.e. a translate of $u$. We will write here 
all identities in the general case. Then, in order to 
prove the global result in Theorem~\ref{t:global-fe} for the defocusing case, we will specialize to three cases:
\begin{enumerate}
    \item The diagonal case $\lambda = \mu$, $u=v$ (and thus $x_0=0$).
    \item The balanced shifted case $\lambda \approx \mu$, with $v= u^{x_0}$ and
    with  $x_0$ arbitrary.
\item The unbalanced case $\lambda \not \approx \mu$, 
    with  $v = u^{x_0}$ and $x_0$ arbitrary.
\end{enumerate}
On the other hand, in order to prove the long time result 
in Theorem~\ref{t:global-fe}, it suffices to consider 
only the last two cases.

We define the interaction Morawetz functional for the functions 
$(u,v)$ and the associated pair of frequencies $(\lambda,\mu)$ as
\begin{equation}\label{IM}
\bI_{\lambda\mu}(u,v) :=   \iint_{x > y} \ms_\lambda(u)(x) \ps_{\lambda}(v) (y) -  
\ps_{\lambda}(u)(x) \ms_{\lambda}(v) (y) \, dx dy.
\end{equation}
We use the modified density-flux identities 
\eqref{dens-flux-m-re}, \eqref{dens-flux-p-re}
for the mass and for the momentum in order to compute the time derivative of $\bI_{\lambda\mu}$ as
\begin{equation}\label{interaction-xilm}
\frac{d}{dt} \bI_{\lambda\mu} =  \bJ^4_{\lambda\mu} + \bJ^6_{\lambda\mu} + \bJ^8_{\lambda\mu} + \bK_{\lambda\mu}. 
\end{equation}

Here the quartic contribution $\bJ^4_{\lambda\mu}$ is 

\begin{equation}\label{J4-deflm}
\begin{aligned}
\bJ^4_{\lambda \mu}(u,v) = & \  \int g_{[<\mu]} \left(M_\lambda(u) E_\mu(v) 
- P_\lambda(u) P_\mu(v)\right)\, 
\\
& \ + g_{[<\lambda]}\left( E_\lambda(u) M_\mu(v) 
-  P_\lambda(u) P_\mu(v)\right)\, dx.
\end{aligned}
\end{equation}
The sixth order term $\bJ^6_{\lambda}$ has the form
\begin{equation}\label{J6-deflm}
\begin{aligned}
\bJ^6_{\lambda\mu}(u,v) =  \int & \ M_\lambda(u) R^4_{\mu,p}(v)+ g_{[<\mu]}B^4_{\lambda,m}(u) E_\mu(v)
- g_{[<\lambda]}P_\lambda(u)B^4_{\mu,p}(v)- R^4_{\lambda,m}(u)P_\mu(v) 
\\ & \ 
+ M_\mu(v) R^4_{\lambda,p}(u)+g_{[<\lambda]} B^4_{\mu,m}(v) E_\lambda(u)
- g_{[<\mu]}P_\mu(v)B^4_{\lambda,p}(u)- R^4_{\mu,m}(v)P_\lambda(u)\,  dx .
\end{aligned}
\end{equation}
Next we have 
\begin{equation}\label{J8-deflm}
\bJ^8_{\lambda\mu}(u,v) =   \int 
B^4_{\lambda,m}(u) R^4_{\mu,p}(v) - R^4_{\lambda,m}(u) B^4_{\lambda,p}(v)
+ B^4_{\mu,m}(v) R^4_{\lambda,p}(u) - R^4_{\mu,m}(v) B^4_{\lambda,p}(u)
 \, dx .
\end{equation}

Finally the last term $\bK_{\lambda}$ has the form
\begin{equation}\label{K8-def-ablm}
\begin{aligned}
\bK^8_{\lambda\mu}(u,v) = \iint_{x > y} & \ \ms_\lambda(u)(x)  R^{\text{full}}_{\mu,p}(v)(y)  + \ps_\mu(v)(y) R^{\text{full}}_{\lambda,m}(u)(x)
\\ & \hspace{-.6in} - 
\ms_\mu(v)(y)  R^{\text{full}}_{\lambda,p}(u)(x)  - \ps_\lambda(u)(x) R^{\text{full}}_{\mu,m}(v)(y)\,
dx dy,
\end{aligned}
\end{equation}
where we have used the condensed notation 
$R^{\text{full}} := R^6 + F^4+ F^{4,bal}$ for both the mass and momentum.

The way the interaction Morawetz identity will be used 
will be to estimate positive terms, of which the primary one is the  leading part of $\bJ^4$ and, in one case,
the leading part of $\bJ^6$, in terms of the other ones,
namely $\bI$, $\bJ^8$ and $\bK^8$, for which we use our bootstrap hypothesis. Before we specialize to the three cases discussed earlier, in this section we estimate these last three terms.

\medskip

\textbf{I. The bound for $\bI_{\lambda\mu }$.} 
This is the leading contribution. Since it is a fixed time bound,
to estimate it we are able to use only the energy estimate \eqref{uk-ee}, which has already been proved. Thus, the constant $C$ does not appear at all. Precisely, we will show that 
\begin{equation}\label{Il-est}
    | \bI_{\lambda\mu}| \lesssim \epsilon^4  c_\lambda^4 \lambda^{1-2s}
    \mu^{-2s}, \qquad \mu \lesssim \lambda.
\end{equation}
To see this, we first observe that by \eqref{uk-ee} we have
\[
\| M_\lambda(u) \|_{L^1_x} \lesssim \| u_\lambda\|_{L^2_x}^2
\lesssim \epsilon^2  c_\lambda^2 \lambda^{-2s}.
\] 
For $s \geq \frac12$ and small enough $\epsilon$ we can combine this  with \eqref{b4-est} to obtain a similar bound for the modified mass density,
\begin{equation}\label{m-l1}
\| \ms_\lambda(u) \|_{L^1_x} \lesssim \epsilon^2  c_\lambda^2 \lambda^{-2s}.
\end{equation}
A similar computation for the momentum yields
\begin{equation}\label{p-l1}
\| \ps_\lambda(u) \|_{L^1_x} \lesssim \epsilon^2  c_\lambda^2 \lambda^{1-2s}.
\end{equation}
Combining the last two bounds we arrive at \eqref{Il-est}.

\medskip

\textbf{II. The bound for $\bJ^8_{\lambda\mu}$.}  We will show that
\begin{equation}\label{j8-estlm}
  \left| \int_0^T \bJ^8_{\lambda\mu}\, dt\right| \lesssim  
  C^6 \epsilon^6 c_\lambda^4 c_\mu^4 \lambda^{\frac43(1-2s)}
    \mu^{-2s+\frac23(1-2s)},
  \qquad \mu \lesssim \lambda.
\end{equation}
This is again a relatively straightforward bound, as 
$\bJ^8_{\lambda\mu}$ is an $8$-linear form with four factors localized at frequency $\lambda$ and four at frequency $\mu$,
and a symbol of size at most $\lambda^2$.
Then it suffices to use six times the $L^6_{t,x}$ bound  \eqref{uk-se-boot} and two $L^\infty_{t,x}$ bounds obtained from \eqref{uk-ee} 
via Bernstein's inequality. Since we assume $\mu \lesssim \lambda$,
it is more efficient to use the $L^6_{t,x}$ bound for the four frequency $\lambda$ factors, and two frequency $\mu$ factors.
This yields \eqref{j8-estlm}.

We remark that the bound in \eqref{j8-estlm} is better than 
the one in \eqref{Il-est} provided that $\epsilon$ is small enough
and $s \geq \frac12$.

We also remark that in the unbalanced case $\mu \not\approx \lambda$ one could alternatively estimate this term  using twice the bilinear $L^2_{t,x}$ bound in \eqref{uab-bi-unbal-boot}, and four $L^\infty_{t,x}$ bounds obtained from \eqref{uk-ee} 
via Bernstein's inequality.

\medskip
\textbf{III. The bound for $\bK_{\lambda\mu}$.}  This has the form
\begin{equation}\label{kl-est-lm}
\left| \int_0^T \bK_{\lambda\mu} \, dt \right| \lesssim 
\epsilon^6 C^6 c_\lambda^2 c_\mu^4  \lambda^{1-2s} \mu^{-2s },
\end{equation}
and is obtained directly by combining the $L^1_x$ bounds in 
\eqref{m-l1}, \eqref{p-l1} and \eqref{pert-flux-source}. 

We again remark that the bound in \eqref{kl-est-lm} is better than the one in \eqref{Il-est}. Here the allowed range for $s$ 
is determined by \eqref{pert-flux-source}.

\subsection{ The diagonal case} \
Here we consider the (self)-interaction Morawetz identities for 
a single dyadic frequency $\lambda$, i.e. when $u=v$. This case
needs to be considered separately only in the defocusing case,
i.e. for Theorem~\ref{t:global-fe}; the objective here is to capture the $L^6_{t,x}$ Strichartz norm of $u_\lambda$, namely 
the bound \eqref{uk-se}. As a secondary byproduct we also obtain  bilinear $L^2_{t,x}$ bound \eqref{uk-bi} for $u_\lambda$, but this can also be seen as a special case of the shifted bound \eqref{uab-bi-bal} which we prove in the next subsection.

We will use the interaction Morawetz identity \eqref{IM} with $\mu = \lambda$. As the contributions of $\bI_{\lambda \mu}$, $\bJ^8_{\lambda \mu}$ and $\bK_{\lambda\mu}$
have already been estimated, it remains to consider the contributions of $\bJ^4_{\lambda\mu}$ and $\bJ^6_{\lambda\mu}$.

 \bigskip

\textbf{I. The bound for $\bJ^4_\lambda$.} At leading order this is a positive definite expression, whose leading term can be explicitly computed. We will show that 
\begin{equation}\label{j4-est}
\int_0^T \bJ^4_\lambda \, dt = 4\| g_{[<\lambda]}^\frac12
\partial_x |u_\lambda|^2\|_{L^2}^2 + O(\epsilon^6 C^6 c_\lambda^4
\lambda^{1-4s}),
\end{equation}
where the error term is negligible.
To see this, we write 
\[
J^4_{\lambda} (u,u) = g_{[<\lambda]} J^4_{main}(u_\lambda,u_\lambda),
\]
where we recall that the expression $J^4_{main}$ is given by \eqref{J4-main}. In the $u = v$ case this yields
\[
\int_0^T \bJ^4_\lambda \, dt = 4\|  g_{[<\lambda]}^\frac12
\partial_x (u_\lambda \bu_\lambda)\|_{L^2}^2 + Err,
\]
where, integrating by parts the second term in $J^4_{main}$, we have 
\[
Err := \int_0^T \int_\R \partial_x^2  g_{[<\lambda]} |u_\lambda|^2
|u_\lambda|^2 \, dx dt.
\]
Here the factor $\partial_x^2  g_{[<\lambda]}$
is at least quadratic in $u_{\ll \lambda}$.
The key feature of the above expression is that it is 
unbalanced, so we can estimate two bilinear expressions,
say $u_{\ll \lambda} \bu_\lambda$ and its conjugate in $L^2_{t,x}$ 
using \eqref{uab-bi-bal-boot}, 
and the two remaining frequency $\lambda$ factors in $L^\infty_x$ using Bernstein's inequality,
in order to get
\[
|Err|\lesssim \epsilon^6 C^6 \sum_{\mu_1,\mu_2 < \lambda}
 (\mu_1+\mu_2)^2 \lambda^{-1} \lambda \mu_1^{-s} \mu_2^{-s} c_\lambda^4 \lambda^{-4s} \lesssim \epsilon^6 C^6 c_\lambda^4 \lambda^{1-4s}  
\]
as needed. For the last step we have used $ s > 1$ in order to handle the $\mu_1$ and $\mu_2$ summation.

\bigskip

\textbf{II. The bound for $\bJ^6_\lambda$.}
Here we first observe that we can harmlessly replace $g_{[<\lambda]}$ with $1$, as $g_{[<\lambda]}-1$ is at least quadratic in $u_{\ll \lambda}$ and has uniform size 
\[
\|g_{[<\lambda]}-1\|_{L^\infty} \lesssim 
\epsilon^2,
\]
Hence its contribution can be estimated perturbatively using six $L^6$ bootstrap bounds \eqref{uk-se-boot} at frequency $\lambda$. 

Since in this case we are working with $u = v$, it follows that 
$\bJ^6_{\lambda}$ is a six-linear form in $u$ which is localized at frequency $\lambda$. Then it becomes important 
to compute the symbol of $\bJ^6_\lambda$ on the diagonal
$\xi_1 = \xi_2=\xi_3=\xi_4=\xi_5=\xi_6:=\xi$. This will be essential later on in order to obtain bounds for the $L^6_{t,x}$ Strichartz norm.

\begin{lemma}
The diagonal trace of the symbol $j^6_{\lambda}$
is closely related to the symbol of the cubic term in the equation \eqref{qnls},
\begin{equation}\label{good-J6}
j^6_{\lambda}(\xi) = c(\xi,\xi,\xi) \phi_\lambda^4 (\xi).
\end{equation}
\end{lemma}
 This follows the same computations as in Lemma~5.1 of \cite{IT-global}; we leave the details to the reader.

The above lemma allows us peel-off a nonnegative diagonal part of $\bJ^6_\lambda$. For this we define the symbol
\[
b_\lambda(\xi) = c(\xi,\xi,\xi)^\frac16 \phi_\lambda^\frac23.
\]
Using this notation, we will show that
\begin{equation}\label{j6-est}
\bJ^6_\lambda(u,\bu,u,\bu,u,\bu) = \| B_\lambda u\|_{L^6_{t,x}}^6 +
O( \epsilon^5 C^5 c_\lambda^5 \lambda^{\frac32-5s}).
\end{equation}
To see this, we decompose 
\[
j^6_{\lambda}(\xi_1,\xi_2,\xi_3,\xi_4,\xi_5,\xi_6) =
b_\lambda(\xi_1) b_\lambda(\xi_2) b_\lambda(\xi_3)
b_\lambda(\xi_4) b_\lambda(\xi_5) b_\lambda(\xi_6)
+\tilde j^6_\lambda(\xi_1,\xi_2,\xi_3,\xi_4,\xi_5,\xi_6),
\]
where $\tilde j^6$ vanishes if all $\xi$'s are equal.
As in \cite{IT-global}, this implies that 
on the diagonal $\Delta^6 \xi = 0$ we have a representation
\[
\tilde j^6_\lambda =  q(\xi)(\xi_{odd} - \xi_{even}),
\]
where $q$ is a symbol of order $1$. Separating variables at frequency $\lambda$, this yields a decomposition 
of $\tilde J^6_\lambda(u)$ as a linear combination\footnote{Infinite but rapidly convergent.}  of expressions  of the form
\[
\lambda L( u_\lambda,\bu_\lambda,u_\lambda, \bu_\lambda,
\partial_x L(u_\lambda,\bu_\lambda)). 
\]
with kernels rapidly decreasing\footnote{This is because the corresponding symbols are smooth on the $\lambda$ scale.} beyond the $\lambda^{-1}$ scale.
To obtain the desired error bound in \eqref{j6-est}, here we use once the bilinear 
$L^2_{t,x}$ bound \eqref{uab-bi-unbal-boot}, three times the $L^6_{t,x}$ bound
\eqref{uk-se-boot} and once the $L^\infty_{t,x}$ bound obtained from 
\eqref{uk-ee-boot} by Bernstein's inequality. Importantly, we do not use six times the $L^6_{t,x}$ bound.

\bigskip

\textbf{III. Proof of the estimates \eqref{uk-se} and \eqref{uk-bi} in the defocusing case.}
Combining the bounds \eqref{Il-est}, \eqref{j4-est},\eqref{j6-est},
\eqref{j8-estlm} and \eqref{kl-est-lm} in the interaction 
Morawetz identity \eqref{IM}, we obtain
the bound
\begin{equation*}
4\| g_{[<\lambda]}
\partial_x \vert u_{\lambda}\vert^2\|_{L^2_{t,x}}^2+  \| B_\lambda u\|_{L^6_{t,x}}^6
\lesssim (\epsilon^4 + C^6 \epsilon^6)  c_\lambda^4 \lambda^{1-4s}.
\end{equation*}
Assuming that $s \geq \frac12$ and that $\epsilon$ is small enough in order to absorb all the $C$ factors, we arrive at 
\begin{equation}\label{conclude:u=v}
\| \partial_x |u_\lambda|^2\|_{L^2_{t,x}}^2+  \| B_\lambda u\|_{L^6_{t,x}}^6
\lesssim \epsilon^4  c_\lambda^4 \lambda^{1-4s}.
\end{equation}
Since $B_\lambda$ has order $\dfrac13$, the bounds 
\eqref{uk-se} and \eqref{uk-bi} follow.  

We remark that this is the step where we crucially use the defocusing character of the problem. This is what insures that the 
two terms on the left in the bounds above have the same sign.

\bigskip

\textbf{IV. Proof of the estimate \eqref{uk-se} in the general case.} Here the argument above does not apply, but instead we 
have the constraint $T \ll \epsilon^{-6}$ on the size of the time interval. Then we can get the $L^6_{t,x}$ bound for $u_\lambda$ by interpolating between the energy bound, which gives
\[
\| |u_\lambda|^2 \|_{L^\infty_t L^1_x} \lesssim C^2 \epsilon^2 c_\lambda^2 \lambda^{-2s},
\]
and the balanced diagonal bilinear $L^2_{t,x}$ bound, which yields
\[
\| \partial_x |u_\lambda|^2 \|_{L^2_t L^2_x} \lesssim C^2 \epsilon^2 c_\lambda^2 \lambda^{\frac12 -2s}.
\]
We obtain
\[
\| |u_\lambda|^2\|^2_{L^\frac{9}2_t L^3_x}  \lesssim C^2 \epsilon^2 c_\lambda^2 \lambda^{\frac{2}{9} -2s},
\]
where by H\"older's inequality we arrive at
\[
\| u_\lambda \|_{L^6_{t,x}} \lesssim  T^{\frac{1}{18}} C \epsilon c_\lambda \lambda^{\frac{1}{9} - s}.
\]
But this implies \eqref{uk-se} 
for $T \ll \epsilon^{-6}$ and $s > 1$.

\subsection{ The balanced shifted case}
Here we consider the interaction Morawetz identities for 
two nearby dyadic frequencies $\mu \approx \lambda$, and prove 
the bilinear $L^2_{t,x}$ bound \eqref{uab-bi-bal}. For simplicity we just set $\mu = \lambda$; this makes no difference in the proof.
We recall that at this point we have already proved the 
energy bound \eqref{uk-ee} and the $L^6_{t,x}$ Strichartz bound
\eqref{uk-se}.

 We use the same interaction Morawetz functional \eqref{IM} and the associated relation  \eqref{interaction-xilm} 
as before, but with $v = u^{x_0}$.  
The bounds for $\bI_\lambda$,  $\bJ^8_\lambda$ and $\bK_\lambda$ 
are the same as before, namely \eqref{Il-est}, \eqref{j8-estlm}
and \eqref{kl-est-lm}. It remains to estimate $\bJ^4_\lambda$ 
and $\bJ^6_\lambda$. This time we will also think of $\bJ^6_\lambda$ as perturbative, so the burden is in the $\bJ^4_\lambda$ 
analysis.

\bigskip

\textbf{I. The bound for $\bJ^4_\lambda$.}  We will show that 
\begin{equation}\label{j4-est-shift}
\int_0^T \bJ^4_\lambda \, dt \approx \|\partial_x (u \bar v)\|_{L^2_{t,x}}^2 + O(\epsilon^6 C^6 c_\lambda^4
\lambda^{1-4s} (1+\lambda |x_0| )).
\end{equation}
In this case $\bJ^4_\lambda$ has a slightly different expression,
\[
\begin{aligned}
\bJ^4_{\lambda}(u,v) = \int & g_{[<\lambda]}^{x_0}\left(M_\lambda(u) E_\lambda(v) - 2 P_\lambda(u) P_\lambda(v)\right)
\\ & \ + g_{[<\lambda]}
\left(M_\lambda(v) E_\lambda(u)
- 2 P_\lambda(u) P_\lambda(v)\right)\, dx.
\end{aligned}
\]
We split the integrand into 
\[
\begin{aligned}
J^4_{\lambda}(u,v) = & \  g_{[<\lambda]}^{x_0} J^4_{main}(u_\lambda,v_\lambda)
+ ( g_{[<\lambda]}-  g_{[<\lambda]}^{x_0})
\left(M_\lambda(v) E_\lambda(u)
- 2 P_\lambda(u) P_\lambda(v)\right)
\\
: = & \  
J^{4,a}_{\lambda}(u,v)+
J^{4,b}_{\lambda}(u,v).
\end{aligned}
\]
The first term is treated exactly as in the previous case
using the identity \eqref{j4-est}.
It is the second term where we have a problem, which arises due to the fact that the two paracoefficients are different.

Here we proceed as in step II(a) of the proof of Lemma~\ref{l:J4}, using the Fundamental Theorem of Calculus with respect to $x_0$ to write
\[
g_{[<\lambda]}-  g_{[<\lambda]}^{x_0} = 
\int_{0}^{x_0} \partial_x g_{[<\lambda]}^y \,dy.
\]
Only after this we use two unbalanced bilinear estimates, in particular  
\[
\| \partial_x u_{<\lambda} u_\lambda \|_{L^2_{t,x}} \lesssim 
C^2 \epsilon^2 \lambda^{-s - \frac12} c_\lambda.
\]
The $y$ integration yields an $|x_0|$ factor. This gives
\[
\begin{aligned}
\left| \int_0^T \bJ^{4,b}_{\lambda} \, dt \right|
\lesssim & \ \epsilon^6 C^4 \lambda^2  \lambda^{-4s} 
c_{\lambda}^4 \lambda^{-1} \lambda |x_0|  
\\
 \lesssim & \    \epsilon^6 C^4   \lambda^{1-4s} 
c_{\lambda}^4 (1+\lambda |x_0|),
\end{aligned}
\]
as needed.

\medskip

\textbf{II. The bound for $J^6_\lambda$: The case $v = u^{x_0}$.}
Since $v_\lambda \neq u_\lambda$ here, we can no longer 
use the fine structure of $j^6_\lambda$ near the diagonal. 
Instead we simply note that $j^6_\lambda$ is a smooth symbol 
of size $\lambda^2$ and we bound it using the $L^6_{t,x}$ bound.
The advantage we have at this point, due to Step VI above,
is that we can use \eqref{uk-se} instead of \eqref{uk-se-boot}.
Then we obtain
\begin{equation}\label{j6-est+}
\left| \int_{0}^T \bJ^6_\lambda \, dt\right| \lesssim
\epsilon^4 c_\lambda^4 \lambda^{1-4s}.
\end{equation}

\bigskip

\textbf{III Proof of \eqref{uab-bi-bal} in the case $\lambda \approx \mu$.}
This is achieved  by combining the bounds \eqref{Il-est}, \eqref{j4-est-shift},\eqref{j6-est+},
\eqref{j8-estlm} and \eqref{kl-est-lm} in the interaction 
Morawetz identity \eqref{interaction-xilm}.

\subsection{The unbalanced case}

Here we consider the interaction Morawetz identities for 
two separated dyadic frequencies $\lambda,\mu$, and prove bilinear $L^2_{t,x}$ bounds. Without any loss of generality we will assume that $\mu \ll \lambda$. Our goal is to prove the bilinear $L^2_{t,x}$ bound \eqref{uab-bi-unbal}.

The analysis in this case largely repeats the analysis in the balanced case, with two notable differences: (i) the $L^6_{t,x}$ Strichartz bound
no longer plays a role, and (ii) the balance of the frequencies $\lambda$ and $\mu$ is now important.

We will use the interaction Morawetz relation in order to prove the bilinear $L^2_{t,x}$ bound \eqref{uab-bi-unbal}. 
We already have the bounds for the expression $\bI_{\lambda\mu}$  and the perturbative terms  $\bJ^8_{\lambda\mu}$ and $\bK_{\lambda\mu}$. The desired bilinear
bound will come from the leading term $\bJ^4_{\lambda\mu}$. The remaining term $\bJ^6_{\lambda\mu}$ will be estimated perturbatively using our bootstrap assumptions.

\bigskip

\textbf{I. The bound for $\bJ^6_{\lambda\mu}$.}
This is
\begin{equation}\label{j6-est+lm}
\left| \int_{0}^T \bJ^6_{\lambda\mu} \, dt\right| \lesssim
\epsilon^6 C^4 c_\lambda^2 c_\mu^4 \lambda^{1-2s} \mu^{1-4s}.
\end{equation}
To see this, we observe that here we have six input frequencies
of which two are $\lambda$ and four are $\mu$, or viceversa.

To avoid a circular argument we separate the case when $\lambda$ and $\mu$ are comparable, when we use \eqref{j6-est+} as in the balanced case.

Otherwise if $\mu \ll \lambda$ then in all cases we can use two bilinear estimates  \eqref{uab-bi-unbal-boot} coupling two $(\lambda,\mu)$ frequency pairs, and then 
estimate the remaining two factors in $L^\infty_{t,x}$ via Bernstein's inequality
from \eqref{uk-ee}.
In both cases, the symbol has size at most $\lambda^2$.
Hence we obtain
\[
\left| \int_{0}^T \bJ^6_{\lambda\mu} \, dt\right|
\lesssim \epsilon^6 C^4n\lambda^2 \lambda^{-1} ( c_\lambda^4 c_\mu^2 \lambda^{1-4s} \mu^{-2s} + c_\mu^4 c_\lambda^2 \mu^{1-4s} \lambda^{-2s}).
\]
Since $\mu < \lambda$ and $s > \frac12$, the second term is the worst and \eqref{j6-est+lm} follows.

\bigskip

\textbf{II. The bound for $\bJ^4_{\lambda \mu}$.}  We will show that 
\begin{equation}\label{j4-estlm}
\int_0^T \bJ^4_{\lambda\mu}\, dt = 4\| 
\partial_x (u \bar v)\|_{L^2_{t,x}}^2 + O(\epsilon^6 C^6 c_\lambda^2 c_\mu^2 \lambda^{1-2s}\mu^{-2s}).
\end{equation}
This is simpler than the computation in the balanced case, because we can treat the difference in the prefactors perturbatively. 
We split it into three parts,
\[
\begin{aligned}
\bJ^4_{\lambda\mu}(u,v) = & \ 
\int  \left(M_\lambda(u) E_\mu(v) 
+ M_\mu(v) E_\lambda(u)
- 2 P_\lambda(u) P_\mu(v)\right)\, dx.
\\ & \ 
+ \int (g_{[<\lambda]}-1)\left( E_\lambda(u) M_\mu(v) 
-  P_\lambda(u) P_\mu(v)\right)\, dx
\\ & \ 
+ \int (g_{[<\mu]}-1) \left( E_\lambda(u) M_\mu(v) 
-  P_\lambda(u) P_\mu(v)\right)\, dx
\\ := & \  \bJ^{4,a}_{\lambda\mu}(u,v) + \bJ^{4,b}_{\lambda\mu}(u,v)
+ \bJ^{4,c}_{\lambda\mu}(u,v).
\end{aligned}
\]

For $J^{4,a}_{\lambda\mu}$ we write as before
\[
J^{4,a}_{\lambda\mu} = M_\lambda(u) E_\mu(v) 
+ M_\mu(v) E_\lambda(u)
- 2 P_\lambda(u) P_\mu(v),
\]
which after integration yields the leading contribution in \eqref{j4-estlm}.

For both $\bJ^{4,b}_{\lambda\mu}(u,v)$ and $\bJ^{4,c}_{\lambda\mu}(u,v)$
we  estimate the  two bilinear expressions 
say $u_{\mu} \bv_\lambda$ and $\bu_{\mu} v_\lambda$ 
in $L^2_{t,x}$ and the prefactors in the uniform norm to obtain 
\[
\begin{aligned}
\left|  \int_0^T \int_\R J^{4,b}_{\lambda\mu}(u,v) \, dx dt \right|
\lesssim & \ \epsilon^6 C^4 \lambda^2 \lambda^{-2s} \mu^{-2s}  c_\lambda^{2} c_\mu^2
\lambda^{-1},   
\end{aligned}
\]
and similarly for $\bJ^{4,c}_{\lambda\mu}(u,v)$. This suffices.

\bigskip

\textbf{VI. The proof of  the bilinear $L^2$ bound \eqref{uab-bi-unbal}.}
This is achieved as in the balanced case by combining the bounds \eqref{Il-est}, \eqref{j4-estlm},\eqref{j6-est+lm},
\eqref{j8-estlm} and \eqref{kl-est-lm} in the interaction 
Morawetz identity \eqref{interaction-xilm}.

 \bibliographystyle{plain}

%\bibliography{1d-global}

\bibliographystyle{plain}

\end{document}